\documentclass[reqno, 10pt]{amsart}
\usepackage{upgreek}
\usepackage{mathtools}
\usepackage{amssymb}
\usepackage{mathrsfs}
\usepackage{amsthm}
\usepackage{enumerate}
\numberwithin{equation}{section}

\usepackage{lmodern}

\usepackage[scr=boondox]{mathalpha}
\usepackage{xcolor}

\usepackage[left=1in, right=1in, top=1.25in, bottom=1.25in]{geometry}

\usepackage{tikz}
\usetikzlibrary{shapes,arrows,arrows.meta}

\usepackage{multicol}
\usepackage{graphicx}
\usepackage[colorlinks,allcolors=blue!75!black]{hyperref}
\usepackage{cancel}
\usepackage{booktabs} 

\usepackage[nameinlink,capitalize]{cleveref}

\makeatletter
\newcommand{\leqnomode}{\tagsleft@true\let\veqno\@@leqno}
\makeatother

\allowdisplaybreaks

\newcommand{\bbL}{\mathbb{L}}
\newcommand{\bbM}{\mathbb{M}}
\newcommand{\bbN}{\mathbb{N}}
\newcommand{\bbP}{\mathbb{P}}

\newcommand{\bbU}{\mathbb{U}}

\newcommand{\bbX}{\mathbb{X}}

\newcommand{\bbUdiv}{\bbU_{\operatorname{div}}}

\def\bB{\mathbf{B}}
\def\bC{\mathbf{C}}

\def\bE{\mathbf{E}}
\def\bF{\mathbf{F}}
\def\bG{\mathbf{G}}
\def\bH{\mathbf{H}}
\def\bI{\mathbf{I}}
\def\bK{\mathbf{K}}
\def\bL{\mathbf{L}}
\def\bM{\mathbf{M}}
\def\bN{\mathbf{N}}

\def\bR{\mathbf{R}}

\def\bU{\mathbf{U}}
\def\bV{\mathbf{V}}
\def\bW{\mathbf{W}}

\def\ba{\mathbf{a}}
\def\bb{\mathbf{b}}
\def\bc{\mathbf{c}}
\def\be{\mathbf{e}}
\def\bg{\mathbf{g}}
\def\bh{\mathbf{h}}

\def\bm{\mathbf{m}}
\def\bn{\mathbf{n}}

\def\bu{\mathbf{u}}
\def\bv{\mathbf{v}}
\def\bw{\mathbf{w}}
\def\bx{\mathbf{x}}
\def\by{\mathbf{y}}
\def\bz{\mathbf{z}}
\renewcommand{\bf}{\mathbf{f}}

\def\bnu{\boldsymbol{\nu}}
\def\bxi{\boldsymbol{\xi}}

\newcommand{\rmb}{\mathrm{b}}
\newcommand{\LL}{\mathrm{L}}
\newcommand{\HH}{\mathrm{H}}

\newcommand{\CC}{\mathrm{C}}

\newcommand{\RR}{\mathrm{R}}
\newcommand{\WW}{\mathrm{W}}

\newcommand{\cB}{\mathcal{B}}

\newcommand{\cP}{\mathcal{P}}
\newcommand{\cQ}{\mathcal{Q}}

\newcommand{\cT}{\mathcal{T}}

\newcommand{\bsm}{\mathbfscr{m}}

\newcommand{\bsu}{\mathbfscr{u}}
\newcommand{\sD}{\mathscr{D}}
\newcommand{\sE}{\mathscr{E}}
\newcommand{\sF}{\mathscr{F}}

\newcommand{\sP}{\mathscr{P}}
\newcommand{\bsP}{\mathbfscr{P}}
\newcommand{\sG}{\mathscr{G}}
\newcommand{\sM}{\mathscr{M}}

\newcommand{\sfe}{\mathsf{e}}
\newcommand{\sfE}{\mathsf{E}}
\newcommand{\sfU}{\mathsf{U}}
\newcommand{\osfU}{\overline{\mathsf{U}}}
\newcommand{\sfu}{\mathsf{u}}
\newcommand{\sfv}{\mathsf{v}}

\newcommand{\obU}{\overline{\mathbf{U}}}

\newcommand{\obM}{\overline{\mathbf{M}}}

\newcommand{\oPhi}{\overline{\Phi}}

\newcommand{\dd}[1][y]{\if#1y\,\fi{\mathrm d}}

\renewcommand{\div}{\operatorname{div}}

\DeclareMathOperator{\curl}{curl}
\DeclareMathOperator{\bcurl}{\boldsymbol{\curl}}

\newcommand{\greennote}[1]{{\color[hsb]{0.3472, 1.0, 0.6588}#1}}

\crefname{subsection}{Section}{Sections}
\Crefname{subsection}{Section}{Sections}

\crefformat{enumi}{(#2#1#3)}
\Crefformat{enumi}{(#2#1#3)}

\theoremstyle{plain}

\newtheorem{thm}{Theorem}[section]
\crefname{thm}{Theorem}{Theorems}
\Crefname{thm}{Theorem}{Theorems}

\newtheorem{proposition}[thm]{Proposition}
\newtheorem{lemma}[thm]{Lemma}
\newtheorem{corollary}[thm]{Corollary}

\newtheorem{remark}[thm]{Remark}

\newtheorem{assump}[thm]{Assumption}
\crefname{assump}{Assumption}{Assumptions}
\Crefname{assump}{Assumption}{Assumptions}

\theoremstyle{definition}

\theoremstyle{remark}

\newenvironment{proof-thm}[1]{\vskip1em \noindent\textbf{Proof of \Cref{#1}.}\\}{\hfill$\square$\\}

\makeatletter
\renewcommand\section{\@startsection{section}{1}%
  \z@%
  {1.5\linespacing\@plus\linespacing}
  {.5\linespacing}
  {\normalfont\scshape\centering}}
\makeatother

\makeatletter
\renewcommand\subsection{\@startsection{subsection}{2}%
\z@{.7\linespacing\@plus.7\linespacing}{.1\linespacing}%
{\normalfont\centering\scshape}}
\makeatother

\makeatletter
\renewcommand\subsubsection{\@startsection{subsubsection}{3}%
\z@{.5\linespacing\@plus.7\linespacing}{.1\linespacing}%
{\normalfont\raggedright\scshape}}
\makeatother

\makeatletter
\def\@tocline#1#2#3#4#5#6#7{\relax
\ifnum #1>\c@tocdepth
\else
\par \addpenalty\@secpenalty\addvspace{#2}%
\begingroup \hyphenpenalty\@M
\@ifempty{#4}{%
\@tempdima\csname r@tocindent\number#1\endcsname\relax
}{%
\@tempdima#4\relax
}%
\parindent\z@ \leftskip#3\relax \advance\leftskip\@tempdima\relax
\rightskip\@pnumwidth plus4em \parfillskip-\@pnumwidth
#5\leavevmode\hskip-\@tempdima
\ifcase #1
\or\or \hskip 1em \or \hskip 2em \else \hskip 3em \fi%
#6\nobreak\relax
\dotfill\hbox to\@pnumwidth{\@tocpagenum{#7}}\par
\nobreak
\endgroup
\fi}
\makeatother


\begin{document}

\title[A priori error analysis for ferrofluids without magnetic regularization]{A priori error analysis for the Shliomis ferrofluid model without magnetic regularization}

\author[G.A.~Benavides]{Gonzalo A.~Benavides}
\author[R.H.~Nochetto]{Ricardo H.~Nochetto}

\address{Department of Mathematics, University  of Maryland College Park, MD 20742}
\email{gonzalob@umd.edu, rhn@umd.edu}

\begin{abstract}
We develop and analyze a family of conforming, energy-stable spatial semidiscretizations of the Shliomis ferrofluid model in two and three space dimensions.
Unlike many existing numerical analyses, we treat the system without introducing magnetic diffusion.
The limited spatial control supplied by the energy law complicates the treatment of magnetic nonlinearities in the error analysis.
Our formulation retains the rotational coupling and nonlinear magnetic relaxation and incorporates the magnetic normal-flux condition naturally, overcoming the dimensional restriction and boundary-enforcement difficulties of the two-potential approach of Benavides, Nochetto and Shakipov (2026).
The error analysis combines consistent stabilization of the Maxwell constraint with a cancellation arising from exact incompressibility and the pointwise symmetry of the Hessian of the magnetic potential.
We prove global solvability of the semidiscrete scheme and establish a priori error estimates for sufficiently regular exact solutions under explicit approximation and projection-stability assumptions.
We also construct an averaged $\HH^2$-conforming projector, following Girault and Scott (2002), that is defined on $\LL^1(\Omega)$, and establish its Sobolev approximation and stability properties.
These properties are needed for the error analysis and are of independent interest.
\end{abstract}

\maketitle

\section{Introduction}\label{sec:intro}
Ferrofluids are colloidal liquids consisting of a carrier fluid (such as water) in which nanoscale ferromagnetic or ferrimagnetic particles are suspended \cite{Rosensweig2013}.
They behave as standard liquids in the absence of magnetic stimuli, while their particles become strongly magnetized by applied magnetic fields.
Due to these interesting properties, these materials are utilized in a vast number of scientific and engineering applications, such as magnetic seals, dynamic loudspeakers, computer hardware, medicine and bioengineering \cite{SchererFiguereido2005,ferrofluids-review,SocoliucEtAl2022}.
Under the assumption that ferrofluids are homogeneous monophase fluids, the two most prominent hydrodynamic models are the Rosensweig model \cite{Rosensweig2013} and the Shliomis model \cite{Shliomis1972,Shliomis2002}.
The main difference between these two models is how they treat the rotation of the ferrofluid particles.
The Rosensweig model sees the angular momentum as an independent variable, while the Shliomis model considers the rotation of the particles as sources of magnetic torque that influence the magnetization field.
Our objective in this paper is to construct energy-stable spatial discretizations of the Shliomis model without magnetic diffusion and establish a priori error estimates for sufficiently regular solutions.

\subsection{The model and its energy law}\label{sec:intro-model}

In this paper we focus on the following version of the Shliomis model.
Let $\Omega$ be a bounded domain in $\RR^{d}$ ($d \in \{2,3\}$) and let $T>0$.
We seek
a velocity field $\bu$, pressure field $p$, magnetization field $\bm$ and magnetic field $\bh$ defined on $\Omega \times (0,T)$, such that
\begin{subequations}\label{ferrofluids-system}
\begin{align}
\bu_t - \nu \Delta \bu + (\bu \cdot \nabla)\bu + \nabla p & = \mu_0(\bm \cdot \nabla)\bh + \frac{\mu_0}{2}\bcurl(\bm\times \bh), \label{eq:momentum}\\
\div \bu & = 0, \label{eq:incomp}\\
\bm_t + (\bu \cdot \nabla)\bm - \frac{1}{2}\bcurl \bu \times \bm & = -\frac{1}{\tau}(\bm - \chi \bh) - \beta \bm\times(\bm\times\bh),
\label{eq:magnetization}\\
\div(\bm+\bh-\bh_a) &= 0, \label{eq:maxwell}\\*
\bcurl \bh & = \mathbf{0} \label{eq:curl-free}.
\end{align}
Here, the positive physical parameters $\nu$, $\mu_0$, $\tau$ and $\chi$ are, respectively, the \emph{kinematic viscosity} of the ferrofluid, the \emph{vacuum permeability}, the \emph{relaxation time}, and \emph{magnetic susceptibility}.
The coefficient $\beta\geq0$ controls the nonlinear magnetic relaxation; $\beta=0$ recovers the model without this term.

The model consists of the \emph{incompressible Navier--Stokes} equations \eqref{eq:momentum}-\eqref{eq:incomp} coupled through their forcing terms to the \emph{magnetization equation} \eqref{eq:magnetization} and the \emph{magnetostatic Maxwell equations} \eqref{eq:maxwell}-\eqref{eq:curl-free}.
The vector field $\bh_a$ is the \emph{applied magnetic field}, which is assumed to satisfy $\div \bh_a = 0$ and $\bcurl \bh_a = \mathbf{0}$.
The term $\bm_t + (\bu \cdot \nabla)\bm - \frac{1}{2}\bcurl \bu \times \bm$ is the \emph{corotational derivative} of the magnetization field $\bm$ and accounts not only for the standard material transport by the fluid flow but also for the rotation of the ferrofluid particles \cite{Shliomis1972}.
In turn, $\mu_0(\bm \cdot \nabla)\bh$ is the volumetric \emph{Kelvin force} that compels the ferrofluid magnetic particles towards regions with higher magnetic field \cite{Rosensweig2013}, while $\frac{\mu_0}{2}\bcurl(\bm\times \bh)$ accounts for the magnetic torque.
The nonlinear relaxation $-\beta\bm\times(\bm\times\bh)$ is orthogonal to $\bm$ and contributes a nonnegative spin dissipation to the energy law below.

We supplement \eqref{eq:momentum}-\eqref{eq:curl-free} with the boundary and initial conditions
\begin{gather}
\bu = \mathbf{0}, \qquad (\bm+\bh-\bh_a)\cdot \bnu = 0, \qquad \text{on $\partial\Omega$} \label{eq:bcs}\\
\bu(0) = \bu_0, \qquad \bm(0) = \bm_0, \qquad \text{on $\Omega$}, \label{eq:ics}
\end{gather}
where $\bnu$ denotes the outward unit normal to $\Omega$.
\end{subequations}
Formally, system \eqref{ferrofluids-system} obeys the following energy law \cite{NochettoTrivisaWeber2019,GiraultRaviart1986}:
\begin{equation}\label{eq:continuous-energy law}
\frac{1}{2}\frac{\dd}{\dd t}\left( \|\bu\|^2 + \frac{\mu_0}{\chi} \|\bm\|^2 + \mu_0 \|\bh\|^2 \right) + \nu \|\nabla \bu\|^2 + \frac{\mu_0}{\tau \chi}\|\bm - \chi \bh\|^2 +\mu_0\beta\|\bm\times\bh\|^2
= \mu_0 \int_\Omega \partial_t \bh_a \cdot \bh,
\end{equation}
which is achieved by testing the momentum equation \eqref{eq:momentum} with $\bu$, the incompressibility condition \eqref{eq:incomp} with $p$ and the magnetization equation \eqref{eq:magnetization} with $\frac{\mu_0}{\chi} (\bm - \chi \bh)$.
It hinges on several integrations by parts, the symmetry of $\nabla\bh$ (because $\bcurl\bh = \mathbf{0}$), the divergence-free conditions of $\bu$ and $\bm + \bh$ and uses the time-differentiated magnetostatic equation $(\bm_t, \bh) = - \frac{1}{2} \frac{\dd}{\dd t} \|\bh\|^2 + (\partial_t \bh_a, \bh)$.

\subsection{Analytical difficulties and magnetic regularization}\label{sec:intro-analysis}

The Shliomis, Rosensweig and related models have been thoroughly analyzed in the past decades;
we refer, for instance, to \cite{AmiratHamdacheMurat2008,AmiratHamdache2008,AmiratHamdache2009,AmiratHamdache2010,NochettoTrivisaWeber2019} for results concerning the existence of weak and strong solutions.
The analysis of the Shliomis and Rosensweig models is particularly complex due to the limited spatial control on $\bm$ and $\bh$ that the energy law provides;
namely $\bm, \bh \in \LL^\infty(0,T;\bL^2(\Omega))$ without any bounds for $\nabla\bm$ or $\nabla\bh$.
This makes it difficult to give an appropriate meaning to the Kelvin force $\mu_0 (\bm \cdot \nabla)\bh$, and consequently to pass to the limit in this term and the other nonlinear terms of the system in the context of, e.g., Faedo--Galerkin schemes.
Regarding the existence of weak solutions, these considerations prompted Amirat and Hamdache in \cite{AmiratHamdache2008} and Amirat, Hamdache and Murat \cite{AmiratHamdacheMurat2008} to respectively consider \emph{regularized} Shliomis and Rosensweig models by introducing a diffusion term $-\sigma \Delta \bm$, $\sigma > 0$ on the left-hand side of the magnetization equation, thus providing sufficient spatial control to circumvent the aforementioned difficulties.

The case $\sigma = 0$ is not only of mathematical interest but it is motivated by the underlying physical model.
Spatial diffusion of $\bm$ is absent from the original formulations of Shliomis and Rosensweig \cite{Shliomis2002,Shliomis1972,Rosensweig2013}, and the physical justification of adding $-\sigma\Delta\bm$ has been questioned \cite{NochettoSalgadoTomas2016,LanthalerWeber2026}.
In \eqref{eq:magnetization}, the evolution of $\bm$ is governed by material transport, corotation and local relaxation; adding spatial diffusion introduces an additional dissipative mechanism and requires further boundary conditions for $\bm$ \cite[Section~2]{NochettoSalgadoTomas2016}.
Direct approximation of the unregularized system therefore preserves the chosen magnetization dynamics without these additional modeling assumptions.

For the unregularized system (i.e.~$\sigma = 0$), Amirat and Hamdache established local strong solutions under stronger regularity assumptions on the data \cite{AmiratHamdache2009,AmiratHamdache2010}.
Later, Nochetto, Trivisa and Weber \cite{NochettoTrivisaWeber2019} showed the existence of global weak solutions to the Rosensweig model by considering a sequence of regularized systems and then recovering a solution as $\sigma \to 0$ by resorting to the notion of \emph{renormalized solutions} in the sense of DiPerna and Lions \cite{DipernaLions1989}, and standard compactness arguments.
A key step in \cite{NochettoTrivisaWeber2019} was to rewrite the Kelvin force as $\mu_0 \div(\bh \otimes (\bm + \bh))$, at the expense of a gradient term.

\subsection{Numerical approximation: results and limitations}\label{sec:intro-numerics}

The limited spatial control on $\bm$ and $\bh$ supplied by the energy law \eqref{eq:continuous-energy law} also conditions the design and analysis of numerical methods.
Although choosing appropriate finite-dimensional subspaces with sufficient regularity makes the discrete nonlinear terms well defined, spatial bounds provided via inverse inequalities may deteriorate as the mesh is refined.
Consequently, discrete energy stability akin to \eqref{eq:continuous-energy law} on its own does not yield the compactness needed to pass to the limit in the magnetic nonlinearities, nor sufficient control to compare a sufficiently regular exact solution with its discrete approximation.

In addition, there is a compatibility issue that further restricts the discretization:
the test function $\bm-\chi\bh$ for the magnetization equation in order to get \eqref{eq:continuous-energy law} may not have an admissible discrete counterpart when the magnetization $\bm$ and magnetic field $\bh$ are approximated in different spaces.

These difficulties have motivated several approaches and simplifications, such as discretizing the regularized system with magnetic diffusion $-\sigma \Delta \bm$, $\sigma > 0$, considering a simpler magnetostatic coupling, using compatible mixed spaces, and reformulating the magnetic equations to permit independent testing of $\bm$ and $\bh$.

Concerns regarding computational efficiency have also motivated energy-stable schemes designed to reduce the algebraic coupling \cite{ZhangHeYang2021,ZhangPanHeYang2025,ZhangHeYang2025,DongHuangHuangYang2026}.
These works establish discrete energy stability and assess accuracy through numerical experiments, but do not provide convergence proofs for the resulting schemes, with or without rates.
In the discussion below, we focus on how the treatment of the magnetic equations affects convergence and error estimates.

\subsubsection*{Early schemes without magnetic regularization}

For the case of $\sigma = 0$, we highlight the work of Nochetto, Salgado and Tomas \cite{NochettoSalgadoTomas2016}, who designed an energy-stable fully discrete scheme for the Rosensweig model coupled to the magnetostatic equations.
Their construction uses a fully discontinuous discrete space for $\bm$, and a standard $\HH^1(\Omega)$-conforming space for the magnetic potential $\phi$, and recovers the magnetic field via the identity $\bh = \nabla \phi$.
They discretize the convective and advective operators involved in the spirit of the discontinuous Galerkin methods of Di Pietro and Ern \cite{DiPietroErn2010}.
Their method reproduces the cancellations of the continuous energy law, and consequently they are able to establish the existence of discrete solutions.
Yet, the lack of $\LL^2(0,T;\bH^1(\Omega))$ control on $\bm$ and $\bh$ prevented them from establishing convergence of the full model.
Their convergence analysis focused, instead, on the simplified scheme resulting from prescribing the magnetic field $\bh$.
Building on \cite{NochettoSalgadoTomas2016}, the same authors developed a phenomenological diffuse-interface model for two-phase ferrofluid flows and designed an energy-stable numerical approximation in \cite{NochettoSalgadoTomas2016TwoPhase}.

We highlight \cite{ZhangHeYang2023} as an important source of inspiration for the development of the present manuscript, as attested in \cref{sec:equivalent-formulation}.
Zhang, He and Yang introduce a magnetic reformulation for a two-phase Shliomis model without magnetic diffusion and derive an unconditionally energy-stable fully discrete scheme.
No convergence analysis was carried out for this scheme.
Their main innovation is to combine the magnetostatic equation, its time derivative and the magnetization equation into an evolution relation for the magnetic potential $\phi$, permitting independent testing of $\bm$ and $\bh$.

\subsubsection*{Error estimates for regularized or simplified models}

Concerning models with magnetic regularization (i.e.~$\sigma>0$), Wu and Xie constructed in \cite{WuXie2023,WuXie2025} energy-stable mixed finite element methods for, respectively, the Shliomis and Rosensweig models.
They showed the discrete solvability of their schemes and derived optimal-order error estimates for spatial semidiscretizations and fully discrete schemes under suitable regularity and additional stability assumptions.
Their convergence analysis relies on the spatial control supplied by magnetic diffusion: absorbing nonlinear terms against the $\sigma$-weighted dissipation introduces inverse powers of $\sigma$ into the coefficients of the Gr\"onwall argument.
These coefficients yield bounds containing factors of the form $\exp\left(C\sigma^{-k}\right)$ for some $k>0$, which are nonuniform as $\sigma\to0$.
Thus these estimates do not establish convergence for \eqref{ferrofluids-system} under any polynomial scaling between $h$ and $\sigma$.
Moreover, convergence for fixed $\sigma>0$ concerns a modified continuous model; approximating \eqref{ferrofluids-system} by such a scheme would additionally require control of the regularization error together with the $\sigma$-dependent discretization error.

Regarding the temporal semidiscretization of the regularized Shliomis model, Dong, Huang and Huang \cite{DongHuangHuang2025} proposed an energy-stable scheme and derived its corresponding error estimates.
As in \cite{WuXie2023,WuXie2025}, their estimates are nonuniform as $\sigma \to 0$.
We refer the reader to \cite{DongHuangHuangTang2025,KeramHuangHe2026} for further convergence results and time discretization techniques for the regularized Rosensweig model.

Other a priori error analyses without magnetic regularization concern, however, simplified magnetic couplings or the stationary versions of the Shliomis model.
For instance, Mao and Sun \cite{MaoSun2025} establish unconditional error estimates for a first-order fully discrete finite element approximation of a simplified Shliomis model with, akin to \cite{NochettoSalgadoTomas2016,NochettoSalgadoTomas2016TwoPhase}, prescribed effective magnetic field.
In turn, Wu and Xie \cite{WuXie2023Parallel} analyze mixed finite element methods under the simplifying assumption that $\bm$ is parallel to $\bh$, effectively obtaining a decoupled system.
For stationary magnetic flows, Benavides, Caucao, Gatica and Sobral \cite{BenavidesCaucaoGaticaSobral2026} develop two different mixed formulations posed in reflexive Banach spaces, analyze their Galerkin approximations, and propose iterative approximations of a simplified magnetization equation in which the advection $(\bu \cdot \nabla)\bm$ is neglected.

\subsubsection*{Error estimates for the coupled unregularized model}

To the best of our knowledge, Benavides, Nochetto and Shakipov \cite{BenavidesNochettoShakipov2026} provided the first a priori error estimates for a spatial discretization of a coupled, time-dependent ferrofluid model without magnetic diffusion and with the magnetic field determined by the magnetostatic equations.
Their construction represents the magnetic field and magnetic induction through two scalar potentials, thereby enforcing the curl and divergence constraints \eqref{eq:curl-free} and \eqref{eq:maxwell} exactly and recovering the cancellations needed for a discrete perturbation estimate.
This establishes quantitative convergence without relying on magnetic diffusion or prescribing the effective magnetic field.

Despite this advance, the approach in \cite{BenavidesNochettoShakipov2026} has important limitations.
Its two-scalar-potential discretization and error analysis are inherently restricted to two space dimensions.
The scalar representation of the divergence-free magnetic induction is specific to 2D; an analogous construction in 3D would require a vector potential and a treatment of its gauge freedom.
The reduced model considered there also omits the rotational coupling terms present in \eqref{eq:momentum} and \eqref{eq:magnetization}.
In addition, the discrete induction potential $\Psi$ must satisfy $\Psi=0$ on $\partial\Omega$, while its normal derivative remains unconstrained.
For the $C^1$-conforming elements used there, imposing this trace condition strongly without inadvertently constraining derivative degrees of freedom is highly nontrivial \cite{KirbyMitchell2019}.
The implementation in \cite[Section~4]{BenavidesNochettoShakipov2026} therefore resorts to an Uzawa--Powell iteration with Nesterov acceleration at each time step to enforce $\Psi=0$ to a prescribed tolerance.
The error analysis also relies on auxiliary projection stability properties that are not fully verified for  the concrete finite element spaces considered there.

\subsection{Contributions of the present work}\label{sec:intro-contributions}

Recall that the error estimates for magnetically regularized models are nonuniform as $\sigma\to0$ and that the physical justification of the added magnetic diffusion has been questioned \cite{NochettoSalgadoTomas2016,LanthalerWeber2026}.
Together with the dimensional and modeling restrictions of the two-potential approach in \cite{BenavidesNochettoShakipov2026}, these concerns motivate the present work.
We develop an energy-stable conforming spatial semidiscretization of the unregularized Shliomis system \eqref{ferrofluids-system} and establish a priori error estimates for sufficiently regular solutions.
The formulation applies in two and three space dimensions, retains the rotational coupling and nonlinear magnetic relaxation, and incorporates the magnetic boundary condition \eqref{eq:bcs} naturally.
To control the magnetic nonlinearities, we combine consistent stabilization of the Maxwell constraint \eqref{eq:maxwell} with a cancellation based on exact incompressibility and the $\HH^2$-conformity of the magnetic potential.

\subsubsection*{A formulation in two and three dimensions}
We approximate $\bm$ by an independent field $\bM$ and $\phi$ by a single scalar potential $\Phi$, coupled to an exactly divergence-free approximation $\bU$ of $\bu$.
No potential for the magnetic induction is needed, so the same choice of magnetic unknowns applies in both two and three space dimensions.
Building on the reformulation in \cite{ZhangHeYang2023}, the discrete equations for $\bM$ and $\Phi$ can be tested independently.
The nonlinear relaxation is incorporated into both equations so that their energy balance includes the spin dissipation $\mu_0\beta\|\bM\times\nabla\Phi\|^2$.
The variational flux equation in \cref{sec:equivalent-formulation} incorporates the magnetic boundary condition \eqref{eq:bcs} naturally, avoiding the auxiliary Uzawa--Powell iteration used in \cite{BenavidesNochettoShakipov2026}.

We choose an $\HH^2$-conforming space for $\Phi$, so that $\bH:=\nabla\Phi\in\bH^1(\Omega)$ and $\nabla\bH=\nabla^2\Phi$ is symmetric.
These properties permit the global integrations by parts used in the error analysis without additional interior-face terms.

\subsubsection*{Consistent stabilization and error analysis}

The central analytical step is to allow a nonzero discrete Maxwell residual $r_h:=\Delta\Phi+\div\bM=\div(\bH+\bM-\bh_a)$, the discrete counterpart of \eqref{eq:maxwell}, and control it through consistent stabilization.
This replaces the exact enforcement of \eqref{eq:maxwell} built into the two-potential construction of \cite{BenavidesNochettoShakipov2026}.
Since the stabilization terms vanish on exact solutions of \eqref{ferrofluids-system}, they are consistent with the continuous model and do not introduce a regularization error;
hence their coefficients can remain fixed under mesh refinement.
Their contribution to the discrete energy law provides control of $r_h$.

The key cancellation below explains how this residual control enters the nonlinear error analysis.
Let $\bV\in\bH^1_0(\Omega)$, $\mathbf{Z}\in\bH^1(\Omega)$, and $\Psi\in\HH^2(\Omega)$ be arbitrary conforming discrete functions satisfying $\div\bV=0$.
The zero trace of $\bV$ eliminates all boundary terms in the integrations by parts below.
Therefore, the incompressibility $\div \bV = 0$ gives
\begin{equation*}
\int_\Omega ((\bV\cdot\nabla)\nabla\Psi)\cdot\nabla\Psi\dd x
=\frac12\int_\Omega \bV\cdot\nabla|\nabla\Psi|^2\dd x=0.
\end{equation*}
Using this vanishing property and the crucial symmetry of $\nabla^2\Psi$, we obtain
\begin{equation}\label{eq:intro-magnetic-cancellation}
\begin{aligned}
\int_\Omega ((\bV\cdot\nabla)\mathbf{Z})\cdot\nabla\Psi\dd x
&=\int_\Omega ((\bV\cdot\nabla)(\mathbf{Z}+\nabla\Psi))\cdot\nabla\Psi\dd x
&&\text{(incompressibility)}\\
&=-\int_\Omega ((\bV\cdot\nabla)\nabla\Psi)\cdot(\mathbf{Z}+\nabla\Psi)\dd x
&&\text{(integration by parts, }\div\bV=0\text{)}\\
&=-\int_\Omega \bigl((\mathbf{Z}+\nabla\Psi)\cdot\nabla\bigr)\nabla\Psi\cdot\bV\dd x
&&\text{(Hessian symmetry)}\\
&=\int_\Omega \bigl((\mathbf{Z}+\nabla\Psi)\cdot\nabla\bigr)\bV\cdot\nabla\Psi\dd x\\
&\quad+\int_\Omega (\div\mathbf{Z}+\Delta\Psi)(\bV\cdot\nabla\Psi)\dd x
&&\text{(integration by parts).}
\end{aligned}
\end{equation}

In the error analysis of \cref{sec:a-priori}, we apply \eqref{eq:intro-magnetic-cancellation} with $(\bV,\mathbf{Z},\Psi)=(\obU,\bE_\bm,E_\phi)$, where $\obU$ is an auxiliary projection of $\bu$ and $\bE_\bm$ and $E_\phi$ are discrete magnetization and potential errors.
Consequently, this yields
\begin{equation*}
\int_\Omega ((\obU \cdot \nabla) \bE_\bm) \cdot \nabla E_\phi
= \int_\Omega ((\bE_\bm + \nabla E_\phi) \cdot \nabla) \obU \cdot \nabla E_\phi
+ \int_\Omega (\div \bE_\bm + \Delta E_\phi)(\obU \cdot \nabla E_\phi),
\end{equation*}
which explicitly transfers derivatives from $\bE_\bm$ to $\obU$ and collects the remaining magnetization derivatives into $\div\bE_\bm+\Delta E_\phi$, precisely the quantity controlled by the stabilization in the discrete error equation.
This is exactly what we require to close the error analysis without imposing separate bounds for $\nabla\bE_\bm$ or $\nabla^2E_\phi$ that are not provided by the discrete energy law.

We prove global solvability of the semidiscrete scheme.
Under additional approximation and stability assumptions on the discrete spaces and auxiliary projections, we derive a priori error estimates for sufficiently regular exact solutions.
We provide concrete finite element spaces in 2D that satisfy these approximation and stability assumptions.
For the 3D choices discussed in \cref{sec:specific-FES}, the error estimates remain conditional on the velocity-projection stability described in \cref{sec:u-p}.

That being said, the stabilization also entails potential computational and analytical drawbacks.
Its fourth-order contribution to the equation for $\Phi$ may worsen the conditioning of the algebraic systems arising after time discretization; a quantitative analysis of this effect and the development of effective preconditioners are left for future work.
Moreover, the present error analysis bounds the stabilization consistency terms through $\bH^1$-approximation of $\bm$ and $\HH^2$-approximation of $\phi$, which complicates the derivation of sharper $\LL^\infty(0,T;\bL^2(\Omega))$ error estimates for $\bm$ and $\bh$ in the spirit of Wheeler's trick.
Nevertheless, our three-dimensional experiment exhibits nearly third-order convergence for both magnetic variables, suggesting that sharper error estimates may be possible.
Whether a projection argument adapted to the coupled stabilization can capture this behavior remains an open question.

\subsubsection*{An averaged $\HH^2$-conforming projector}

An additional contribution of this paper is the construction and analysis of an averaged projector $J_h$ onto the underlying $\HH^2$-conforming macroelement spaces.
Building upon the averaging principle of Girault and Scott \cite{GiraultScott2002}, the construction utilizes volume averages, does not require pointwise derivatives, and is consequently well-defined for all functions in $\LL^1(\Omega)$.
For the full and reduced Hsieh--Clough--Tocher and Walkington families described in \cref{sec:magnetic-potential}, \cref{lem:potential-interpolation} establishes optimal approximation in $\WW^{j,p}(\Omega)$ ($j=0,1,2$) and $\WW^{2,p}(\Omega)$-stability for every $1\leq p\leq\infty$ on general shape-regular triangulations.
These estimates are of independent interest, and their $p=\infty$ case proves the uniform stability needed for the auxiliary magnetic potential in our nonlinear error analysis.
This supplies an explicitly constructed alternative to the Ritz-projection stability assumptions in Benavides, Nochetto and Shakipov \cite[Section~3.7]{BenavidesNochettoShakipov2026}.

\subsection*{Structure of the paper}

The rest of this paper is organized as follows.
In \cref{sec:prelim}, we introduce some standard notation for functional spaces and differential operators to be used throughout the rest of this paper.
Building upon \cite{ZhangHeYang2023}, in \cref{sec:equivalent-formulation}, we rewrite system \eqref{ferrofluids-system} by deriving an evolution equation for the magnetic potential.
In \cref{sec:discretization}, we introduce the abstract semidiscrete scheme and establish its solvability and energy stability.
\cref{sec:a-priori} is devoted to deriving a priori error estimates for smooth enough solutions.
Concrete finite element spaces and auxiliary projections are discussed in detail in \cref{sec:specific-FES}, with the exception of the proof of the stability and approximation properties of $J_h$, which is postponed to \cref{app:proof-interp}.
Finally, in \cref{sec:experiments}, we present numerical experiments assessing spatial convergence in two and three dimensions.

\subsection{Preliminary notation and identities}\label{sec:prelim}
Let $\Omega \subseteq \RR^{d}$ ($d \in \{2,3\}$) be a bounded and simply connected polygonal (resp.~polyhedral) Lipschitz domain.
We use the notation  $\LL^q(\Omega)$ and $\WW^{m,q}(\Omega)$ for the scalar Lebesgue and Sobolev spaces with integrability $q \in [1,\infty]$ and differentiability $m \geq 0$; we denote the corresponding norms respectively by $\|\cdot\|_{0,q;\Omega}$ and $\|\cdot\|_{m,q;\Omega}$. We also write $\WW^{0,q}(\Omega) = \LL^q(\Omega)$ and $\HH^m(\Omega) := \WW^{m,2}(\Omega)$ with corresponding norm $\|\cdot\|_{m,\Omega} := \|\cdot\|_{m,2;\Omega}$;
for $m=0$ we simply write $\|\cdot\|_\Omega$.
The associated Sobolev seminorms are denoted by using $|\cdot|$ instead of $\|\cdot\|$.
When there is no ambiguity, we will omit the explicit dependence on $\Omega$.
We denote by $\WW^{m,q}_\#(\Omega)$ the closed subspace of scalar-valued functions $v$ in $\WW^{m,q}(\Omega)$ that satisfy $\int_\Omega v =0$.
In turn, $\WW^{1,q}_0(\Omega)$ denotes the closed subspace of $\WW^{1,q}(\Omega)$ of functions with vanishing trace.
For any scalar function space $V$ defined on $\Omega$, we denote by $\bV$ its $d$-vectorial counterpart, and we use the same notation for its norms.

For any normed space $V$ its dual space is denoted by $V'$, whose induced norm is given by $\|f\|_{V'} := \sup_{0 \neq v \in V} \frac{f(v)}{\|v\|_V}$.
We denote the duality pairing between $V'$ and $V$ by $\langle \cdot, \cdot \rangle_{V' \times V}$;
when there is no ambiguity, we will simply write $\langle \cdot,\cdot \rangle$.

Throughout this article, for any integrability parameter $p \in [1,\infty]$, we denote by $p'$ its conjugate exponent.
For any pair of functions $v \in \LL^p(\Omega)$ and $w \in \LL^{p'}(\Omega)$ we write $(v,w) := \int_\Omega v \cdot w$.
We naturally extend this notation for vector- and tensor-valued functions.

For the rest of this section, lowercase Greek letters denote scalar-valued functions, whereas bold lowercase Roman letters denote vector-valued functions.
For $d=3$, we define
\begin{equation*}
\bcurl \bv := [\partial_2 v_3 - \partial_3 v_2, \partial_3 v_1 - \partial_1 v_3, \partial_1 v_2 - \partial_2 v_1],
\end{equation*}
and the vector product $\bv \times \bw$ between $\bv$ and $\bw$ is defined in the standard way.

For $d = 2$, we define
\begin{align*}
\bcurl \varphi & := [\partial_2 \varphi, -\partial_1 \varphi], \qquad \curl \bv := \partial_1 v_2 - \partial_2 v_1,\\
\bv \times \bw & := v_1 w_2 - v_2 w_1, \qquad \omega \times \bv := [-\omega v_2, \omega v_1].
\end{align*}
We also set $\bv\times\omega:=-\omega\times\bv=[\omega v_2,-\omega v_1]$, so that the double cross product in \eqref{eq:magnetization} is well defined in two dimensions.
The notation introduced above allows us to extend the classical \emph{triple product} identity for $d=3$,
\begin{equation}\label{eq:triple-product-identity}
\bu \cdot (\bv \times \bw) = (\bu \times \bv) \cdot \bw,
\end{equation}
to $d=2$:
\begin{equation*}
\omega \, (\bv \times \bw) = (\omega \times \bv) \cdot \bw.
\end{equation*}
For $d\in\{2,3\}$ and $\ba,\bb,\bc\in\RR^d$, the cross products satisfy
\begin{equation}\label{eq:spin-identities}
-\bigl(\ba\times(\ba\times\bb)\bigr)\cdot\bc=(\ba\times\bb)\cdot(\ba\times\bc),
\quad
\bigl(\ba\times(\ba\times\bb)\bigr)\cdot\ba=0,
\quad
\bigl(\ba\times(\ba\times\bb)\bigr)\cdot\bb=-|\ba\times\bb|^2.
\end{equation}
In two dimensions, the product on the right-hand side of the first identity is the product of two scalars.

We finish this section by introducing the trilinear form $\cB$ defined by
\begin{equation*}
\cB(\bu, \bv, \bw) :=  ((\bu \cdot \nabla) \bv, \bw)_\Omega,
\end{equation*}
for sufficiently smooth vector-valued functions such that the integral above makes sense.
Integration by parts yields
\begin{equation}\label{eq:general-IBP}
\cB(\bu,\bv,\bw) + \cB(\bu,\bw,\bv) = \int_{\partial\Omega} (\bu \cdot \bnu)(\bv \cdot \bw) - (\div\bu, \bv \cdot \bw)_\Omega.
\end{equation}
In particular, if $\bu$ satisfies $\div \bu = 0$ in $\Omega$ and $\bu \cdot \bnu = 0$ on $\partial\Omega$, then
\begin{equation}\label{eq:skew-symmetry}
\cB(\bu,\bv,\bw) = - \cB(\bu,\bw,\bv).
\end{equation}
On the other hand, if $\bv = \nabla\psi$ for some twice-differentiable scalar function $\psi$, then the symmetry of the Hessian $\nabla^2\psi$ implies
\begin{equation}\label{eq:symmetry}
\cB(\bu,\bv,\bw) = \cB(\bw,\bv,\bu).
\end{equation}
%

%

\begin{remark}[two-dimensional model]
For $d=2$, the model in \eqref{ferrofluids-system} and the subsequent analysis use $\curl$ whenever its argument is a vector field, including velocities, test functions, error fields, and the applied field.
The operator $\bcurl$ is retained on scalar arguments, in particular on $\bm\times\bh$ in the momentum equation.
Thus $\bcurl\bu\times\bm$ becomes $(\curl\bu)\times\bm$, the applied-field condition becomes $\curl\bh_a=0$, and rotational pairings such as $(\bm\times\bh,\curl\bv)$ are scalar $\LL^2$ products.
In particular, $-\bm\times(\bm\times\bh)=(\bm\times\bh)\times\bm$ is a planar vector, while $\|\bm\times\bh\|$ is the scalar $\LL^2$-norm.
With the planar products defined above, the same integration-by-parts and cancellation identities utilized throughout the rest of this article are still valid.
\end{remark}

\section{Equivalent formulation}\label{sec:equivalent-formulation}
In this section, building on \cite{ZhangHeYang2023}, we equivalently rewrite \eqref{ferrofluids-system} by deriving an evolution equation for the magnetic potential that allows $\bm$ and $\bh$ to be tested independently when they lie in different finite-dimensional spaces;
this is a major difference with respect to \cite{BenavidesNochettoShakipov2026}.
The calculations in this section are formal: we assume that all fields are sufficiently smooth up to the fixed boundary to justify the differentiations, normal traces, and integrations by parts below.
We first notice that, because $\Omega$ is simply connected, condition \eqref{eq:curl-free} is equivalent to the existence of a scalar field $\phi$ with $\int_\Omega \phi = 0$ such that $\bh = \nabla \phi$.
Because we will differentiate the magnetic potential $\phi$ in time, we now must determine its initial condition.
Since $\bm$ and $\bh = \nabla\phi$ are related via the magnetostatic equation \eqref{eq:maxwell} and the normal-flux condition \eqref{eq:bcs}, the initial potential $\phi_0 \in \HH^1(\Omega) \cap \LL^2_\#(\Omega)$ is determined by
\begin{equation}\label{eq:initial-Maxwell-compatibility}
(\bm_0+\nabla\phi_0-\bh_a(0),\nabla\psi)=0
\qquad\text{for every }\psi\in \HH^1_\#(\Omega).
\end{equation}
For $\bm_0,\bh_a(0)\in \bL^2(\Omega)$, this weak Neumann problem has a unique mean-zero solution.
Thus $\phi_0$ is fully determined by $\bm_0$ and $\bh_a(0)$.

Set
\begin{equation}\label{eq:maxwell-defect}
\bg:=\bm+\nabla\phi-\bh_a.
\end{equation}
The magnetostatic and normal-flux constraints in \eqref{eq:maxwell} and \eqref{eq:bcs} read
\begin{equation}\label{eq:maxwell-strong-constraints}
\div\bg=0 \quad\text{in }\Omega,
\qquad
\bg\cdot\bnu=0 \quad\text{on }\partial\Omega.
\end{equation}
Using \eqref{eq:magnetization}, we obtain the identity
\begin{equation}\label{eq:augmented-flux}
\partial_t\bg+\frac{1}{\tau}\bg
=\nabla\phi_t-(\bu\cdot\nabla)\bm
  +\frac12(\bcurl\bu)\times\bm +\frac{1+\chi}{\tau}\nabla\phi
  -\beta\bm\times(\bm\times\nabla\phi)
  -\partial_t\bh_a-\frac1\tau\bh_a
  =:\bF.
\end{equation}
Applying $\partial_t+\tau^{-1}$ to \eqref{eq:maxwell-strong-constraints} therefore gives the augmented equation and its boundary condition in strong form:
\begin{equation}\label{eq:augmented-strong}
\div\bF=0 \quad\text{in }\Omega,
\qquad
\bF\cdot\bnu=0 \quad\text{on }\partial\Omega.
\end{equation}
In particular, since $\div\bh_a=0$, the interior equation in \eqref{eq:augmented-strong} can be written explicitly as
\begin{equation}\label{eq:augmented-potential-strong}
\partial_t\Delta\phi+\frac{1+\chi}{\tau}\Delta\phi
=\div\left( (\bu\cdot\nabla)\bm
 -\frac12(\bcurl\bu)\times\bm
 +\beta\bm\times(\bm\times\nabla\phi)\right).
\end{equation}
The applied field enters through the boundary condition in \eqref{eq:augmented-strong}.

Conversely, suppose that the magnetization equation \eqref{eq:magnetization}, and \eqref{eq:augmented-flux} and \eqref{eq:augmented-strong} hold.
Taking the divergence and the normal component of \eqref{eq:augmented-flux} shows that $r:=\div\bg$ and $q:=\bg\cdot\bnu$ satisfy
\begin{equation}\label{eq:constraint-propagation}
\partial_t r+\frac1\tau r=0 \quad\text{in }\Omega,
\qquad
\partial_t q+\frac1\tau q=0 \quad\text{on }\partial\Omega.
\end{equation}
Thus $r(t)=e^{-t/\tau}r(0)$ and $q(t)=e^{-t/\tau}q(0)$.
Under the smoothness assumption above, the initial compatibility condition in \eqref{eq:initial-Maxwell-compatibility} gives $r(0)=q(0)=0$.
Consequently, \eqref{eq:maxwell-strong-constraints} (equivalently \eqref{eq:maxwell} and the second equation of \eqref{eq:bcs}) holds for all times.

In conclusion, recalling that $\bh = \nabla\phi$, we are able to equivalently write the bulk equations and boundary conditions of \eqref{ferrofluids-system} as
\begin{align}
\bu_t - \nu \Delta \bu + (\bu \cdot \nabla)\bu + \nabla p & = \mu_0(\bm \cdot \nabla)\bh + \frac{\mu_0}{2}\bcurl(\bm\times \bh), \tag{\ref{eq:momentum}}\\
\div \bu & = 0, \tag{\ref{eq:incomp}}\\
\bm_t + (\bu \cdot \nabla)\bm - \frac{1}{2}\bcurl \bu \times \bm & = -\frac{1}{\tau}(\bm - \chi \bh) -\beta\bm\times(\bm\times\bh), \tag{\ref{eq:magnetization}}\\
\bF +\beta\bm\times(\bm\times\nabla\phi) & = \nabla\phi_t-(\bu\cdot\nabla)\bm +\frac12(\bcurl\bu)\times\bm +\frac{1+\chi}{\tau}\nabla\phi -\partial_t\bh_a-\frac1\tau\bh_a, \label{eq:full-magnetic-flux}\\
  \div\bF & = 0, \label{eq:full-magnetic-incomp}\\
\int_\Omega \phi & = 0 \label{eq:pot-normaliz},
\end{align}
and
\begin{equation}\label{eq:total-magnetic-flux}
\bu = \mathbf{0}, \quad \text{on $\partial\Omega$}, \qquad \bF \cdot \bnu = 0, \quad \text{on $\partial\Omega$}.
\end{equation}
The evolution formulation is supplemented with the initial conditions \eqref{eq:ics} for $\bu$ and $\bm$, and $\phi(0)=\phi_0$, where $\phi_0$ is the mean-zero solution of \eqref{eq:initial-Maxwell-compatibility}.
This compatible initial potential is required for the constraint defects in \eqref{eq:constraint-propagation} to vanish for all times.

\section{Semidiscretization}\label{sec:discretization}
In this section, we present our abstract semidiscrete scheme under general assumptions that will be specified throughout.
We prove an energy law consistent with \eqref{eq:continuous-energy law} and global existence and uniqueness for the reduced solution $(\bU,\bM,\Phi)$.
The pressure is determined modulo the nullspace of the transpose of the discrete divergence operator.

We begin by considering general finite dimensional spaces
\begin{equation}\label{eq:fin-dim-subspaces}
\bbU \times \bbP \subset \bW^{1,\infty}_0(\Omega) \times \LL^\infty_\#(\Omega), \qquad
\bbM \subset \bW^{1,\infty}(\Omega), \qquad
\bbX \subset \WW^{2,\infty}_\#(\Omega),
\end{equation}
for, respectively, the velocity-pressure pair $(\bu,p)$, the magnetization field $\bm$ and the scalar potential $\phi$.
We will specify these spaces later in \cref{sec:specific-FES};
for the moment we simply assume the following.
\begin{assump}[strongly divergence-free pair]\label{assump:strongly-div-free}
We assume that $(\bbU,\bbP)$ is \emph{strongly divergence-free}, i.e., if $\bV \in \bbU$ is such that $(Q, \div\bV) = 0$ for each $Q \in \bbP$, then $\bV \in \bbUdiv$, where
\begin{equation*}
\bbUdiv := \{\bV \in \bbU: \div \bV = 0, \quad \text{on $\Omega$}\},
\end{equation*}
\end{assump}

The $\HH^2$-conformity of $\bbX$ is essential: its elements have globally defined weak Hessians, which are symmetric almost everywhere.
These properties and \cref{assump:strongly-div-free} justify the global integration by parts used in the energy and error identities of this section and \cref{sec:a-priori} below, without the occurrence of interior-face jump terms.

We now introduce the abstract semidiscrete scheme for the formulation presented in \cref{sec:equivalent-formulation}, which is formally equivalent to the system \eqref{ferrofluids-system}.
Let $\kappa > 0$ be a fixed stabilization parameter independent of the chosen finite dimensional spaces \eqref{eq:fin-dim-subspaces}.
We seek a time-dependent tuple $(\bU,P,\bM,\Phi)$ with values in $\bbU \times \bbP \times \bbM \times \bbX$, such that for all test functions $(\bV_\bu,V_p,\bV_\bm,V_\phi) \in \bbU \times \bbP \times \bbM \times \bbX$ and almost every $t \in (0,T)$, the following equations hold
\begin{subequations}\label{all-semidiscrete-scheme}
\begin{equation}\label{eq:semidiscrete-scheme}
\begin{aligned}
(\partial_t \bU,\bV_\bu) + \nu(\nabla \bU, \nabla \bV_\bu) + \cB(\bU,\bU,\bV_\bu) -(P, \div \bV_\bu) & = \mu_0 \cB(\bM,\nabla\Phi,\bV_\bu) + \frac{\mu_0}{2} (\bM \times \nabla\Phi, \bcurl \bV_\bu), \\
(V_p, \div \bU) & = 0,\\[0.5ex]
(\partial_t \bM, \bV_\bm) + \cB(\bU, \bM, \bV_\bm) - \frac{1}{2} (\bcurl \bU, \bM \times \bV_\bm) &\\
\quad +\kappa \chi\, (\Delta\Phi + \div\bM, \div \bV_\bm)
 & = - \frac{1}{\tau} (\bM - \chi \nabla\Phi, \bV_\bm) +\beta(\bM\times\nabla\Phi,\bM\times\bV_\bm),\\[0.5ex]
(\nabla \partial_t \Phi, \nabla V_\phi) + \frac{1+\chi}{\tau}(\nabla \Phi, \nabla V_\phi) - \cB(\bU,\bM, \nabla V_\phi) &\\
\quad + \kappa(\Delta\Phi + \div\bM, \Delta V_\phi)
+ \frac{1}{2}(\bcurl \bU, \bM \times \nabla V_\phi) & = \left(\frac{1}{\tau}\bh_a + \partial_t \bh_a, \nabla V_\phi\right) -\beta(\bM\times\nabla\Phi,\bM\times\nabla V_\phi).
\end{aligned}
\end{equation}
We supplement these equations with the initial conditions
\begin{equation}\label{eq:discrete-initial-conditions}
\bU(0) = \bU_0, \qquad \bM(0) = \bM_0, \qquad \Phi(0) = \Phi_0,
\end{equation}
where $(\bU_0, \bM_0, \Phi_0) \in \bbUdiv \times \bbM \times \bbX$ are appropriate discrete approximations of the continuous initial data $\bu_0$, $\bm_0$ and $\phi_0$.
Taking $\bV_\bu \in \bbUdiv$ in \eqref{eq:semidiscrete-scheme} eliminates the discrete pressure $P$ and leads to a reduced system.
\end{subequations}

\begin{remark}[consistency terms and the magnetostatic equation \eqref{eq:maxwell}]
Notice that, unlike the original formulation \eqref{ferrofluids-system}, the equivalent formulation introduced in \cref{sec:equivalent-formulation} does not explicitly impose the magnetostatic equation $\div(\bm+\bh-\bh_a) = 0$.
The addition of the stabilization terms $\kappa \chi\, (\Delta\Phi + \div\bM, \div \bV_\bm)$ and $\kappa(\Delta\Phi + \div\bM, \Delta V_\phi)$ to the formulation in \eqref{all-semidiscrete-scheme} provides $\LL^2(0,T;\LL^2(\Omega))$-control of the residual $\Delta\Phi+\div\bM$, as shown below in \eqref{eq:disc-energy-stability}.
Since $\div\bh_a=0$, this is the residual of the interior magnetostatic equation in \eqref{eq:maxwell}.
This will be essential for the a priori error analysis carried out in \cref{sec:a-priori}.
\end{remark}

\begin{lemma}[local existence and uniqueness for the reduced scheme]\label{lem:existence-discrete-solutions}
Suppose that the spaces in \eqref{eq:fin-dim-subspaces} satisfy \cref{assump:strongly-div-free} and that $(\bU_0,\bM_0,\Phi_0)\in\bbUdiv\times\bbM\times\bbX$.
For every $\bh_a\in\HH^1(0,T;\bL^2(\Omega))$, there exists $0<\widetilde T\le T$ such that the reduced scheme obtained from \eqref{all-semidiscrete-scheme} by taking $\bU,\bV_\bu\in\bbUdiv$ and eliminating $P$ admits a unique solution
\begin{equation*}
(\bU,\bM,\Phi)\in AC([0,\widetilde T];\bbUdiv\times\bbM\times\bbX).
\end{equation*}
\begin{proof}
Restrict $\bV_\bu$ in \eqref{eq:semidiscrete-scheme} to $\bbUdiv$.
The mass matrices associated to $(\partial_t \bU, \bV_\bu)$ and $(\partial_t \bM, \bV_\bm)$ are positive definite, and so is the stiffness matrix associated with $(\nabla\partial_t\Phi,\nabla V_\phi)$, since $\|\nabla\Psi\|=0$ and $\Psi\in\bbX\subset\HH^1_\#(\Omega)$ imply $\Psi=0$.
Set $N:=\dim\bbUdiv+\dim\bbM+\dim\bbX$.
By expanding $(\bU,\bM,\Phi)$ in fixed bases of these spaces and inverting the three aforementioned matrices, the reduced system can be recast as the nonautonomous ODE
\begin{equation}\label{eq:coeff-ODE}
\dot{\bz}(t)=\bf(\bz(t))+\bg(t),
\qquad \bz(0)=\bz_0,
\end{equation}
where $\bz(t)\in\RR^{N}$ is the coefficient vector of $(\bU,\bM,\Phi)$, and $\bz_0$ is that of $(\bU_0,\bM_0,\Phi_0)$.
The function $\bf:\RR^{N}\to\RR^{N}$ is polynomial in the entries of $\bz$, with total degree at most three, while $\bg\in\LL^2(0,T;\RR^{N})$ comes from the forcing $\tau^{-1}\bh_a+\partial_t\bh_a\in\LL^2(0,T;\bL^2(\Omega))$ in the potential equation.
For each $R>0$, there are constants $M_R,L_R<\infty$ such that, for $|\bz|,|\bw|\le R$ and almost every $t$,
\begin{equation*}
|\bf(\bz)+\bg(t)|\le M_R+|\bg(t)|,
\qquad
|\bf(\bz)-\bf(\bw)|\le L_R|\bz-\bw|.
\end{equation*}
Since $\bg\in\LL^2(0,T;\RR^{N})\subset\LL^1(0,T;\RR^{N})$, these bounds verify the hypotheses for local existence and uniqueness in the sense of Carath\'eodory; see \cite[Section~2.3, p.~42]{Teschl2012}.
Consequently, \eqref{eq:coeff-ODE} has a unique solution $\bz\in AC([0,\widetilde T];\RR^{N})$ for some $0<\widetilde T\le T$.
\end{proof}
\end{lemma}

We shall show later that $\widetilde T = T$.
We now define some quantities that will appear in the energy law satisfied by the discrete system:
\begin{subequations}
\begin{align}
\label{eq:energy} \sE[\bU,\bM,\Phi] & := \frac{1}{2} \left(\|\bU\|^2 + \frac{\mu_0}{\chi} \|\bM\|^2 + \mu_0 \|\nabla \Phi\|^2\right),\\
\label{eq:dissipation} \sD[\bU,\bM,\Phi] & := \frac{1}{2}\left( \nu\|\nabla \bU\|^2
+ \frac{\mu_0}{\tau} \int_\Omega Q(\bM, \nabla \Phi)
+ \kappa \mu_0 \|\Delta\Phi + \div\bM\|^2
\right),
\end{align}
\end{subequations}
where $Q$ is the quadratic form defined as
\begin{equation}\label{eq:Q-def}
Q(\ba, \bb)
:= \frac{1}{\chi} |\ba|^2 - \ba \cdot \bb + (1+\chi) |\bb|^2
= \frac{1}{\chi} |\ba - \chi \bb|^2 + (\ba + \bb) \cdot \bb
= \frac{1}{\chi}\left|\ba-\frac{\chi}{2}\bb\right|^2 + \left(1+\frac{3\chi}{4}\right)|\bb|^2,
\qquad \forall \ba, \bb \in \RR^d.
\end{equation}
For every $\chi>0$, this shows that $Q$ is positive definite, that is, there exists $C_{\text{pos}} := C_{\text{pos}}(\chi) > 0$ such that
\begin{equation*}
Q(\ba, \bb) \geq C_{\text{pos}} \left( |\ba|^2 + |\bb|^2  \right).
\end{equation*}
This, in particular, implies that
\begin{equation}\label{eq:simpler-dissipation}
\sD^*[\bU,\bM,\Phi] := \frac{1}{2}\left( \nu\|\nabla \bU\|^2 + \frac{\mu_0 C_{\text{pos}}}{\tau} \|\bM\|^2 + \frac{\mu_0 C_{\text{pos}}}{\tau} \|\nabla\Phi\|^2 + \kappa \mu_0 \|\Delta\Phi + \div\bM\|^2 \right) \leq \sD[\bU,\bM,\Phi].
\end{equation}
Notice that we do not include the nonlinear spin dissipation in the definitions of the quadratic functionals $\sD$ and $\sD^*$.

\begin{lemma}[discrete energy law]\label{lem:disc-energy-law}
Under the hypotheses of \cref{lem:existence-discrete-solutions}, let $[0,\widetilde T]$, with $0<\widetilde T\le T$, be any interval of existence of an absolutely continuous solution $(\bU,\bM,\Phi)$ of the reduced scheme.
Then, for almost every $t\in(0,\widetilde T)$,
\begin{equation}\label{eq:disc-energy-law}
\frac{\dd}{\dd t} \sE[\bU,\bM,\Phi] + 2 \sD[\bU,\bM,\Phi]
+\mu_0\beta\|\bM\times\nabla\Phi\|^2
= \frac{\mu_0}{\tau} (\bh_a, \nabla \Phi) + \mu_0 (\partial_t \bh_a, \nabla \Phi).
\end{equation}
\begin{proof}
Absolute continuity of $\bU$, $\bM$, and $\nabla\Phi$ in the fixed finite-dimensional spatial spaces justifies the chain rule
\begin{equation*}
\frac{1}{2}\frac{\dd}{\dd t}\|\bv(t)\|^2
=(\partial_t\bv(t),\bv(t)),
\qquad \bv\in\{\bU,\bM,\nabla\Phi\},
\quad\text{for a.e. }t\in(0,\widetilde T).
\end{equation*}
Finite dimensionality also allows the equations to be tested with the solution at almost every time.
Choosing $\bV_\bu = \bU$, $\bV_\bm = \frac{\mu_0}{\chi}\bM$ and $V_\phi = \mu_0 \Phi$ in the reduced form of \eqref{eq:semidiscrete-scheme}, and using the skew-symmetry property \eqref{eq:skew-symmetry} of the trilinear form $\cB$  with $\div\bU = 0$, basic properties of the cross product, and the definition \eqref{eq:Q-def} of $Q$, we arrive at
\begin{equation*}
\frac{1}{2}\frac{\dd}{\dd t}\|\bU\|^2+\nu\|\nabla\bU\|^2
=\mu_0\cB(\bM,\nabla\Phi,\bU)+\frac{\mu_0}{2}(\bM\times\nabla\Phi,\bcurl\bU),
\end{equation*}
\begin{equation*}
\frac{\mu_0}{2\chi}\frac{\dd}{\dd t}\|\bM\|^2
+\kappa\mu_0(\Delta\Phi+\div\bM,\div\bM)
+\frac{\mu_0}{\chi\tau}(\bM-\chi\nabla\Phi,\bM)=0,
\end{equation*}
\begin{multline*}
\frac{\mu_0}{2}\frac{\dd}{\dd t}\|\nabla\Phi\|^2
+\mu_0\frac{1+\chi}{\tau}\|\nabla\Phi\|^2
+\kappa\mu_0(\Delta\Phi+\div\bM,\Delta\Phi)
+\mu_0\beta\|\bM\times\nabla\Phi\|^2\\
=\frac{\mu_0}{\tau}(\bh_a,\nabla\Phi)+\mu_0(\partial_t\bh_a,\nabla\Phi)
+\mu_0\cB(\bU,\bM,\nabla\Phi)-\frac{\mu_0}{2}(\bcurl\bU,\bM\times\nabla\Phi).
\end{multline*}
The nonlinear relaxation vanishes in the magnetization identity because $\bM\times\bM=0$ and contributes $\mu_0\beta\|\bM\times\nabla\Phi\|^2$ in the potential identity, on account of \eqref{eq:spin-identities}.
Notice that $\cB(\bM,\nabla\Phi,\bU) = \cB(\bU,\nabla\Phi,\bM)$ and $\cB(\bU,\bM,\nabla\Phi) = -\cB(\bU,\nabla\Phi,\bM)$ according to \eqref{eq:symmetry} and \eqref{eq:skew-symmetry} because $\nabla^2\Phi$ is symmetric and $\div \bU = 0$.
Then,
by summing these last three identities and performing simple algebraic manipulations we arrive at \eqref{eq:disc-energy-law}.
We omit further details.
\end{proof}
\end{lemma}
\begin{remark}[consistency with the formal energy law]
The discrete energy law \eqref{eq:disc-energy-law} is consistent with the formal energy law \eqref{eq:continuous-energy law}.
In fact, a sufficiently smooth solution $(\bu,p,\bm,\phi)$ of \eqref{ferrofluids-system} satisfies $\Delta\phi+\div\bm=0$ and $(\bm+\nabla\phi,\nabla\phi)=(\bh_a,\nabla\phi)$, whence
\begin{equation*}
\frac{\mu_0}{\tau} \int_\Omega Q(\bm, \nabla \phi) - \frac{\mu_0}{\tau} \int_\Omega \bh_a \cdot \nabla\phi
= \frac{\mu_0}{\chi\tau}\|\bm-\chi\nabla\phi\|^2.
\end{equation*}
\end{remark}
\begin{corollary}[discrete energy stability]\label{cor:disc-energy-stability}
Under the hypotheses of \cref{lem:disc-energy-law}, the solution $(\bU,\bM,\Phi)$ satisfies, for every $t\in[0,\widetilde T]$,
\begin{equation}\label{eq:disc-energy-stability}
\sE[\bU,\bM,\Phi](t) + \int_0^t \sD^*[\bU,\bM,\Phi] \dd s
+\mu_0\beta\int_0^t\|\bM\times\nabla\Phi\|^2\dd s
\leq \sE[\bU_0,\bM_0,\Phi_0]
+ \frac{\tau \mu_0}{2 C_{\text{pos}}} \int_0^t \left\|\frac{1}{\tau} \bh_a +  \partial_t \bh_a\right\|^2 \dd s.
\end{equation}
\begin{proof}
We estimate the right-hand side of the discrete energy law \eqref{eq:disc-energy-law} via Young's inequality:
\begin{equation*}
\left(\frac{\mu_0}{\tau} \bh_a + \mu_0 \partial_t \bh_a, \nabla \Phi\right)
\leq \frac{\tau}{2\mu_0 C_{\text{pos}}} \left\|\frac{\mu_0}{\tau} \bh_a + \mu_0 \partial_t \bh_a\right\|^2 + \frac{\mu_0 C_{\text{pos}}}{2\tau}\|\nabla\Phi\|^2,
\end{equation*}
whose second term can be absorbed into the left-hand side of \eqref{eq:disc-energy-law} on account of \eqref{eq:simpler-dissipation}, thus yielding
\begin{equation*}
\frac{\dd}{\dd t} \sE[\bU,\bM,\Phi] + \sD^*[\bU,\bM,\Phi]
+\mu_0\beta\|\bM\times\nabla\Phi\|^2
\leq \frac{\tau}{2\mu_0 C_{\text{pos}}} \left\|\frac{\mu_0}{\tau} \bh_a + \mu_0 \partial_t \bh_a\right\|^2.
\end{equation*}
The function $t\mapsto\sE[\bU,\bM,\Phi](t)$ is absolutely continuous, so integration over $(0,t)$ gives \eqref{eq:disc-energy-stability} for every $t\in[0,\widetilde T]$.
\end{proof}
\end{corollary}

\begin{remark}[pressure recovery]\label{rem:pressure-recovery}
Let $\bbP'$ be the dual of $\bbP$ equipped with the $\LL^2$-norm, and define the discrete divergence operator $B:\bbU\to\bbP'$ by
\begin{equation*}
\langle B\bV,Q\rangle_{\bbP'\times\bbP}:=(Q,\div\bV),
\qquad \bV\in\bbU,\quad Q\in\bbP.
\end{equation*}
Its transpose $B^{\mathrm t}:\bbP\to\bbU'$ is characterized by $\langle B^{\mathrm t}Q,\bV\rangle=(Q,\div\bV)$.
The pressure is determined by the momentum equation modulo
\begin{equation}\label{eq:discrete-pressure-nullspace}
\bbN:=\ker B^{\mathrm t}
=\{Q\in\bbP:(Q,\div\bV)=0\text{ for every }\bV\in\bbU\}.
\end{equation}
Since $\bbP$ is finite dimensional, the uniqueness of the recovered pressure is equivalent to $\bbN=\{0\}$, or equivalently to surjectivity of $B$.
\end{remark}

\begin{proposition}[existence and uniqueness]\label{prop:semidiscrete-solvability}
Suppose that the spaces in \eqref{eq:fin-dim-subspaces} satisfy \cref{assump:strongly-div-free} and that $(\bU_0,\bM_0,\Phi_0)\in\bbUdiv\times\bbM\times\bbX$.
For every $\bh_a\in\HH^1(0,T;\bL^2(\Omega))$, the scheme in \eqref{all-semidiscrete-scheme} has a unique solution
\begin{equation*}
(\bU,\bM,\Phi)\in\HH^1(0,T;\bbUdiv\times\bbM\times\bbX),
\end{equation*}
and admits a pressure $P\in\LL^2(0,T;\bbP)$ that is unique modulo $\LL^2(0,T;\bbN)$ (cf.~\eqref{eq:discrete-pressure-nullspace}).
The solution satisfies the energy bound in \eqref{eq:disc-energy-stability}.
\begin{proof}
We continue from the proof of \cref{lem:existence-discrete-solutions}.
Let $T_{\max}\in(0,T]$ be the supremum of the times to which $\bz$ can be continued within the prescribed interval $[0,T]$.
Applying \eqref{eq:disc-energy-stability} on $[0,\widetilde T]$ for arbitrary $0<\widetilde T<T_{\max}$, and using positive definiteness of the energy and equivalence of norms in the fixed finite-dimensional space, yields
\begin{equation}\label{eq:coeff-energy-bound}
\sup_{0\le t<T_{\max}}|\bz(t)|\le R_h<\infty,
\end{equation}
where $R_h$ may depend on the mesh, the initial data, and the forcing norm over $(0,T)$, but not on $\widetilde T$.
The continuation criterion for Carath\'eodory equations \cite[Theorem~10]{Prazak2024} therefore applies.
To make this step explicit, set $M_{R_h}:=\max_{|\by|\le R_h}|\bf(\by)|$ and integrate \eqref{eq:coeff-ODE} to obtain, for $0\le s<t<T_{\max}$,
\begin{equation}\label{eq:coeff-time-continuity}
|\bz(t)-\bz(s)|
\le M_{R_h}(t-s)+\int_s^t|\bg(r)|\,\dd r
\le M_{R_h}(t-s)+\|\bg\|_{\LL^2(0,T;\RR^{N})}(t-s)^{1/2}.
\end{equation}
Thus $\bz(t)$ has a finite limit as $t\uparrow T_{\max}$.
If $T_{\max}<T$, local existence starting from this limit extends $\bz$ beyond $T_{\max}$, a contradiction.
Hence $T_{\max}=T$, and \eqref{eq:coeff-time-continuity} defines the continuous extension to $t=T$.

By \eqref{eq:coeff-energy-bound} and recalling that $T_{\max} = T$, both $\bz$ and $\bf(\bz)$ belong to $\LL^\infty(0,T;\RR^{N})$.
Equation \eqref{eq:coeff-ODE} therefore yields
\begin{equation}\label{eq:coeff-time-derivative-bound}
\|\dot{\bz}\|_{\LL^2(0,T;\RR^{N})}
\le T^{1/2}M_{R_h}+\|\bg\|_{\LL^2(0,T;\RR^{N})}<\infty.
\end{equation}
The integral form of \eqref{eq:coeff-ODE} then gives $\bz\in AC([0,T];\RR^{N})$, with its almost-everywhere derivative also its weak derivative, so $\bz\in\HH^1(0,T;\RR^{N})$.
Equivalently, by expanding on the aforementioned finite-dimensional bases, $(\bU,\bM,\Phi)\in\HH^1(0,T;\bbUdiv\times\bbM\times\bbX)$.

For almost every $t$, let $\ell(t)\in\bbU'$ be the momentum residual obtained by omitting the pressure term in \eqref{eq:semidiscrete-scheme}.
The term containing $\partial_t\bU$ is square integrable in time, and the remaining coefficients are bounded by \eqref{eq:coeff-energy-bound}, so $\ell\in\LL^2(0,T;\bbU')$.
The reduced momentum equation says that $\ell(t)$ vanishes on $\bbUdiv=\ker B$.
Finite-dimensional linear algebra identifies the range of $B^{\mathrm t}$ with the space of annihilator functionals of $\ker B$, whence there is $P(t)\in\bbP$ with $B^{\mathrm t}P(t)=\ell(t)$.
Choosing $P(t)$ in the $\LL^2$-orthogonal complement of $\bbN$ gives a fixed linear recovery map and hence $P\in\LL^2(0,T;\bbP)$.
Any other pressure differs by an element of $\LL^2(0,T;\bbN)$.
Conversely, every solution of the full scheme restricts to the same reduced system, which proves uniqueness of $(\bU,\bM,\Phi)$.
\end{proof}
\end{proposition}
\section{A priori error analysis}\label{sec:a-priori}
In this section, we carry out an a priori error analysis for our abstract discrete scheme \eqref{all-semidiscrete-scheme} under precise but general assumptions.
Let $(\bu,p,\bm,\phi)$ be a smooth weak solution corresponding to a smooth applied magnetic field $\bh_a$ and initial data $\bu_0$, $\bm_0$ and $\phi_0$, with $\phi_0$ specified by the compatibility condition \eqref{eq:initial-Maxwell-compatibility}.

As is customary, we introduce suitable projections $\osfU := (\obU,\obM,\oPhi)$ of $\sfu := (\bu,\bm,\phi)$  with values in $\bbU \times \bbM \times \bbX$, and derive the nonlinear equations satisfied by the discrete error $\sfE := (\bU-\obU, \bM-\obM, \Phi-\oPhi)$.
These equations are similar to \eqref{eq:semidiscrete-scheme} except that their right-hand sides contain consistency terms and byproducts of the underlying nonlinear interactions.
We next apply \cref{cor:disc-energy-stability} (\nameref*{cor:disc-energy-stability}) to $\sfE$ with the corresponding right-hand side.

We start by making abstract assumptions on the projections.
We will make them concrete, along with specific choices of spaces $(\bbU,\bbP,\bbM,\bbX)$ in \cref{sec:specific-FES}.

\begin{assump}\label{assump:proj-time-commute}
Let $\cP_h : \bH^1_0(\Omega) \times \bH^1(\Omega) \times \HH^2_\#(\Omega) \rightarrow \bbU \times \bbM \times \bbX$ denote the \emph{global auxiliary projection}.
We require that $\cP_h$ is linear, bounded and independent of time.
Writing $\osfU(t)=\mathcal P_h\sfu(t)$ and provided $\sfu$ is sufficiently regular in time (to be specified later), differentiation in time gives
\begin{equation}\label{eq:projection-time-commutation}
\partial_t\osfU=\mathcal P_h(\partial_t\sfu).
\end{equation}
\end{assump}
\begin{assump}[strongly divergence-free projection]\label{assump:vel-aux-div-free}
We assume that for each $t \in (0,T]$,  $\obU(t) \in \bbUdiv$.
\end{assump}

We respectively define the Galerkin, projection, and discrete errors, as follows:
\begin{align*}
\sfe & := (\be_\bu, \be_\bm, e_\phi) := (\bu-\bU, \bm - \bM, \phi - \Phi),\\
\sfe^p & := (\be_\bu^p, \be_\bm^p, e_\phi^p) := (\bu-\obU, \bm - \obM, \phi - \oPhi),\\
\sfE & := (\bE_\bu, \bE_\bm, E_\phi) := (\bU-\obU, \bM - \obM, \Phi - \oPhi).
\end{align*}
Notice that the time-dependent discrete error $\sfE$ takes values in $\bbU_{\div} \times \bbM \times \bbX$.
Moreover, the three error quantities above are related via the identity
\begin{equation*}
\sfe = \sfe^p - \sfE.
\end{equation*}
We now introduce some nonlinear operators associated to the ferrofluid system and its discretization \eqref{all-semidiscrete-scheme} that will be useful in the subsequent analysis.
For a sufficiently regular triple $\sfv := (\bsu,\bsm,\upphi)$, and discrete test functions $(\bV_\bu, \bV_\bm, V_\phi) \in \bbU_{\div} \times \bbM \times \bbX$, we define
\begin{subequations}
\begin{equation}\label{def:op-U}
\langle \sM_\bu \sfv, \bV_\bu \rangle:=
(\partial_t \bsu,\bV_\bu) + \nu(\nabla \bsu, \nabla \bV_\bu) + \cB(\bsu,\bsu,\bV_\bu) - \mu_0 \cB(\bsm,\nabla\upphi,\bV_\bu) - \frac{\mu_0}{2} (\bsm \times \nabla\upphi, \bcurl \bV_\bu),
\end{equation}
\begin{equation}\label{def:op-m}
\langle \sM_\bm \sfv, \bV_\bm \rangle :=
(\partial_t \bsm, \bV_\bm) + \cB(\bsu, \bsm, \bV_\bm) - \frac{1}{2} (\bcurl \bsu, \bsm \times \bV_\bm) + \frac{1}{\tau} (\bsm - \chi \nabla\upphi, \bV_\bm)
 - \beta(\bsm \times \nabla\upphi,\bsm\times\bV_\bm),
\end{equation}
\begin{equation}\label{def:op-phi}
\langle \sM_\phi \sfv, V_\phi \rangle :=
(\nabla \partial_t \upphi, \nabla V_\phi) + \frac{1+\chi}{\tau}(\nabla \upphi, \nabla V_\phi) - \cB(\bsu,\bsm, \nabla V_\phi) + \frac{1}{2}(\bcurl \bsu, \bsm \times \nabla V_\phi)
 + \beta(\bsm\times\nabla\upphi,\bsm\times\nabla V_\phi),
\end{equation}
\begin{equation}\label{def:op-forcing-phi}
\langle F_\phi, V_\phi \rangle := \frac{1}{\tau} (\bh_a, \nabla V_\phi) + (\partial_t \bh_a, \nabla V_\phi).
\end{equation}
\end{subequations}
Since the operators $\sM_\bu$, $\sM_\bm$ and $\sM_\phi$ are those appearing in the ferrofluids system \eqref{eq:semidiscrete-scheme}, a sufficiently smooth weak solution $\sfu$ of the equivalent formulation introduced in \cref{sec:equivalent-formulation} satisfies
\begin{equation}\label{eq:nonlinops-ex-sol}
\begin{aligned}
\langle \sM_\bu \sfu, \bV_\bu \rangle & = 0,\\
\langle \sM_\bm \sfu, \bV_\bm \rangle & = 0,\\
\langle \sM_\phi \sfu, V_\phi \rangle & = \langle F_\phi, V_\phi \rangle,\\
\Delta\phi + \div\bm & = 0.
\end{aligned}
\end{equation}
In contrast, the Galerkin solution $\sfU := (\bU,\bM,\Phi)$ satisfies
\begin{equation}\label{eq:nonlinops-disc-sol}
\begin{aligned}
\langle \sM_\bu \sfU, \bV_\bu \rangle & = 0,\\
\langle \sM_\bm \sfU, \bV_\bm \rangle +\kappa \chi\, (\Delta\Phi + \div\bM, \div \bV_\bm) & = 0,\\
\langle \sM_\phi \sfU, V_\phi \rangle +\kappa (\Delta\Phi + \div\bM, \Delta V_\phi)& = \langle F_\phi, V_\phi \rangle.
\end{aligned}
\end{equation}

\subsection{Discrete error equation}\label{sec:discrete-error-equation}

The following quantity will prove useful when handling the cubic nonlinear magnetic relaxation.
For $\ba,\bb\in\RR^d$, define
\begin{equation}\label{eq:cubic-nonlinearity}
\bG(\ba,\bb):=(\ba\times\bb)\times\ba
\end{equation}
and
\begin{equation}\label{eq:spin-mixed-remainder}
\bK_{\bM,\Phi}:=\bG(\bM,\nabla\Phi)-\bG(\obM,\nabla\oPhi)-\bG(\bE_\bm,\nabla E_\phi).
\end{equation}
\begin{lemma}[error equation]
The discrete error $\sfE := (\bE_\bu, \bE_\bm, E_\phi) \in \bbU_{\div} \times \bbM \times \bbX$ satisfies the following semidiscrete variational problem
\begin{equation}\label{eq:error-equation}
\begin{aligned}
\langle \sM_\bu \sfE, \bV_\bu \rangle & = \langle \bR_\bu, \bV_\bu \rangle,\\
\langle \sM_\bm \sfE, \bV_\bm \rangle + \kappa \chi\, (\Delta E_\phi + \div\bE_\bm, \div \bV_\bm) & = \langle \bR_\bm, \bV_\bm \rangle,\\
\langle \sM_\phi \sfE, V_\phi \rangle +\kappa (\Delta E_\phi + \div\bE_\bm, \Delta V_\phi) & = \langle R_\phi, V_\phi \rangle,
\end{aligned}
\end{equation}
for all discrete test functions $(\bV_\bu, \bV_\bm, V_\phi) \in \bbU_{\div} \times \bbM \times \bbX$.
The residual operators $\bR_\bu$, $\bR_\bm$ and $R_\phi$ are given by
\begin{subequations}\label{all-residuals}
\begin{align}
\label{eq:residual-u}
\langle \bR_\bu, \bV_\bu \rangle & := \langle \bN_\bU, \bV_\bu \rangle - \langle \sM_\bu \osfU, \bV_\bu \rangle,\\
\label{eq:residual-m}
\langle \bR_\bm, \bV_\bm \rangle & := \langle \bN_\bM, \bV_\bm \rangle - \langle \sM_\bm \osfU, \bV_\bm \rangle - \kappa\chi\, (\Delta\oPhi + \div\obM, \div \bV_\bm), \\
\label{eq:residual-phi}
\langle R_\phi, V_\phi \rangle & := \langle N_\Phi, V_\phi \rangle - \langle \sM_\phi \osfU - F_\phi, V_\phi \rangle -\kappa(\Delta\oPhi + \div\obM, \Delta V_\phi),
\end{align}
\end{subequations}
where, in turn, $\bN_\bU$, $\bN_\bM$ and $N_\Phi$ account for the nonlinear nature of the system, and are given by
\begin{subequations}\label{all-nonlinearities}
\begin{multline}
\label{eq:nonlinearity-u} \langle \bN_\bU, \bV_\bu \rangle := - \cB(\bE_\bu, \obU, \bV_\bu) - \cB(\obU, \bE_\bu, \bV_\bu) + \mu_0 \cB(\bE_\bm, \nabla\oPhi, \bV_\bu) + \mu_0 \cB(\obM, \nabla E_\phi, \bV_\bu)\\ \qquad + \frac{\mu_0}{2}(\bE_\bm \times \nabla\oPhi, \bcurl \bV_\bu) + \frac{\mu_0}{2}(\obM \times \nabla E_\phi, \bcurl \bV_\bu),
\end{multline}
\begin{equation}\label{eq:nonlinearity-m}
\begin{aligned}
\langle \bN_\bM, \bV_\bm \rangle
&:= - \cB(\bE_\bu, \obM, \bV_\bm) - \cB(\obU, \bE_\bm, \bV_\bm)\\
&\quad + \frac{1}{2}(\bcurl \bE_\bu, \obM \times \bV_\bm)
+ \frac{1}{2}(\bcurl \obU, \bE_\bm \times \bV_\bm)
+ \beta(\bK_{\bM,\Phi},\bV_\bm),
\end{aligned}
\end{equation}
\begin{equation}\label{eq:nonlinearity-phi}
\begin{aligned}
\langle N_\Phi, V_\phi \rangle
&:= \cB(\bE_\bu,\obM,\nabla V_\phi) + \cB(\obU,\bE_\bm,\nabla V_\phi)\\
&\quad - \frac{1}{2}(\bcurl \bE_\bu, \obM \times \nabla V_\phi)
- \frac{1}{2}(\bcurl \obU, \bE_\bm \times \nabla V_\phi)
- \beta(\bK_{\bM,\Phi},\nabla V_\phi).
\end{aligned}
\end{equation}
\end{subequations}
\begin{proof}
To handle the quadratic nonlinearities in \eqref{eq:semidiscrete-scheme}, let us first consider two generic Galerkin quantities $X$ and $Y$ and projections $\overline{X}$ and $\overline{Y}$, and observe that
\begin{equation*}
(X-\overline{X})(Y-\overline{Y})
= XY - \overline{X}\,\overline{Y} - (X-\overline{X})\overline{Y} - \overline{X}(Y - \overline{Y}),
\end{equation*}
where the products are between scalar-, vector-, or tensor-valued fields depending on the context.
We next apply this identity to $\bE_\bu = \bU - \obU$, $\bE_\bm = \bM - \obM$ and $E_\phi = \Phi - \oPhi$ to obtain
\begin{gather*}
\cB(\bE_\bu,\bE_\bm,\bV_\bm)
= \cB(\bU,\bM,\bV_\bm) - \cB(\obU,\obM,\bV_\bm) - \cB(\bE_\bu,\obM,\bV_\bm) - \cB(\obU,\bE_\bm,\bV_\bm)\\
\bE_\bm\times\nabla E_\phi
= \bM\times\nabla\Phi - \obM\times\nabla\oPhi - \bE_\bm\times\nabla\oPhi - \obM\times\nabla E_\phi.
\end{gather*}
Similar expansions apply for the velocity convection, the Kelvin force  and the term involving $(\bcurl\bU)\times\bM$.
In contrast, the time derivatives, viscous term, and linear relaxation terms are linear in $\osfU$ and $\sfE$.

For the nonlinear relaxation, we have that for every $\bv\in\bL^2(\Omega)$, combining \eqref{eq:spin-identities} and \eqref{eq:spin-mixed-remainder} yields
\begin{equation}\label{eq:spin-error-decomposition}
(\bE_\bm\times\nabla E_\phi,\bE_\bm\times\bv) = (\bM\times\nabla\Phi,\bM\times\bv) - (\obM\times\nabla\oPhi,\obM\times\bv) - (\bK_{\bM,\Phi},\bv).
\end{equation}
With $\bv=\bV_\bm$ and $\bv=\nabla V_\phi$, identity \eqref{eq:spin-error-decomposition} contributes $+\beta(\bK_{\bM,\Phi},\bV_\bm)$ and $-\beta(\bK_{\bM,\Phi},\nabla V_\phi)$ to $\langle\bN_\bM,\bV_\bm\rangle$ and $\langle N_\Phi,V_\phi\rangle$, respectively.
The pure-error cubic contribution is already included in the definitions of $\sM_\bm\sfE$ and $\sM_\phi\sfE$.
Collecting all mixed terms according to \eqref{all-nonlinearities}, we therefore obtain
\begin{equation}\label{eq:operator-error-decomposition}
\begin{aligned}
\langle\sM_\bu\sfE,\bV_\bu\rangle & = \langle\sM_\bu\sfU,\bV_\bu\rangle - \langle\sM_\bu\osfU,\bV_\bu\rangle + \langle\bN_\bU,\bV_\bu\rangle,\\
\langle\sM_\bm\sfE,\bV_\bm\rangle & = \langle\sM_\bm\sfU,\bV_\bm\rangle - \langle\sM_\bm\osfU,\bV_\bm\rangle + \langle\bN_\bM,\bV_\bm\rangle,\\
\langle\sM_\phi\sfE,V_\phi\rangle & = \langle\sM_\phi\sfU,V_\phi\rangle - \langle\sM_\phi\osfU,V_\phi\rangle + \langle N_\Phi,V_\phi\rangle.
\end{aligned}
\end{equation}
Regarding the stabilization terms, since $\Delta \phi + \div \bm = 0$, we have that
\begin{equation}\label{eq:stabilization-error-decomposition}
(\Delta E_\phi + \div \bE_\bm) = \Delta\Phi+\div\bM + (\Delta e_\phi^p + \div \be_\bm^p).
\end{equation}
Substituting \eqref{eq:operator-error-decomposition} and \eqref{eq:stabilization-error-decomposition} into the semidiscrete scheme \eqref{eq:semidiscrete-scheme} yields the error equation \eqref{eq:error-equation}.
This finishes the proof.
\end{proof}
\end{lemma}

\begin{corollary}[energy law for the discrete error]
The discrete error $\sfE$ satisfies
\begin{equation}\label{eq:error-energy-law}
\frac{\dd}{\dd t} \sE[\bE_\bu,\bE_\bm,E_\phi] + 2\sD[\bE_\bu,\bE_\bm,E_\phi]
+\mu_0\beta\|\bE_\bm \times \nabla E_\phi\|^2
= 
\langle \bR_\bu, \bE_\bu \rangle + \frac{\mu_0}{\chi} \langle \bR_\bm, \bE_\bm \rangle + \mu_0 \langle R_\phi, E_\phi \rangle.
\end{equation}
\begin{proof}
Since \eqref{eq:error-equation} has the same structure as \eqref{eq:semidiscrete-scheme}, it suffices to choose $\bV_\bu = \bE_\bu$, $\bV_\bm = \frac{\mu_0}{\chi}\bE_\bm$ and $V_\phi = \mu_0 E_\phi$ in \eqref{eq:error-equation} and proceed as in the proof of \cref{lem:disc-energy-law}.
Further details are omitted.
\end{proof}
\end{corollary}
We let $\bbUdiv'$, $\widetilde\bbM'$ and $\widetilde\bbX'$ be the dual spaces of $\bbUdiv$, $\bbM$ and $\bbX$ induced by the norms $\bV_\bu \mapsto \|\nabla\bV_\bu\|$, $\bV_\bm \mapsto \|\bV_\bm\|$ and $V_\phi \mapsto \|\nabla V_\phi\|$, respectively.
Since $\bbM \subseteq \bH^1(\Omega)$ and $\bbX \subseteq \HH^2(\Omega) \cap \HH^1_\#(\Omega)$, we write $\widetilde\bbM'$ and $\widetilde\bbX'$ to distinguish them from the dual spaces of $\bbM$ and $\bbX$ with respect to the more natural $\bH^1$- and $\HH^2$-norms.

We now establish the discrete stability of the discrete error.
Before that, however, we need the following lemma that will be useful for handling the cubic nonlinearity $\bK_{\bM,\Phi}$.

\begin{lemma}[algebraic estimate for the spin remainder]\label{lem:spin-remainder}
The term $\bK_{\bM,\Phi}$ defined in \eqref{eq:spin-mixed-remainder} satisfies
\begin{multline}\label{eq:spin-remainder-bound}
\frac{1}{\chi}(\bK_{\bM,\Phi},\bE_\bm-\chi\nabla E_\phi)
-\|\bE_\bm\times\nabla E_\phi\|^2
+\frac{1}{2}\|\bM\times\nabla\Phi-\obM\times\nabla\oPhi\|^2\\
\leq\frac{1}{2\chi^2}\left\|\obM - \chi\nabla\oPhi\right\|_{0,\infty}^2\|\bE_\bm\|^2
+\frac{1}{2}\|\obM\times\nabla\oPhi\|_{0,\infty}
\left(\|\bE_\bm\|^2+\|\nabla E_\phi\|^2\right).
\end{multline}
\begin{proof}
Let $\bW:=\bE_\bm-\chi\nabla E_\phi$.
We have that $(\bG(\bE_\bm,\nabla E_\phi),\bW)=-\chi\|\bE_\bm\times\nabla E_\phi\|^2$ because of identity \eqref{eq:spin-identities}, and consequently, using the definition of \eqref{eq:spin-mixed-remainder} we have that
\begin{equation}\label{eq:spin-pure-error-cancellation}
\frac{1}{\chi}(\bK_{\bM,\Phi},\bW)
=\frac{1}{\chi}(\bG(\bM,\nabla\Phi)-\bG(\obM,\nabla\oPhi),\bW)
+\|\bE_\bm\times\nabla E_\phi\|^2.
\end{equation}
We now manipulate $\bG(\bM,\nabla\Phi)-\bG(\obM,\nabla\oPhi)$.
Recalling definition \eqref{eq:cubic-nonlinearity} of $\bG$, by adding and subtracting $(\obM\times\nabla\oPhi)\times\bM$ we have
\begin{equation}\label{eq:spin-cubic-difference}
\bG(\bM,\nabla\Phi)-\bG(\obM,\nabla\oPhi)
=(\bM\times\nabla\Phi-\obM\times\nabla\oPhi)\times\bM
+(\obM\times\nabla\oPhi)\times\bE_\bm,
\end{equation}
and so \eqref{eq:spin-pure-error-cancellation} reduces to
\begin{equation}\label{eq:spin-pure-error-cancellation-rewritten}
\frac{1}{\chi}(\bK_{\bM,\Phi},\bW)
=\frac{1}{\chi}(\bM\times\nabla\Phi-\obM\times\nabla\oPhi,\bM \times \bW)
+ \frac{1}{\chi} (\obM\times\nabla\oPhi, \bE_\bm \times \bW)
+\|\bE_\bm\times\nabla E_\phi\|^2,
\end{equation}
where we have used the triple product identity \eqref{eq:triple-product-identity}.

In turn, using $\bM=\obM+\bE_\bm$ and $\nabla\Phi=\nabla\oPhi+\nabla E_\phi$ yields
\begin{equation}\label{eq:spin-weighted-cross-product}
\bM\times\bW
=-\chi(\bM\times\nabla\Phi-\obM\times\nabla\oPhi)
- \bE_\bm\times\left(\obM - \chi \nabla\oPhi\right),
\end{equation}
and $\bE_\bm\times\bW=-\chi\bE_\bm\times\nabla E_\phi$.
Plugging these two identities back in \eqref{eq:spin-pure-error-cancellation-rewritten} and performing some algebraic simplifications, we thus get
\begin{multline}\label{eq:spin-signed-remainder}
\frac{1}{\chi}(\bK_{\bM,\Phi},\bW)
= \|\bE_\bm\times\nabla E_\phi\|^2 - \|\bM\times\nabla\Phi-\obM\times\nabla\oPhi\|^2\\
- \frac{1}{\chi}(\bM\times\nabla\Phi-\obM\times\nabla\oPhi,\bE_\bm\times\left(\obM - \chi \nabla\oPhi\right))
-(\obM\times\nabla\oPhi, \bE_\bm\times\nabla E_\phi)
.
\end{multline}
We now bound the two scalar products in \eqref{eq:spin-signed-remainder} using H\"older's and Young's inequalities.
For the first one, we obtain
\begin{multline}\label{eq:spin-mixed-product-bound}
\frac{1}{\chi}\left|(\bM\times\nabla\Phi-\obM\times\nabla\oPhi,\bE_\bm\times(\obM-\chi\nabla\oPhi))\right|
\leq\frac{1}{\chi}\|\bM\times\nabla\Phi-\obM\times\nabla\oPhi\|
\|\obM-\chi\nabla\oPhi\|_{0,\infty}\|\bE_\bm\|\\
\leq\frac{1}{2}\|\bM\times\nabla\Phi-\obM\times\nabla\oPhi\|^2
+\frac{1}{2\chi^2}\|\obM-\chi\nabla\oPhi\|_{0,\infty}^2\|\bE_\bm\|^2.
\end{multline}
For the second one, we have
\begin{equation}\label{eq:spin-projection-product-bound}
\left|(\obM\times\nabla\oPhi,\bE_\bm\times\nabla E_\phi)\right|
\leq\|\obM\times\nabla\oPhi\|_{0,\infty}\|\bE_\bm\|\|\nabla E_\phi\|
\leq\frac{1}{2}\|\obM\times\nabla\oPhi\|_{0,\infty}
\left(\|\bE_\bm\|^2+\|\nabla E_\phi\|^2\right).
\end{equation}
Substituting \eqref{eq:spin-mixed-product-bound} and \eqref{eq:spin-projection-product-bound} into \eqref{eq:spin-signed-remainder}, collecting the squared norms, and recalling that $\bW=\bE_\bm-\chi\nabla E_\phi$ yields \eqref{eq:spin-remainder-bound}.
\end{proof}
\end{lemma}

\begin{lemma}[energy stability of the discrete error]\label{lem:stability-discrete-error}
The discrete error $\sfE$ satisfies the following bound:
\begin{multline}\label{eq:stability-discrete-error}
\sE[\bE_\bu,\bE_\bm,E_\phi](t) + \int_0^t \sD^*[\bE_\bu,\bE_\bm,E_\phi] \dd s
+\frac{\mu_0\beta}{2}\int_0^t\|\bM\times\nabla\Phi-\obM\times\nabla\oPhi\|^2\dd s\\
\leq  \Big(\sE_0 + \sG(t) \Big) \exp\left(\|\eta\|_{\LL^1([0,t])} \right),
\end{multline}
where
\begin{align}
\sG(t) & := \int_0^t  \Big(\frac{1}{\nu} \|\sM_\bu \osfU\|_{\bbUdiv'}^2 + \frac{\mu_0 \tau}{2C_{\text{pos}}\chi^2} \|\sM_\bm \osfU\|_{\widetilde\bbM'}^2 + \frac{\mu_0 \tau}{2C_{\text{pos}}} \|\sM_\phi \osfU - F_\phi\|_{\widetilde\bbX'}^2 + \kappa \mu_0\|\Delta\oPhi + \div\obM\|^2\Big) \dd s, \label{eq:consistency-terms}\\
\sE_0 & := \sE[\bE_\bu,\bE_\bm,E_\phi](0), \label{eq:initial-error}
\end{align}
and $\eta := \eta(\|\obU\|_{1,\infty}, \|\obM\|_{1,\infty}, \|\oPhi\|_{2,\infty}, \kappa, \mu_0, \chi, \tau, \beta, \nu, C_{\text{pos}}) > 0$ is an integrable time-dependent function, increasing in its first three arguments.
\begin{proof}
According to the definition \eqref{all-residuals} of the residuals, the right-hand side of \eqref{eq:error-energy-law} reads
\begin{equation}\label{eq:proto-manipulating-nonlinearities}
\begin{aligned}
\langle \bR_\bu, \bE_\bu \rangle + \frac{\mu_0}{\chi} \langle \bR_\bm, \bE_\bm \rangle + \mu_0 \langle R_\phi, E_\phi \rangle
& = - \langle \sM_\bu \osfU, \bE_\bu \rangle - \frac{\mu_0}{\chi} \langle \sM_\bm \osfU, \bE_\bm \rangle - \mu_0 \langle \sM_\phi \osfU - F_\phi, E_\phi \rangle\\
& \quad + \langle \bN_\bU, \bE_\bu \rangle + \frac{\mu_0}{\chi} \langle \bN_\bM, \bE_\bm \rangle + \mu_0 \langle N_\Phi, E_\phi \rangle\\
& \quad - \kappa \mu_0\, (\Delta\oPhi + \div\obM, \div \bE_\bm + \Delta E_\phi).
\end{aligned}
\end{equation}
We now manipulate the terms involving $\bN_\bU$, $\bN_\bM$ and $N_\Phi$, by passing as many derivatives as possible to the auxiliary projection $\osfU$.
First, because $\div \obU = 0$ a.e.~on $\Omega$, \eqref{eq:nonlinearity-u} and \eqref{eq:skew-symmetry} imply that
\begin{multline}
\langle \bN_\bU, \bE_\bu \rangle
= - \cB(\bE_\bu, \obU, \bE_\bu) + \mu_0 \cB(\bE_\bm, \nabla\oPhi, \bE_\bu) + \mu_0 \cB(\obM, \nabla E_\phi, \bE_\bu)\\ \qquad + \frac{\mu_0}{2}(\bE_\bm \times \nabla\oPhi, \bcurl \bE_\bu) + \frac{\mu_0}{2}(\obM \times \nabla E_\phi, \bcurl \bE_\bu).
\end{multline}
In turn, recalling definition \eqref{eq:nonlinearity-m}, identity \eqref{eq:skew-symmetry} with $\div \obU = 0$ implies that
\begin{equation*}
\langle \bN_\bM, \bE_\bm \rangle = - \cB(\bE_\bu, \obM, \bE_\bm) + \frac{1}{2}(\bcurl \bE_\bu, \obM \times \bE_\bm) +\beta(\bK_{\bM,\Phi},\bE_\bm).
\end{equation*}
Finally, using \eqref{eq:skew-symmetry}, the symmetry of $\nabla^2 E_\phi$ and performing an integration by parts that accounts for the possibly nonzero $\div (\bE_\bm + \nabla E_\phi)$ term, we have that
\begin{multline}\label{eq:key-magnetic-cancellation}
\cB(\obU,\bE_\bm,\nabla E_\phi)
= \cB(\obU,\bE_\bm + \nabla E_\phi,\nabla E_\phi)
= -\cB(\obU,\nabla E_\phi,\bE_\bm + \nabla E_\phi)\\
= -\cB(\bE_\bm + \nabla E_\phi,\nabla E_\phi,\obU)
= \cB(\bE_\bm + \nabla E_\phi,\obU,\nabla E_\phi) + (\div(\bE_\bm+\nabla E_\phi), \obU \cdot \nabla E_\phi).
\end{multline}
%

It follows that
\begin{equation*}
\begin{aligned}
\langle N_\Phi, E_\phi \rangle
& = \cB(\bE_\bu,\obM,\nabla E_\phi) + \cB(\obU,\bE_\bm,\nabla E_\phi)\\
& \quad - \frac{1}{2}(\bcurl \bE_\bu, \obM \times \nabla E_\phi)
- \frac{1}{2}(\bcurl \obU, \bE_\bm \times \nabla E_\phi)
- \beta(\bK_{\bM,\Phi},\nabla E_\phi).
\end{aligned}
\end{equation*}
Taking into account these last identities, \eqref{eq:proto-manipulating-nonlinearities} can be rewritten as
\begin{equation}\label{eq:rewriting-residue}
\begin{aligned}
\langle \bR_\bu, & \bE_\bu \rangle +\frac{\mu_0}{\chi} \langle \bR_\bm, \bE_\bm \rangle + \mu_0 \langle R_\phi, E_\phi \rangle\\
& = - \cB(\bE_\bu, \obU, \bE_\bu) + \mu_0 \cB(\bE_\bm, \nabla\oPhi, \bE_\bu)\\
& \quad + \bcancel{\mu_0 \cB(\obM, \nabla E_\phi, \bE_\bu)} + \frac{\mu_0}{2}(\bE_\bm \times \nabla\oPhi, \bcurl \bE_\bu) + \cancel{\frac{\mu_0}{2}(\obM \times \nabla E_\phi, \bcurl \bE_\bu)}\\
& \quad - \frac{\mu_0}{\chi} \cB(\bE_\bu, \obM, \bE_\bm) + \frac{\mu_0}{2\chi}(\bcurl \bE_\bu, \obM \times \bE_\bm) + \bcancel{\mu_0\cB(\bE_\bu,\obM,\nabla E_\phi)}\\
& \quad + \mu_0 \cB(\bE_\bm + \nabla E_\phi,\obU,\nabla E_\phi) + \mu_0(\div \bE_\bm + \Delta E_\phi, \obU \cdot \nabla E_\phi)\\ & \quad \cancel{- \frac{\mu_0}{2}(\bcurl \bE_\bu, \obM \times \nabla E_\phi)} - \frac{\mu_0}{2}(\bcurl \obU, \bE_\bm \times \nabla E_\phi)\\
& \quad + \frac{\mu_0 \beta}{\chi}(\bK_{\bM,\Phi},\bE_\bm - \chi \nabla E_\phi)\\
& \quad - \langle \sM_\bu \osfU, \bE_\bu \rangle - \frac{\mu_0}{\chi} \langle \sM_\bm \osfU, \bE_\bm \rangle - \mu_0 \langle \sM_\phi \osfU - F_\phi, E_\phi \rangle\\
& \quad - \kappa \mu_0\, (\Delta\oPhi + \div\obM, \div \bE_\bm + \Delta E_\phi),
\end{aligned}
\end{equation}
where the nontrivial but crucial cancellation is due to the identity
\begin{equation*}
\cB(\obM, \nabla E_\phi, \bE_\bu) = \cB(\bE_\bu, \nabla E_\phi, \obM) = -\cB(\bE_\bu, \obM, \nabla E_\phi),
\end{equation*}
which relies on \eqref{eq:symmetry} (symmetry of $\nabla^2 E_\phi$) and \eqref{eq:skew-symmetry} ($\div \bE_\bu = 0$).

Substituting identity \eqref{eq:rewriting-residue} in \eqref{eq:error-energy-law}, and realizing that no derivatives act on $\bE_\bm$ and $\nabla E_\phi$, except for the stabilization, we now use H\"older's inequalities to estimate each of the terms on the right-hand side of \eqref{eq:rewriting-residue}.
We bound the terms involving the trilinear form $\cB$ as follows: $|\cB(\ba,\bb,\bc)| \leq \|\ba\| \|\nabla\bb\|_{0,\infty}\|\bc\|$.
In turn, the term involving $\Delta E_\phi$ is bounded as follows
\begin{equation*}
\mu_0|(\div \bE_\bm + \Delta E_\phi, \obU \cdot \nabla E_\phi)|
\leq \frac{\mu_0\kappa}{4}\|\div \bE_\bm + \Delta E_\phi\|^2 + \frac{\mu_0}{\kappa} \|\obU\|_{0,\infty}^2 \|\nabla E_\phi\|^2.
\end{equation*}
Consistency terms are controlled via Young's inequality:
\begin{align*}
|\langle \sM_\bu \osfU, \bE_\bu \rangle| & \leq \frac{1}{\nu} \|\sM_\bu \osfU\|_{\bbUdiv'}^2 + \frac{\nu}{4} \|\nabla \bE_\bu\|^2,\\
\frac{\mu_0}{\chi} |\langle \sM_\bm \osfU, \bE_\bm \rangle| & \leq \frac{\mu_0 \tau}{2C_{\text{pos}}\chi^2} \|\sM_\bm \osfU\|_{\widetilde\bbM'}^2 + \frac{\mu_0 C_{\text{pos}}}{2\tau} \|\bE_\bm\|^2,\\
\mu_0 |\langle \sM_\phi \osfU - F_\phi, E_\phi \rangle| & \leq \frac{\mu_0 \tau}{2C_{\text{pos}}} \|\sM_\phi \osfU - F_\phi\|_{\tilde\bbX'}^2 + \frac{\mu_0 C_{\text{pos}}}{2\tau} \|\nabla E_\phi\|^2,\\
\kappa \mu_0 |(\Delta\oPhi + \div\obM, \div \bE_\bm + \Delta E_\phi)| & \leq \kappa \mu_0\|\Delta\oPhi + \div\obM\|^2 + \frac{\kappa \mu_0}{4}\|\div \bE_\bm + \Delta E_\phi\|^2.
\end{align*}
For the two rotational terms involving $\bcurl\bE_\bu$, use $\|\bcurl\bE_\bu\|=\|\nabla\bE_\bu\|$, since $\bE_\bu\in\bH^1_0(\Omega)$ satisfies $\div \bE_\bu = 0$.
Their combined contribution satisfies
\begin{multline*}
\frac{\mu_0}{2}\left|(\bE_\bm\times\nabla\oPhi,\bcurl\bE_\bu)\right|
+\frac{\mu_0}{2\chi}\left|(\bcurl\bE_\bu,\obM\times\bE_\bm)\right|\\
\leq\frac{\mu_0}{2}
\left(\|\nabla\oPhi\|_{0,\infty}+\chi^{-1}\|\obM\|_{0,\infty}\right)
\|\bE_\bm\|\,\|\nabla\bE_\bu\|\\
\leq\frac{\nu}{4}\|\nabla\bE_\bu\|^2
+\frac{\mu_0^2}{4\nu}
\left(\|\nabla\oPhi\|_{0,\infty}+\chi^{-1}\|\obM\|_{0,\infty}\right)^2
\|\bE_\bm\|^2.
\end{multline*}
The velocity consistency term consumes $\nu\|\nabla\bE_\bu\|^2/4$, and the two rotational terms together consume another $\nu\|\nabla\bE_\bu\|^2/4$.
This leaves the coefficient $\nu/2$ required in $\sD^*$.
The remaining rotational coefficient is included in $\eta$.

To estimate the term involving $\bK_{\bM,\Phi}$ in \eqref{eq:rewriting-residue}, multiply the bound \eqref{eq:spin-remainder-bound} from \cref{lem:spin-remainder} by $\mu_0\beta$.
The pure-error term $\mu_0\beta\|\bE_\bm\times\nabla E_\phi\|^2$ cancels with the spin dissipation on the left-hand side of \eqref{eq:error-energy-law}.
The coefficients on the right-hand side of \eqref{eq:spin-remainder-bound}, multiplied by $\mu_0\beta$, are bounded in terms of $\|\obM\|_{0,\infty}$, $\|\nabla\oPhi\|_{0,\infty}$, $\mu_0$, $\chi$, and $\beta$, so we include their contribution in $\eta$.

Altogether, recalling definition \eqref{eq:simpler-dissipation} of $\sD^*$, we arrive at
\begin{multline*}
\frac{\dd}{\dd t} \sE[\bE_\bu,\bE_\bm,E_\phi] + \sD^*[\bE_\bu,\bE_\bm,E_\phi] + \frac{\mu_0\beta}{2}\|\bM\times\nabla\Phi-\obM\times\nabla\oPhi\|^2\\
\leq \eta \Big( \|\bE_\bu\|^2 + \|\bE_\bm\|^2 + \|\nabla E_\phi\|^2  \Big)
+ \frac{1}{\nu} \|\sM_\bu \osfU\|_{\bbUdiv'}^2 + \frac{\mu_0 \tau}{2C_{\text{pos}}\chi^2} \|\sM_\bm \osfU\|_{\widetilde\bbM'}^2\\ + \frac{\mu_0 \tau}{2C_{\text{pos}}} \|\sM_\phi \osfU - F_\phi\|_{\tilde\bbX'}^2 + \kappa \mu_0\|\Delta\oPhi + \div\obM\|^2.
\end{multline*}
By integrating in time, it transpires that
\begin{multline}\label{eq:integrated-error-stability}
\sE[\bE_\bu,\bE_\bm,E_\phi](t) + \int_0^t \sD^*[\bE_\bu,\bE_\bm,E_\phi] \dd s
+\frac{\mu_0\beta}{2}\int_0^t\|\bM\times\nabla\Phi-\obM\times\nabla\oPhi\|^2\dd s\\
\leq \sE_0 + \int_0^t \eta(s) \Big( \|\bE_\bu\|^2 + \|\bE_\bm\|^2 + \|\nabla E_\phi\|^2  \Big) \dd s + \sG(t).
\end{multline}
The equivalence between the sum of the squared norms and the energy functional $\sE$ bounds the time integral on the right-hand side of \eqref{eq:integrated-error-stability} by $\tilde C\int_0^t\eta(s)\sE[\bE_\bu,\bE_\bm,E_\phi](s)\dd s$, where $\tilde C=\tilde C(\mu_0,\chi)>0$.
The entire left-hand side of \eqref{eq:integrated-error-stability} dominates $\sE[\bE_\bu,\bE_\bm,E_\phi](t)$ because both dissipation terms are nonnegative.
As $\sG$ is nondecreasing, Gr\"onwall's inequality applied to that entire left-hand side yields
\begin{multline*}
\sE[\bE_\bu,\bE_\bm,E_\phi](t) + \int_0^t \sD^*[\bE_\bu,\bE_\bm,E_\phi] \dd s
+\frac{\mu_0\beta}{2}\int_0^t\|\bM\times\nabla\Phi-\obM\times\nabla\oPhi\|^2\dd s\\
\leq \Big( \sE_0 + \sG(t) \Big) \exp\left(\tilde C \|\eta\|_{\LL^1([0,t])} \right).
\end{multline*}
Upon redefining $\eta$ as $\tilde C\eta$, we obtain \eqref{eq:stability-discrete-error}.
\end{proof}
\end{lemma}

\begin{corollary}[discrete perturbation estimate]\label{cor:stability-discretization-error}
The Galerkin error $\sfe$ satisfies the following estimate:
\begin{multline}\label{eq:stability-discretization-error}
\sE[\be_\bu,\be_\bm,e_\phi](t) + \int_0^t \sD^*[\be_\bu,\be_\bm,e_\phi] \dd s +\frac{\mu_0\beta}{2}\int_0^t\|\bm\times\nabla\phi-\bM\times\nabla\Phi\|^2\dd s
\\
\leq 2 \sE[\be_\bu^p,\be_\bm^p,e_\phi^p](t)
+ 2 \int_0^t \sD^*[\be_\bu^p,\be_\bm^p,e_\phi^p] \dd s\\
+\mu_0\beta \int_0^t\|\bm\times\nabla\phi-\obM\times\nabla\oPhi\|^2\dd s
+ 2(\sE_0 + \sG(t))\exp\left(\|\eta\|_{\LL^1([0,t])} \right),
\end{multline}
\begin{proof}
It follows from \eqref{eq:stability-discrete-error} by exploiting the quadratic structure of $\sE$ and $\sD^*$, recalling that $\sfe = \sfe^p - \sfE$, and using the inequality $|\ba+\bb|^2 \leq 2|\ba|^2 + 2|\bb|^2$.
For the spin term, we apply the same inequality after adding and subtracting $\obM\times\nabla\oPhi$ in $\bm\times\nabla\phi-\bM\times\nabla\Phi$.
\end{proof}
\end{corollary}

\subsection{Bounding consistency errors}\label{sec:bounding-consistency}
By using that $(\bu,p,\bm,\phi)$ satisfies the momentum and magnetization equations, we now estimate the consistency error $\sG(t)$ given by \eqref{eq:consistency-terms} by appropriate Sobolev norms of the projection error $\sfe^p$ and its time derivative.




Because of the symmetry of $\nabla^2\oPhi$ and $\nabla^2 \phi$, and the skew-symmetry of the trilinear form $\cB$ (cf.~\eqref{eq:skew-symmetry}) and in view of \eqref{eq:nonlinops-ex-sol}, we have for $\sM_\bu \osfU$:
\begin{multline*}
\langle \sM_\bu \osfU, \bV_\bu \rangle
= (\partial_t(\obU-\bu), \bV_\bu) - \left[\cB(\obU,\bV_\bu,\obU) - \cB(\bu,\bV_\bu, \bu)\right] + \nu(\nabla(\obU-\bu), \nabla \bV_\bu)
\\ + \mu_0 \left[\cB(\bV_\bu,\obM,\nabla\oPhi) - \cB(\bV_\bu,\bm,\nabla\phi)\right] - \frac{\mu_0}{2} \left[ (\obM \times \nabla\oPhi - \bm \times \nabla\phi, \bcurl \bV_\bu) \right], \qquad \forall\, \bV_\bu \in \bbUdiv
\end{multline*}
%
whence,
\begin{equation}\label{eq:bound-consistency-u}
\begin{aligned}
\|\sM_\bu \osfU\|_{\bbUdiv'}
& \leq C_P\|\partial_t\be_\bu^p\| + \big(\|\bu\|_{0,\infty} + \|\obU\|_{0,\infty}\big)\|\be_\bu^p\| + \nu\|\nabla \be_\bu^p\|\\
& \qquad + \mu_0 C_P\|\nabla\be_\bm^p\|\|\nabla\oPhi\|_{0,\infty} + C_P\mu_0\|\nabla \bm\|_{0,\infty}\|\nabla e_\phi^p\| + \frac{\mu_0}{2}\|\be_\bm^p\|\|\nabla\oPhi\|_{0,\infty} + \frac{\mu_0}{2}\|\bm\|_{0,\infty}\|\nabla e_\phi^p\|\\
& \leq C_1 \Big( \|\partial_t\be_\bu^p\| + \|\be_\bu^p\|_1 + \|\be_\bm^p\|_1 + \|\nabla e_\phi^p\| \Big),
\end{aligned}
\end{equation}
where $C_1 > 0$ depends on physical parameters, the positive constant in \eqref{subeqs:w1inf-stab} and Sobolev norms of the continuous solution $(\bu,\bm,\phi)$.
Similarly, for $\sM_\bm \osfU$ we get
\begin{multline*}
\langle \sM_\bm \osfU, \bV_\bm \rangle
= (\partial_t (\obM-\bm), \bV_\bm) + \left[\cB(\obU,\obM,\bV_\bm) - \cB(\bu,\bm,\bV_\bm) \right]\\
 - \frac{1}{2} \left[ (\bcurl \obU, \obM \times \bV_\bm) - (\bcurl \bu, \bm \times \bV_\bm) \right]\\
+ \frac{1}{\tau} \left[(\obM-\bm - \chi \nabla(\oPhi-\phi), \bV_\bm)\right]
- \beta[(\bG(\obM,\nabla\oPhi) - \bG(\bm,\nabla\phi), \bV_\bm)] \qquad \forall\, \bV_\bm \in \bbM.
\end{multline*}
To handle the cubic nonlinearity, notice that
\begin{align*}
\|\bG(\obM,\nabla\oPhi) - \bG(\bm,\nabla\phi)\|
& \leq \|(\obM \times \nabla\oPhi - \bm \times \nabla\phi) \times \obM \| + \|(\bm \times \nabla\phi) \times (\bm- \obM)\|\\
& \leq  \left( \|\bm\|_{0,\infty} \|\nabla e_\phi^p\| + \|\nabla \oPhi\|_{0,\infty}\|\be_\bm^p\| \right) \|\obM\|_{0,\infty} + \|\bm \times \nabla \phi\|_{0,\infty} \|\be_\bm^p\|
\end{align*}
Consequently,
\begin{equation}\label{eq:bound-consistency-m}
\begin{aligned}
\|\sM_\bm \osfU\|_{\widetilde\bbM'}
& \leq \|\partial_t \be_\bm^p\| + \|\bu\|_{0,\infty}\|\nabla\be_\bm^p\| + \|\nabla\obM\|_{0,\infty}\|\be_\bu^p\|\\
& \qquad + \frac{1}{2}\|\bcurl\bu\|_{0,\infty}\|\be_\bm^p\| + \frac{1}{2}\|\obM\|_{0,\infty}\|\bcurl \be_\bu^p\| + \frac{1}{\tau}\|\be_\bm^p\| + \frac{\chi}{\tau}\|\nabla e_\phi^p\|\\ & \qquad  + \beta\|\bG(\obM,\nabla\oPhi) - \bG(\bm,\nabla\phi)\|\\
& \leq C_2 \Big( \|\be_\bu^p\|_1 + \|\partial_t \be_\bm^p\| + \|\be_\bm^p\|_1 + \|\nabla e_\phi^p\| \Big).
\end{aligned}
\end{equation}
In turn,
\begin{multline*}
\langle \sM_\phi \osfU - F_\phi, V_\phi \rangle
= (\nabla \partial_t(\oPhi-\phi),\nabla V_\phi) + \frac{1+\chi}{\tau}(\nabla (\oPhi-\phi), \nabla V_\phi)\\
- \left[ \cB(\obU,\obM,\nabla V_\phi) - \cB(\bu,\bm,\nabla V_\phi) \right] + \frac{1}{2}\left[ (\bcurl \obU, \obM \times \nabla V_\phi) - (\bcurl \bu, \bm \times \nabla V_\phi) \right]\\
+ \beta[(\bG(\obM,\nabla\oPhi) - \bG(\bm,\nabla\phi), \nabla V_\phi)]
, \qquad \forall\, V_\phi \in \bbX,
\end{multline*}
whence
\begin{equation}\label{eq:bound-consistency-phi}
\begin{aligned}
\|\sM_\phi \osfU - F_\phi\|_{\tilde\bbX'}
& \leq \|\nabla \partial_t e_\phi^p\| + \frac{1+\chi}{\tau}\|\nabla e_\phi^p\| + \|\bu\|_{0,\infty}\|\nabla \be_\bm^p\| + \|\nabla \obM\|_{0,\infty}\|\be_\bu^p\|\\
& \qquad + \frac{1}{2}\|\bcurl \bu\|_{0,\infty}\|\be_\bm^p\| + \frac{1}{2}\|\obM\|_{0,\infty}\|\bcurl \be_\bu^p\| + \beta\|\bG(\obM,\nabla\oPhi) - \bG(\bm,\nabla\phi)\|\\
& \leq C_3 \Big(\|\be_\bu^p\|_1 + \|\be_\bm^p\|_1 + \|\nabla \partial_t e_\phi^p\| + \|\nabla e_\phi^p\| \Big).
\end{aligned}
\end{equation}
Finally, using that $\Delta\phi + \div\bm = 0$ yields
\begin{equation*}
\kappa \mu_0\|\Delta\oPhi + \div\obM\|^2
= \kappa \mu_0\|\Delta(\oPhi-\phi) + \div(\obM-\bm)\|^2
\leq 2\kappa \mu_0\left( \|\Delta e_\phi^p\|^2 + d\|\nabla \be_\bm^p\|^2  \right).
\end{equation*}
Combining bounds \eqref{eq:bound-consistency-u}, \eqref{eq:bound-consistency-m}, and \eqref{eq:bound-consistency-phi}, we deduce that the consistency terms in $\sG$ satisfy
\begin{multline}\label{eq:bulk-consistency-rates}
\|\sM_\bu \osfU\|_{\bbUdiv'}^2 + \|\sM_\bm \osfU\|_{\widetilde\bbM'}^2 + \|\sM_\phi \osfU - F_\phi\|_{\tilde\bbX'}^2 + \|\Delta\oPhi + \div\obM\|^2\\
\leq C_{\text{bulk}} \left( \|\be_\bu^p\|_1^2 + \|\partial_t \be_\bu^p\|^2 + \|\be_\bm^p\|_1^2 + \|\partial_t \be_\bm^p\|^2 + \|e_\phi^p\|_2^2 + \|\partial_t e_\phi^p\|_1^2 \right),
\end{multline}
where $C_{\text{bulk}} = C_{\text{bulk}}(\|\obU\|_{1,\infty}, \|\obM\|_{1,\infty}, \|\oPhi\|_{2,\infty}, \|\bu\|_{1,\infty}, \|\bm\|_{1,\infty}, \beta\|\bm \times \nabla \phi\|_{0,\infty},\kappa, \mu_0, \chi, \tau, \beta, \nu, C_{\text{pos}}) > 0$ is increasing in its first six arguments.
We stress that this constant is uniform as $\beta \to 0$.

\subsection{Error estimates}\label{sec:error-estimates}
Let $\{\cT_h\}_{h>0}$ be a sequence of conforming triangulations of $\Omega$ by simplices (triangles in 2D and tetrahedra in 3D), indexed by their mesh size $h := \max_{T \in \cT_h} \operatorname{diam}(T)$.
For each $h>0$ we assume that the finite dimensional spaces $\bbU$, $\bbP$, $\bbM$ and $\bbX$ are subordinate to $\cT_h$.
Solely for simplicity of notation, we do not write the explicit dependence on $h$ of these spaces nor the functions belonging to them.

Before stating our general a priori error estimate we must specify the approximation and stability properties of the finite dimensional spaces \eqref{eq:fin-dim-subspaces} used in the discrete scheme \eqref{all-semidiscrete-scheme}.
Later, in \cref{sec:specific-FES}, we shall specify finite element spaces and auxiliary projections that are compatible (with the exception of one estimate that remains conditional) with our abstract assumptions.

In all three assumptions below, the constant $C>0$ is a universal constant that may depend on the geometric properties of the sequence of meshes $\{\cT_h\}_{h>0}$ (e.g.~shape-regularity or quasi-uniformity), but is independent of $h$.

\begin{assump}[approximability properties of the finite element spaces]\label{assump:fes-approx}
There exist positive integers $k_\bu$, $k_\bm$ and $k_\phi$ such that the auxiliary projections $\obU$, $\obM$, $\oPhi$ of the exact solutions $\bu$, $\bm$, $\phi$ satisfy:
\begin{subequations}\label{subeqs:approx}
\begin{align}
\|\bu-\obU\|_1 &  \leq C h^{k_\bu} \|\bu\|_{k_\bu+1}, \label{eq:approx-u}\\
\|\bm-\obM\|_1 & \leq C h^{k_\bm} \|\bm\|_{k_\bm+1}, \label{eq:approx-m}\\
\|\phi-\oPhi\|_2 & \leq C h^{k_\phi} \|\phi\|_{k_\phi+2} \label{eq:approx-phi},
\end{align}
for almost every $t\in(0,T)$.
Moreover,
\begin{align}
\|\partial_t\bu- \partial_t\obU\|_{\LL^2(0,T;\bL^2)} &  \leq C h^{k_\bu} \|\partial_t \bu\|_{\LL^2(0,T;\bH^{k_\bu})}, \label{eq:approx-dt-u}\\
\|\partial_t\bm-\partial_t\obM\|_{\LL^2(0,T;\bL^2)} & \leq C h^{k_\bm} \|\partial_t \bm\|_{\LL^2(0,T;\bH^{k_\bm})}, \label{eq:approx-dt-m}\\
\|\partial_t\phi- \partial_t\oPhi\|_{\LL^2(0,T;\HH^1)} & \leq C h^{k_\phi} \|\partial_t \phi\|_{\LL^2(0,T;\HH^{k_\phi+1})} \label{eq:approx-dt-phi},
\end{align}
\end{subequations}
\end{assump}

\begin{assump}[stability of the projections]\label{assump:W1infty-stability}
For almost every $t\in(0,T)$, the following stability estimates hold:
\begin{subequations}\label{subeqs:w1inf-stab}
\begin{align}
\|\obU\|_{1,\infty} &  \leq C \|\bu\|_{1,\infty}, \label{W1inf-stab-Stokes}\\
\|\obM\|_{1,\infty} & \leq C \|\bm\|_{1,\infty}, \label{W1inf-stab-magnetization}\\
\|\oPhi\|_{2,\infty} & \leq C \|\phi\|_{2,\infty} \label{W1inf-stab-potential}.
\end{align}
\end{subequations}
\end{assump}

\begin{assump}[approximation of the initial conditions]\label{assump:ic-approx}
The initial condition triple $(\bU_0, \bM_0, \Phi_0) \in \bbUdiv \times \bbM \times \bbX$ (cf.~\eqref{eq:discrete-initial-conditions}) satisfies
\begin{align}
\|\bu_0-\bU_0\| &  \leq C h^{k_\bu} \|\bu_0\|_{k_\bu}, \label{eq:approx-ic-u}\\
\|\bm_0-\bM_0\| & \leq C h^{k_\bm} \|\bm_0\|_{k_\bm}, \label{eq:approx-ic-m}\\
\|\phi_0-\Phi_0\|_1 & \leq C h^{k_\phi} \|\phi_0\|_{k_\phi+1} \label{eq:approx-ic-phi}.
\end{align}
\end{assump}

\begin{remark}[required regularity for projections]
The norms appearing on the right-hand sides of the estimates in \cref{assump:fes-approx,assump:W1infty-stability,assump:ic-approx} and the implied regularity of the functions being approximated are consistent with the classical approximation theory literature.
Projectors/interpolators that are defined on more regular functions and may require higher-order norms on the right-hand sides of the aforementioned estimates are also valid options provided the exact solution is regular enough.
\end{remark}

We now establish the rates of convergence of the proposed abstract Galerkin scheme.

\begin{thm}[A priori error estimate]\label{thm:error-estimate}
Let $\sfU = (\bU,\bM,\Phi)$ be the Galerkin solution of the semidiscrete system \eqref{all-semidiscrete-scheme} provided by \cref{prop:semidiscrete-solvability} (\nameref*{prop:semidiscrete-solvability}).
Assume that \cref{assump:strongly-div-free,assump:proj-time-commute,assump:fes-approx,assump:W1infty-stability,assump:vel-aux-div-free} hold.
Let $\sfu = (\bu,\bm,\phi)$ be an exact solution of \eqref{ferrofluids-system} with the following regularity:
\begin{equation}\label{eq:error-theorem-regularity}
\begin{aligned}
\bu
&\in \LL^2(0,T;\bH^{k_\bu+1}(\Omega))
\cap \LL^\infty(0,T;\bW^{1,\infty}(\Omega)),
&\partial_t\bu
&\in \LL^2(0,T;\bH^{k_\bu}(\Omega)),\\
\bm
&\in \LL^2(0,T;\bH^{k_\bm+1}(\Omega))
\cap \LL^\infty(0,T;\bW^{1,\infty}(\Omega)),
&\partial_t\bm
&\in \LL^2(0,T;\bH^{k_\bm}(\Omega)),\\
\phi
&\in \LL^2(0,T;\HH^{k_\phi+2}(\Omega))
\cap \LL^\infty(0,T;\WW^{2,\infty}(\Omega)),
&\partial_t\phi
&\in \LL^2(0,T;\HH^{k_\phi+1}(\Omega)).
\end{aligned}
\end{equation}
Set $k:=\min\{k_\bu,k_\bm,k_\phi\}$.
Assume that the discrete initial conditions satisfy
\begin{equation}\label{eq:initial-energy-accuracy}
\sE[\bu_0-\bU_0,\bm_0-\bM_0,\phi_0-\Phi_0]
\leq C_{\mathrm{init}}h^{2k},
\end{equation}
where $C_{\mathrm{init}}$ is independent of $h$.

Then, for every $0<h\leq1$ and $t\in[0,T]$, the total discretization error $\sfe$ satisfies the a priori estimate
\begin{equation}\label{eq:bound-general}
\sE[\be_\bu,\be_\bm,e_\phi](t) + \int_0^t \sD^*[\be_\bu,\be_\bm,e_\phi] \dd s + \beta \int_0^t\|\bm\times\nabla\phi-\bM\times\nabla\Phi\|^2\dd s
\leq C h^{2k},
\end{equation}
where $C>0$ is independent of $h$ and $t$.
The constant may depend on $\Omega$, the physical parameters, the stabilization parameter $\kappa$, the final time $T$, the norms in \eqref{eq:error-theorem-regularity}, $C_{\mathrm{init}}$, and the uniform constants in \cref{assump:fes-approx,assump:W1infty-stability}.
With all other parameters fixed, it is uniform as $\beta\to0$ provided these solution norms, $C_{\mathrm{init}}$, and the approximation and stability constants remain uniformly bounded.
\begin{proof}
\smallskip
We estimate the right-hand side of \eqref{eq:stability-discretization-error} in several steps.

\noindent\textbf{1.~Projection errors.}
We first establish the following bound for the projection error $\sfe^p$:
\begin{equation}\label{eq:projection-errors-for-theorem}
\begin{aligned}
&\sup_{0\leq s\leq T}
\left(
\|\be_\bu^p(s)\|^2
+\|\be_\bm^p(s)\|^2
+\|e_\phi^p(s)\|_1^2
\right)\\
&\quad+\int_0^T
\Big(
\|\be_\bu^p\|_1^2
+\|\partial_t\be_\bu^p\|^2
+\|\be_\bm^p\|_1^2
+\|\partial_t\be_\bm^p\|^2
+\|e_\phi^p\|_2^2
+\|\partial_t e_\phi^p\|_1^2
\Big)\dd s
\leq Ch^{2k}.
\end{aligned}
\end{equation}
Integrating the spatial approximation estimates in \cref{assump:fes-approx} and using their time-derivative counterparts together with \eqref{eq:error-theorem-regularity} bounds the integral in \eqref{eq:projection-errors-for-theorem} by $Ch^{2k}$.
Moreover, \cref{assump:proj-time-commute} and \eqref{eq:error-theorem-regularity} imply that $\be_\bu^p,\be_\bm^p\in\HH^1(0,T;\bL^2(\Omega))$ and $e_\phi^p\in\HH^1(0,T;\HH^1(\Omega))$.
To control the supremum, let $Y$ be a real Hilbert space and $w\in\HH^1(0,T;Y)$, which we identify with its absolutely continuous representative on $[0,T]$.
The function $t\mapsto\|w(t)\|_Y^2$ is absolutely continuous and satisfies the chain rule
\begin{equation*}
\frac{\dd}{\dd t}\|w(t)\|_Y^2
=2(w(t),\partial_t w(t))_Y,
\qquad\text{for almost every }t\in(0,T).
\end{equation*}
Consequently, for every $s,t\in[0,T]$,
\begin{equation}\label{eq:time-energy-comparison}
\begin{aligned}
\|w(t)\|_Y^2
&=\|w(s)\|_Y^2+2\int_s^t(w(r),\partial_t w(r))_Y\dd r\\
&\leq\|w(s)\|_Y^2+2\int_0^T\|w(r)\|_Y\|\partial_t w(r)\|_Y\dd r.
\end{aligned}
\end{equation}
Integrating \eqref{eq:time-energy-comparison} with respect to $s$ over $(0,T)$ and dividing by $T$ replaces $\|w(s)\|_Y^2$ by $T^{-1}\|w\|_{\LL^2(0,T;Y)}^2$.
Applying the Cauchy--Schwarz inequality to the remaining integral and taking the supremum over $t\in[0,T]$ gives
\begin{equation}\label{eq:time-trace-bound}
\sup_{0\leq t\leq T}\|w(t)\|_Y^2
\leq \frac{1}{T}\|w\|_{\LL^2(0,T;Y)}^2
+2\|w\|_{\LL^2(0,T;Y)}\|\partial_t w\|_{\LL^2(0,T;Y)}.
\end{equation}
In particular, \eqref{eq:time-trace-bound} controls $\|w(0)\|_Y^2$ using only these time-integrated norms, without a separate assumption on the initial value.
Apply \eqref{eq:time-trace-bound} to $(w,Y)=(\be_\bu^p,\bL^2(\Omega))$, $(\be_\bm^p,\bL^2(\Omega))$, and $(e_\phi^p,\HH^1(\Omega))$.
The integral bound already established controls both factors on the right-hand side by $Ch^k$, proving \eqref{eq:projection-errors-for-theorem}, including the projection-error bound at $t=0$.

\smallskip\smallskip
\noindent\textbf{2.~Coefficients and consistency.}
From the $\LL^\infty$-based stability estimates in \cref{assump:W1infty-stability} and \eqref{eq:error-theorem-regularity}, we have that
\begin{equation}\label{eq:uniform-time-Linfty-based}
\|\obU(s)\|_{1,\infty}+\|\obM(s)\|_{1,\infty}+\|\oPhi(s)\|_{2,\infty}
\leq C, \qquad \text{a.e.~$s \in (0,T)$}.
\end{equation}
Consequently, the functions $\eta$ and $C_{\text{bulk}}$ from \cref{lem:stability-discrete-error} and \eqref{eq:bulk-consistency-rates} satisfy
\begin{equation}\label{eq:error-theorem-coefficient-bounds}
\|\eta\|_{\LL^1(0,T)}+\|C_{\text{bulk}}\|_{\LL^\infty(0,T)}\leq C.
\end{equation}
Recall the definition \eqref{eq:consistency-terms} of $\sG$.
Integrating the consistency bound \eqref{eq:bulk-consistency-rates} in time, using the bound \eqref{eq:projection-errors-for-theorem} of the projection errors and estimate \eqref{eq:error-theorem-coefficient-bounds}, therefore implies
\begin{equation}\label{eq:consistency-for-theorem}
\sG(t) \leq Ch^{2k}, \qquad \forall t \in [0,T].
\end{equation}

\smallskip
\noindent\textbf{3.~Projected and initial energies.}
The Maxwell residual $\Delta e_\phi^p+\div\be_\bm^p$ satisfies
\begin{equation}\label{eq:maxwell-res-bound}
\|\Delta e_\phi^p+\div\be_\bm^p\|^2
\leq2\|\Delta e_\phi^p\|^2+2\|\div \be_\bm^p\|^2
\leq C\left(\|e_\phi^p\|_2^2+\|\be_\bm^p\|_1^2\right).
\end{equation}
In turn, we bound the spin residual $\bm\times\nabla\phi-\obM\times\nabla\oPhi$ as
\begin{equation}\label{eq:spin-res-bound}
\begin{aligned}
\|\bm\times\nabla\phi-\obM\times\nabla\oPhi\|^2
& \leq \left( \|\obM\|_{0,\infty} \|\nabla e_\phi^p\| + \|\nabla \phi\|_{0,\infty}\|\be_\bm^p\| \right)^2\\
& \leq \tilde C \left( \|e_\phi^p\|_1^2 + \|\be_\bm^p\|^2\right),
\end{aligned}
\end{equation}
where $\tilde C$ is independent of time and depends only on the uniform projection stability constants from \cref{assump:W1infty-stability} and the norms of $\phi$ and $\bm$ in $\LL^\infty(0,T;\WW^{1,\infty}(\Omega))$ and $\LL^\infty(0,T;\bW^{1,\infty}(\Omega))$, respectively.

Recall the definitions of $\sE$ and $\sD^*$ in \eqref{eq:energy} and \eqref{eq:simpler-dissipation}.
Therefore, combining \eqref{eq:maxwell-res-bound} and \eqref{eq:spin-res-bound} with \eqref{eq:projection-errors-for-theorem} results in
\begin{equation}\label{eq:projected-energy-for-theorem}
\sup_{0\leq s\leq T}\sE[\be_\bu^p,\be_\bm^p,e_\phi^p](s)
+\int_0^T\sD^*[\be_\bu^p,\be_\bm^p,e_\phi^p]\dd s
+\beta \int_0^T\|\bm\times\nabla\phi-\obM\times\nabla\oPhi\|^2\dd s
\leq Ch^{2k}.
\end{equation}
The initial accuracy assumption \eqref{eq:initial-energy-accuracy} yields $\sE[\be_\bu,\be_\bm,e_\phi](0)\leq C_{\mathrm{init}}h^{2k}$.
Since $\sfE=\sfe^p-\sfe$ and $\sE$ is quadratic, the initial discrete-error energy $\sE_0$ defined in \eqref{eq:initial-error} therefore satisfies
\begin{equation}\label{eq:initial-discrete-energy-for-theorem}
\sE_0
\leq2\sE[\be_\bu^p,\be_\bm^p,e_\phi^p](0)
+2\sE[\be_\bu,\be_\bm,e_\phi](0)
\leq Ch^{2k}.
\end{equation}

\smallskip
\noindent\textbf{4.~Galerkin error.}
Substituting bounds \eqref{eq:error-theorem-coefficient-bounds}, \eqref{eq:consistency-for-theorem}, \eqref{eq:projected-energy-for-theorem}, and \eqref{eq:initial-discrete-energy-for-theorem} into the discrete perturbation estimate \eqref{eq:stability-discretization-error} yields
\begin{equation*}
\sE[\be_\bu,\be_\bm,e_\phi](t)
+\int_0^t\sD^*[\be_\bu,\be_\bm,e_\phi]\dd s
+ \beta\int_0^t\|\bm\times\nabla\phi-\bM\times\nabla\Phi\|^2\dd s
\leq Ch^{2k}, \qquad \forall t \in [0,T].
\end{equation*}
This finishes the proof.
\end{proof}
\end{thm}

The regularity \eqref{eq:error-theorem-regularity} implies
\begin{equation*}
\bu \in\CC([0,T];\bH^{k_\bu}(\Omega)),\qquad
\bm \in\CC([0,T];\bH^{k_\bm}(\Omega)), \qquad
\phi \in\CC([0,T];\HH^{k_\phi+1}(\Omega)),
\end{equation*}
because each component belongs to $\HH^1(0,T;Y)$ with the corresponding spatial Hilbert space $Y$.
Consequently, the initial-data norms in \cref{assump:ic-approx} are finite, and that assumption is a sufficient condition for \eqref{eq:initial-energy-accuracy}.
No separate spatial approximation bound for the auxiliary projections at $t=0$ is needed.

\begin{remark}[vanishing of penalization]
The stabilization terms $\kappa \chi\, (\Delta\Phi + \div\bM, \div \bV_\bm)$ and $\kappa(\Delta\Phi + \div\bM, \Delta V_\phi)$ in \eqref{eq:semidiscrete-scheme} provide least-squares control of the discrete Maxwell residual $r_h:=\Delta\Phi+\div\bM$, which vanishes identically on an exact constrained solution.
Since $\Delta\phi+\div\bm=0$, estimate \eqref{eq:bound-general} implies
\begin{equation*}
\|r_h\|_{\LL^2(0,T;\LL^2(\Omega))}=\mathcal O(h^k),
\end{equation*}
and shows that the added numerical dissipation vanishes as $h\to0$ with fixed coefficients in \eqref{eq:semidiscrete-scheme}.
This stabilization does not introduce a fixed magnetic diffusivity but rather controls the combined quantity $\Delta\Phi + \div\bM$ without separately controlling $\nabla\bM$ and $\Delta\Phi$.

\end{remark}

\section{Specific finite element spaces and auxiliary projections}\label{sec:specific-FES}

We now specify finite element spaces and auxiliary projections $\osfU := (\obU, \obM, \oPhi)$ that are compatible with \cref{assump:strongly-div-free,assump:proj-time-commute,assump:fes-approx,assump:W1infty-stability,assump:ic-approx,assump:vel-aux-div-free}, with the (possible) exception of estimate \eqref{W1inf-stab-Stokes} for $d=3$ (see \cref{sec:u-p} for further details).
We also choose discrete initial conditions $(\bU_0,\bM_0,\Phi_0)\in\bbUdiv\times\bbM\times\bbX$ satisfying \eqref{eq:initial-energy-accuracy} under the regularity assumptions of \cref{assump:ic-approx}.
These initial conditions are obtained by solving linear projection problems, and hence they are computationally easy to implement.

From now on, we will assume that $\{\cT_h\}_{h>0}$ is shape-regular.
For each $h>0$, we let $\cT_h^\rmb$ be the barycentric refinement (also called the Alfeld split) of $\cT_h$, which is obtained by joining the barycenter of each element of $\cT_h$ to each face to produce $(d+1)$ subsimplices, as explained in \cref{sec:magnetic-potential} below.

\subsection{Velocity-pressure pair}\label{sec:u-p}

For $d\in\{2,3\}$, the most classical finite element pair that satisfies \cref{assump:strongly-div-free} and is inf-sup stable on $\cT_h^\rmb$ for $\ell_\bu\geq d$ \cite[Corollary 6.2]{GuzmanNeilan2018} is the Scott--Vogelius pair \cite{ScottVogelius1985}.
It consists of continuous piecewise vector-valued polynomials of degree at most $\ell_\bu$ for the velocity and discontinuous piecewise scalar polynomials of degree at most $\ell_\bu-1$ for the pressure.
We write $\bbU =  \bsP_{\ell_\bu}(\cT_h^\rmb) \cap \bH_0^1(\Omega)$ and $\bbP = \sP_{\ell_\bu-1}(\cT_h^\rmb) \cap \LL^2_\#(\Omega)$.
Here, $\sP_\ell(\cT)$ denotes piecewise polynomials of degree at most $\ell$ on $\cT$, without continuity constraints.
For the approximation of the initial condition, the discrete Leray projection $\bU_0 = \cP_{h,\div} \bu_0 \in \bbU_{\div}$ of $\bu_0$ is defined as the first component of the solution of the problem:
find $(\bU_0,P_0) \in \bbU \times \bbP$ such that
\begin{equation}\label{eq:disc-Leray}
\begin{aligned}
(\bU_0, \bV_\bu) - (\div \bV_\bu, P_0) &= (\bu_0, \bV_\bu), \qquad \forall\, \bV_\bu \in \bbU,\\
-(\div \bU_0, V_p) & = 0, \qquad \forall\, V_p \in \bbP.
\end{aligned}
\end{equation}
The discrete Leray projection is the $\bL^2(\Omega)$-orthogonal projection onto $\bbUdiv$.
Its best-approximation property, combined with a flux-preserving polynomial approximation and the local divergence lifting in \cite[Theorem 3.1]{GuzmanNeilan2018}, yields \eqref{eq:approx-ic-u} with $k_\bu=\ell_\bu$.

Regarding the auxiliary projection, for $d=2$ we utilize the \emph{local Fortin projection} $\sF_h : \bW^{1,1}_0(\Omega) \rightarrow \bbU$ from \cite[Section 3]{Eickmann2026} and set $\obU = \sF_h \bu$.
From \cite[Theorem 7 and Remark 9]{Eickmann2026}, we deduce that $\obU(t) \in \bbU_{\div}$ for each time $t$ and estimates \eqref{eq:approx-u}, \eqref{eq:approx-dt-u} and \eqref{W1inf-stab-Stokes} hold for $k_\bu = \ell_\bu$.

For $d=3$, the existence of a Fortin operator with the same characteristics is still a pending issue.
However, two of the authors of \cite{Eickmann2026} have suggested in a private communication \cite{GuzmanNeilan2026private} that it is highly likely that the arguments from \cite{Eickmann2026} could be adapted to the 3D setting.
We do not address this question as it is outside the scope of the present paper, and leave it pending.

We finish this section by mentioning that other inf-sup stable velocity-pressure \emph{macroelement} pairs that satisfy \cref{assump:strongly-div-free} are available in the literature, such as Arnold--Qin \cite{ArnoldQin1992}, Alfeld--Sorokina \cite{AlfeldSorokina2016}, and the Guzmán--Neilan pairs found in \cite{GuzmanNeilan2014-triangle,GuzmanNeilan2014-tetrahedra,GuzmanNeilan2018}.
We are unaware, however, of any results that could be utilized to construct an auxiliary projection $\obU$ of $\bu$ with the stability and approximation properties \eqref{eq:approx-u}, \eqref{eq:approx-dt-u} and \eqref{W1inf-stab-Stokes}.

\subsection{Magnetization}
For $d \in \{2,3\}$, we know from \cite[Theorems 2.1, 3.1 and 4.1]{ScottZhang1990} that the Scott--Zhang interpolant $\mathcal{S\!Z}_h: \WW^{1,1}(\Omega) \rightarrow C(\overline\Omega) \cap \sP_{\ell_\bm}(\cT_h)$, where $C(\overline\Omega) \cap \sP_{\ell_\bm}(\cT_h)$ is the space of continuous piecewise Lagrange elements of degree $\ell_\bm \geq 1$, is stable in $\WW^{1,\infty}(\Omega)$ and optimal (in the sense of Bramble--Hilbert) in $\HH^1(\Omega)$.
Moreover, we know (see e.g., \cite{CiarletHandbookError1991}) that the $\LL^2$-projection $\cQ_h: \LL^2(\Omega) \rightarrow C(\overline\Omega) \cap \sP_{\ell_\bm}(\cT_h)$ is optimal in $\LL^2(\Omega)$ as well.

Consequently, setting $\bbM = \bC(\overline\Omega) \cap \bsP_{\ell_\bm}(\cT_h)$, $\bM_0 = \cQ_h \bm_0$ and $\obM = \mathcal{S\!Z}_h \bm$ (where the operators are understood as acting componentwise) yields estimates \eqref{eq:approx-m}, \eqref{eq:approx-dt-m}, \eqref{W1inf-stab-magnetization} and \eqref{eq:approx-ic-m} with $k_\bm = \ell_\bm$.

\subsection{Magnetic potential}\label{sec:magnetic-potential}
In this section, we construct the auxiliary potential $\oPhi$ by using an operator that takes into account the degrees of freedom associated to the derivative values.
This operator will allow us to retain $\HH^2$-conforming macroelements without requiring pointwise derivatives of the function being approximated, while preserving optimal approximation rates (in the sense of Bramble--Hilbert).
The construction follows the averaging principle of Girault and Scott \cite{GiraultScott2002}; an averaged Hsieh--Clough--Tocher (HCT) construction is also discussed in \cite[Remark 5.2]{DuanDuLiuZhang2019}.
Since the potential space has no essential boundary conditions, we use volume averages for every degree of freedom and enforce only the zero-mean condition.
We present the construction and the estimates needed here, including $\WW^{2,\infty}$ stability.

On each $T\in\cT_h$, let $S(T)$ be the full \cite[Section 46]{CiarletHandbookError1991} or reduced \cite{Ciarlet1978} HCT space when $d=2$, or a full \cite{Walkington2014} or reduced \cite[eq.~(4.3), Lemmas 4.1 and 4.5, and Remark 4.2]{BrubeckLianParker2026} Walkington tetrahedral macroelement space when $d=3$; see also the reduced construction in \cite{BrubeckGithub2026}.
Each macroelement is subdivided by joining its barycenter to its faces, and one fixed macroelement family is used throughout $\{\cT_h\}_{h>0}$.
For a reduced Walkington space, we require the chosen reduction to commute with translations and uniform dilations:
\begin{equation}\label{eq:potential-reduced-similarity}
S(\bF(K))=\{w\circ\bF^{-1}:w\in S(K)\},
\qquad \bF(\bx)=\bb+\alpha\bx,
\end{equation}
for every nondegenerate tetrahedron $K$, every $\bb\in\RR^3$, and every $\alpha>0$.
We also require its homogeneous defining relations, when expressed in fixed polynomial bases on a reference Alfeld split, to have coefficients depending continuously on the ordered vertex coordinates.
The retained degrees of freedom must be unisolvent on every nondegenerate tetrahedron.
We denote by $\widetilde{\bbX} \subseteq \HH^2(\Omega)$ the space that results from globally assembling all local spaces $S(T)$, $T \in \cT_h$.
In particular, shared Hermite degrees of freedom are identified, including the vertex Hessians for Walkington's element.
As attested in \cite{CiarletHandbookError1991,Ciarlet1978,Walkington2014,BrubeckLianParker2026,BrubeckGithub2026}, each of the local spaces $S(T)$ contains $\sP_r(T)$, where
\begin{equation}\label{eq:potential-reproduction-degree}
\text{$r = 3$ (full HCT); \quad $r=2$ (reduced HCT); \quad $r=4$ (Walkington); \quad $r=3$ (reduced Walkington)}.
\end{equation}
Here, $r$ is the \emph{polynomial reproduction degree}, i.e., the largest integer for which $\sP_r(T) \subseteq S(T)$.
We set
\begin{equation*}
\bbX:=\widetilde{\bbX}\cap\LL^\infty_\#(\Omega)
\subset \WW^{2,\infty}_\#(\Omega).
\end{equation*}
For each global degree of freedom $\lambda_a$, we fix, once and for all, an incident macroelement $T_a \in \cT_h$.
We use a fixed orientation for each shared degree of freedom associated to normal derivatives.
Degrees of freedom located on the boundary are retained.

For each degree of freedom $a$, let $\psi_a \in S(T_a)$ be the uniquely determined dual function with respect to the $\LL^2(T_a)$ inner product; that is:
\begin{equation}\label{eq:local-dual-def}
\int_{T_a} w\psi_a= \lambda_a(w), \quad\forall w \in S(T_a).
\end{equation}
%

For each $T \in \cT_h$, let $Q_T:\LL^1(T)\rightarrow S(T)$ be defined by
\begin{equation}\label{eq:potential-local-moment-projection}
\int_T(Q_Tv)w=\int_Tvw,
\qquad \forall w\in S(T),\quad v\in\LL^1(T).
\end{equation}
Since $S(T)\subset\LL^\infty(T)$, the right-hand side of \eqref{eq:potential-local-moment-projection} is finite, and the positive definiteness of the $\LL^2(T)$ mass matrix ensures that $Q_Tv$ is uniquely determined.
For each $v \in \LL^1(\Omega)$, we define $J_hv \in \tilde\bbX$ as the function whose degrees of freedom $\lambda_a(J_h v)$ are uniquely determined by
\begin{equation}\label{eq:potential-averaged-dofs}
\lambda_a(J_hv):=\lambda_a(Q_{T_a}v)
=\int_{T_a}v\psi_a,
\end{equation}
for every global degree of freedom $\lambda_a$.
The resulting linear operator $J_h : \LL^1(\Omega) \rightarrow \tilde\bbX$ is a projector (i.e.~$J_h^2 = J_h$) and reproduces $\sP_r$.
For functions of zero mean, we correct $J_h$ as follows:
we define the linear operator $I_h^\phi: \LL^1_\#(\Omega) \rightarrow \bbX$ by
\begin{equation}\label{eq:potential-mean-corrected-projector}
I_h^\phi v:=J_hv-\frac{1}{|\Omega|}\int_\Omega J_hv, \qquad \forall v \in \LL^1_\#(\Omega).
\end{equation}
Clearly, $I_h^\phi$ is also a projector.
\begin{lemma}[approximability and stability properties of $J_h$]\label{lem:potential-interpolation}
For each $0<h\leq1$ and every $1\leq p\leq\infty$, the operator $J_h$ satisfies
\begin{equation}\label{eq:potential-averaged-approximation}
\|v-J_h v\|_{j,p}\leq Ch^{s-j}\|v\|_{s,p},
\qquad
j\in\{0,1,2\},\quad 2\leq s\leq r+1,
\end{equation}
for every integer $s$ in this range and every $v\in\WW^{s,p}(\Omega)$.
Moreover,
\begin{equation}\label{eq:potential-averaged-stability}
\|J_h v\|_{2,p}\leq C\|v\|_{2,p},
\qquad v\in\WW^{2,p}(\Omega).
\end{equation}
The constants may depend on $p$, the fixed macroelement family and its shape-regularity bound, but are independent of $h$.
\end{lemma}

We postpone the proof of \cref{lem:potential-interpolation} to \cref{app:proof-interp}.

\begin{corollary}[approximability and stability properties of $I_h^\phi$]\label{cor:Ih-interpolation}
Estimates \eqref{eq:potential-averaged-approximation} and \eqref{eq:potential-averaged-stability} are also valid for $I_h^\phi$ provided the spaces $\WW^{s,p}(\Omega)$ and $\WW^{2,p}(\Omega)$ are replaced by $\WW^{s,p}_\#(\Omega)$ and $\WW^{2,p}_\#(\Omega)$, respectively.
\begin{proof}
Notice that if $\int_\Omega v = 0$, then H\"older's inequality gives, for every $1\leq p\leq\infty$,
\begin{equation*}
|\Omega|^{1/p}\left|\frac{1}{|\Omega|}\int_\Omega J_hv\right|
\leq\|J_hv-v\|_{0,p}, \qquad
|\Omega|^{1/p}\left|\frac{1}{|\Omega|}\int_\Omega J_hv\right|
\leq\|J_hv\|_{0,p},
\end{equation*}
with the convention $1/\infty:=0$.
The common left-hand side is the $\LL^p(\Omega)$ norm of the mean $|\Omega|^{-1}\int_\Omega J_hv$ of $J_h v$ over $\Omega$.
Then, the result follows from the triangle inequality and estimates \eqref{eq:potential-averaged-approximation} and \eqref{eq:potential-averaged-stability}.
\end{proof}
\end{corollary}

We now choose $\oPhi(t)=I_h^\phi\phi(t)$.
\cref{cor:Ih-interpolation} yields for any integer $1\leq k_\phi\leq r-1$ and $(p,j,s)=(2,2,k_\phi+2)$
\begin{equation}\label{H2-approx-potential}
\|e_\phi^p\|_2
=\|\phi-I_h^\phi\phi\|_2
\leq Ch^{k_\phi}\|\phi\|_{k_\phi+2}.
\end{equation}
The operator $I_h^\phi$ is fixed in time and bounded on $\LL^2_\#(\Omega)$ for each mesh, so it commutes with the weak time derivative.
Applying \eqref{eq:potential-averaged-approximation} for $I_h^\phi$ with $(p,j,s)=(2,1,k_\phi+1)$ gives
\begin{equation}\label{eq:potential-time-approximation}
\|\partial_t e_\phi^p\|_{\LL^2(0,T;\HH^1)}
\leq Ch^{k_\phi}
\|\partial_t\phi\|_{\LL^2(0,T;\HH^{k_\phi+1})}.
\end{equation}
For the approximation of the initial condition $\phi_0$, we choose $\Phi_0 \in \bbX$ as the Ritz projection of $\phi_0$ with Neumann conditions.
Namely, $\Phi_0$ is the unique solution of
\begin{equation*}
\int_\Omega \nabla \Phi_0 \cdot \nabla V_\phi = \int_\Omega \nabla \phi_0 \cdot \nabla V_\phi, \qquad \forall \, V_\phi \in \bbX.
\end{equation*}
Quasi-best approximation yields
\begin{equation}\label{eq:phi0-best-approx}
\|\phi_0 - \Phi_0\|_1
\leq C \inf_{\substack{V_\phi \in \bbX}} \|\phi_0 - V_\phi\|_1
\leq C \|\phi_0 - I_h^\phi \phi_0\|_1.
\end{equation}
Choosing again $(p,j,s) = (2,1,k_\phi+1)$ gives
\begin{equation*}
\|\phi_0-\Phi_0\|_1\leq Ch^{k_\phi}\|\phi_0\|_{k_\phi+1}.
\end{equation*}
Together with \eqref{eq:potential-averaged-stability} for $p=\infty$, these are precisely \eqref{eq:approx-phi}, \eqref{eq:approx-dt-phi}, \eqref{W1inf-stab-potential}, and \eqref{eq:approx-ic-phi}.
The full Walkington space allows for $k_\phi = 3$, the reduced Walkington and full HCT spaces allow for $k_\phi = 2$, while the reduced HCT allows for $k_\phi = 1$.
All these bounds use exactly the potential regularity stated in \eqref{eq:error-theorem-regularity}, without additional pointwise differentiability assumptions.
\section{Computational experiments}\label{sec:experiments}

In this section we illustrate the performance of our numerical scheme presented in \cref{sec:discretization} in 2D and 3D domains.
Our implementation utilizes the open-source finite element library \texttt{Firedrake} \cite{firedrake} and the temporal discretization library \texttt{Irksome} \cite{irksome}.
In all experiments below we consider the physical parameters $\nu = 0.1$, $\mu_0 = 10^{-6}$, $\tau = 10^{-5}$ and $\chi = 0.5$.
For time discretization, we utilize the Qin--Zhang symplectic DIRK scheme \cite{QinZhang1992}, which retains its formal second-order accuracy under variable time stepping.
We focus on spatial discretization and leave the intricacies of its coupling with temporal discretizations for future research.
Non-homogeneous Dirichlet conditions for the velocity are implemented in the discrete scheme in the usual way.
We define the discretization error associated to $\bh = \nabla\phi$ as
\begin{equation*}
\be_\bh := \nabla\phi - \nabla\Phi = \nabla e_\phi.
\end{equation*}
Finally, we mention that in all convergence tables the scalar parameter $h$ denotes the mesh size of the original mesh $\cT_h$ and not of its barycentric refinement $\cT_h^\rmb$.

\subsection{Convergence test in 2D}\label{sec:conv-test-2D}
In this test, we experimentally corroborate the spatial convergence rates asserted by \cref{thm:error-estimate} in 2D.
Our computational domain is the unit square $\Omega = (0,1)^2$ and we consider uniform triangulations $\cT_h$ of size $h$.
We set the final time to $T = 0.2$ and consider uniform time steps $\Delta t = h^{3/2}/32$, shortening the last time step if necessary to reach $T$.
We set the stabilization parameter to $\kappa = 0.1$ and the nonlinear relaxation parameter to $\beta = 10^{-5}$.
We consider the applied magnetic field $\bh_a = \nabla\phi_a$, where the potential $\phi_a(t,x,y) = (1-t)^2 (x^2 - y^2)$ is harmonic, and manufactured solutions
\begin{gather}
\bu = e^{-t} [\sin(\pi x)^2 \sin(2\pi y), 0.5 \sin(2\pi x) (\cos(2\pi y) - 1)], \qquad p = (x-1)^3 \sin(2\pi y t) - C_1(t),\\
\bh = \nabla\phi, \qquad \phi = t^2  e^x y - C_2(t), \qquad \bm = t^3 \bcurl(x^2 (1-x)^2 y^2 (1-y)^2 ) - \bh + \bh_a,
\end{gather}
where $C_1(t)$ and $C_2(t)$ are such that $\int_\Omega p = \int_\Omega \phi = 0$ for each $t \in [0,T]$.
We add forcing terms $\bf_\bu$ and $\bf_\bm$ to the right-hand sides of \eqref{eq:momentum} and \eqref{eq:magnetization}, respectively, so that the manufactured fields satisfy the resulting equations.
In the equivalent formulation (cf.~\cref{sec:equivalent-formulation}) this requires replacing \eqref{eq:full-magnetic-incomp} by $\div(\bF + \bf_\bm) = 0$ and its corresponding boundary condition \eqref{eq:total-magnetic-flux} by $(\bF + \bf_\bm) \cdot \bnu = 0$ on $\partial\Omega$.
In the semidiscrete scheme \eqref{all-semidiscrete-scheme}, we add the term $-(\bf_\bm, \nabla V_\phi)$ to the right-hand side of the last equation of \eqref{eq:semidiscrete-scheme}.

For spatial discretization we utilize the lowest-order Scott–Vogelius pair for $(\bu,p)$ as discussed in \cref{sec:u-p}, continuous piecewise quadratic polynomials for $\bm$, and the full Hsieh--Clough--Tocher finite element for $\phi$.
Hence, $k_\bu = k_\bm = k_\phi =2$.
In agreement with the discussion in \cref{sec:specific-FES}, we initialize the velocity, magnetization and magnetic potential fields by respectively using the discrete Leray projection \eqref{eq:disc-Leray}, the $\bL^2$ projection and the Ritz projection with Neumann boundary conditions.

We summarize our findings in \cref{tab:convergence_rates_2D}, where $e_{\mathrm{spin}}:=\bm\times\bh-\bM\times\nabla\Phi$.
Experimental convergence rates are in agreement with \cref{thm:error-estimate}.
The rate in $\LL^\infty \bL^2$ for $\bu$ is one order higher than the rate guaranteed by \cref{thm:error-estimate}, while the rates for $\bm$ and $\bh$ do not enjoy this additional order of convergence.
This hints that an argument in the spirit of Wheeler's trick may be implementable in our context to derive higher rates for the velocity field.
The fourth-order Maxwell stabilization introduces consistency terms that are controlled using $\HH^2$ estimates in the present analysis, complicating the derivation of sharper magnetic error estimates.
Whether the stabilization also explains the observed quadratic convergence in this experiment remains an open question.
Time discretization should not be the reason for the observed quadratic convergence of $\bm$ and $\bh$:
the discretization error introduced by the Qin–Zhang method behaves, at least formally, as $(\Delta t)^2 \sim h^3$.
Moreover, on the finest mesh, the $\LL^\infty(\bL^2)$ errors for $\bh$ and $\bm$ obtained with $\Delta t=h^{3/2}/16$ (not reported in this manuscript) differ by less than $0.25\%$ from those reported in \cref{tab:convergence_rates_2D}.
This strongly suggests that temporal discretization error does not control the observed magnetic convergence rates.
We leave a rigorous investigation of these questions for future research.

\begin{table}
    \setlength{\tabcolsep}{3pt}
    \resizebox{\textwidth}{!}{
	\begin{tabular}{c c c c c c c c c c c c c c}
	\toprule
	DOFs & $h$ & $\|\be_\bu\|_{\LL^\infty(\bL^2)}$ & Rate & $\|\nabla\be_\bu\|_{\LL^2(\bbL^2)}$ & Rate & $\|\be_\bh\|_{\LL^\infty(\bL^2)}$ & Rate & $\|\be_\bm\|_{\LL^\infty(\bL^2)}$ & Rate & $\|\div(\be_\bm+\be_\bh)\|_{\LL^2(\LL^2)}$ & Rate & $\|e_{\mathrm{spin}}\|_{\LL^2(\LL^2)}$ & Rate\\
	\midrule
	   3783 & 0.1768 & 6.3260\texttt{E}-03 &   -- & 1.6034\texttt{E}-01 &   -- & 2.5990\texttt{E}-06 &   -- & 1.8607\texttt{E}-06 &   -- & 3.6221\texttt{E}-05 &   -- & 7.3217\texttt{E}-07 &   -- \\
	  14727 & 0.0884 & 8.4083\texttt{E}-04 & 2.91 & 4.8924\texttt{E}-02 & 1.71 & 4.1881\texttt{E}-07 & 2.63 & 5.2240\texttt{E}-07 & 1.83 & 1.4480\texttt{E}-05 & 1.32 & 2.3732\texttt{E}-07 & 1.63 \\
	  58119 & 0.0442 & 1.1223\texttt{E}-04 & 2.91 & 1.3363\texttt{E}-02 & 1.87 & 1.2443\texttt{E}-07 & 1.75 & 1.6524\texttt{E}-07 & 1.66 & 6.2531\texttt{E}-06 & 1.21 & 6.6221\texttt{E}-08 & 1.84 \\
	 230919 & 0.0221 & 1.4387\texttt{E}-05 & 2.96 & 3.4375\texttt{E}-03 & 1.96 & 3.2603\texttt{E}-08 & 1.93 & 4.4775\texttt{E}-08 & 1.88 & 1.8949\texttt{E}-06 & 1.72 & 1.6940\texttt{E}-08 & 1.97 \\
	 920583 & 0.0110 & 1.8110\texttt{E}-06 & 2.99 & 8.6616\texttt{E}-04 & 1.99 & 8.2766\texttt{E}-09 & 1.98 & 1.1401\texttt{E}-08 & 1.97 & 3.9010\texttt{E}-07 & 2.28 & 4.2806\texttt{E}-09 & 1.98 \\
	\bottomrule
	\end{tabular}
    }
    \centering
    \caption{\emph{Spatial convergence rates in 2D for $k_\bu = k_\bm = k_\phi = 2$ and $\beta = 10^{-5}$}}
    \label{tab:convergence_rates_2D}
\end{table}

\subsection{Convergence test in 3D}\label{sec:conv-test-3D}
In this test, we explore the convergence behavior of the scheme in three dimensions.
We consider the computational domain $\Omega = (0,1)^3$ uniformly partitioned with tetrahedra of size $h$.
We set $T=0.1$ and discretize $[0,T]$ with uniform time steps $\Delta t = h^{3/2}/16$, shortening the last time step if necessary to reach $T$.
We set the stabilization parameter to $\kappa = 0.1$ and the nonlinear relaxation parameter to $\beta = 0$.
The applied magnetic field $\bh_a = \nabla\phi_a$ is given by the harmonic potential $\phi_a(t,x,y,z) = (1 - t)^2  ((x^5 - 10  x^3  y^2 + 5  x  y^4) + (y^5 - 10  y^3  z^2 + 5  y  z^4))$.
We consider the manufactured solutions
\begin{gather}
\bu = e^{-t} [\sin(\pi x) \cos(\pi y) \cos(\pi z), -2 \cos(\pi x) \sin(\pi y) \cos(\pi z),\cos(\pi x) \cos(\pi y) \sin(\pi z)],\\
p = ((x - 1)^2 (z - 1)) \sin(2 \pi y \,t) - C_1(t), \qquad \bh = \nabla\phi, \qquad \phi = (t + 1) ((x^3 - y^2)^2 + (y^2 + z^3)^2) - C_2(t),\\
\bm = \bb - \bh + \bh_a, \qquad \bb = \bcurl [\omega(x)\omega(y)\omega(z), \omega(x)\omega(y)\omega(z), \omega(x)\omega(y)\omega(z)], \qquad \omega(s) = s^2 (1 - s)^2,
\end{gather}
where $C_1(t)$ and $C_2(t)$ are such that $\int_\Omega p = \int_\Omega \phi = 0$ for each time $t \in [0,T]$.
We modify the momentum \eqref{eq:momentum} and magnetization equations \eqref{eq:magnetization} as in \cref{sec:conv-test-2D} in order to ensure that $(\bu,p,\bm,\phi)$ satisfy the resulting equations.

For spatial discretization we consider the second-order Guzmán--Neilan pair found in \cite[Section 6.2]{GuzmanNeilan2018} (called ``Guzman--Neilan 2nd kind $H^1$'' in \texttt{Firedrake}), which consists of continuous piecewise vector-valued polynomials of degree $\leq 2$ enriched with ``modified'' face bubbles on $\cT_h^\rmb$ for $\bu$, and piecewise discontinuous scalar-valued polynomials of degree $\leq 1$ on $\cT_h^\rmb$.
For the magnetization field $\bm$, we consider continuous piecewise vector-valued polynomials of degree $\leq 2$ on $\cT_h$.
Finally, the magnetic potential $\phi$ is discretized using the reduced Walkington macroelement space \cite{BrubeckGithub2026}.
Hence, $k_\bu = k_\bm = k_\phi = 2$.
We stress that this choice of exact solutions and finite element discretization is not strictly covered by our theory for the following two reasons:
\begin{itemize}
  \item $\bu$ does not vanish on $\partial\Omega$.
  Hence, repeating the integration by parts in \eqref{eq:key-magnetic-cancellation} in the context of the error analysis would introduce a boundary term $\int_{\partial\Omega} (\obU \cdot \nabla E_\phi)(\bE_\bm + \nabla E_\phi) \cdot \bnu$, which is not controllable by the discrete energy law.
  \item As we discussed in \cref{sec:u-p}, we are unaware of the existence of an appropriate auxiliary projection $\obU$ using the Guzmán--Neilan pair that satisfies the $\WW^{1,\infty}$-stability estimate \eqref{W1inf-stab-Stokes} in \cref{assump:W1infty-stability}.
\end{itemize}
As in \cref{sec:conv-test-2D} and in agreement with the discussion in \cref{sec:specific-FES}, we initialize the velocity, magnetization and magnetic potential fields by respectively using the discrete Leray projection \eqref{eq:disc-Leray} (in this case, with non-zero Dirichlet conditions to account for the non-vanishing trace of $\bu$), the $\bL^2$ projection and the Ritz projection with Neumann boundary conditions.
Our numerical experiments are summarized in \cref{tab:convergence_rates_3D}.
The $\LL^\infty(\bL^2)$ errors for $\bu$, $\bh$, and $\bm$ exhibit nearly third-order convergence along the refinement path $\Delta t=h^{3/2}/16$, while the Maxwell residual converges at approximately second order.
The convergence rate of the velocity error in the $\LL^2(\bH^1)$ seminorm remains close to $1.75$ over the reported refinements, which may be an artifact of the preasymptotic regime.

\begin{table}
    \setlength{\tabcolsep}{3pt}
    \resizebox{\textwidth}{!}{
\begin{tabular}{c c c c c c c c c c c c}
\toprule
DOFs & $h$ & $\|\be_\bu\|_{\LL^\infty(\bL^2)}$ & Rate & $\|\nabla\be_\bu\|_{\LL^2(\bbL^2)}$ & Rate & $\|\be_\bh\|_{\LL^\infty(\bL^2)}$ & Rate & $\|\be_\bm\|_{\LL^\infty(\bL^2)}$ & Rate & $\|\div(\be_\bm+\be_\bh)\|_{\LL^2(\LL^2)}$ & Rate\\
\midrule
   18392 & 0.6124 & 5.1163\texttt{E}-02 &   -- & 5.6253\texttt{E}-01 &   -- & 2.4637\texttt{E}-02 &   -- & 2.7468\texttt{E}-02 &   -- & 6.6716\texttt{E}-01 &   -- \\
  138528 & 0.3062 & 8.6851\texttt{E}-03 & 2.56 & 1.6680\texttt{E}-01 & 1.75 & 4.0077\texttt{E}-03 & 2.62 & 4.0954\texttt{E}-03 & 2.75 & 2.3536\texttt{E}-01 & 1.50 \\
 1077296 & 0.1531 & 1.1967\texttt{E}-03 & 2.86 & 5.0298\texttt{E}-02 & 1.73 & 5.3837\texttt{E}-04 & 2.90 & 5.6492\texttt{E}-04 & 2.86 & 6.3612\texttt{E}-02 & 1.89 \\
 8501328 & 0.0765 & 1.4565\texttt{E}-04 & 3.04 & 1.4930\texttt{E}-02 & 1.75 & 6.8908\texttt{E}-05 & 2.97 & 7.4518\texttt{E}-05 & 2.92 & 1.4971\texttt{E}-02 & 2.09 \\
\bottomrule
\end{tabular}
    }
    \centering
    \caption{\emph{Spatial convergence rates in 3D for $k_\bu = k_\bm = k_\phi = 2$}}
    \label{tab:convergence_rates_3D}
\end{table}

\section*{Acknowledgments}
We thank Mansur Shakipov (University of Houston) for his candid comments that helped improve the overall presentation of this paper.
We thank Pablo Brubeck (University of Oxford) for his implementation of the reduced Walkington element \cite{BrubeckGithub2026} in \texttt{Firedrake} upon the first author's request.

Both authors were partially supported by NSF grant DMS-2512392.

\section*{Declaration of Generative AI Use}
The decision to combine the magnetic reformulation with Maxwell-residual stabilization was the authors' idea.
The first author used ChatGPT (OpenAI), including the GPT-6 Astra and GPT-6.1 Sol models, to assist with mathematical exploration and algebraic calculations, drafting and revising proofs, literature searches, numerical code development and optimization, interpretation of computational results, and manuscript drafting and editing.
This assistance included an initial draft of the $\HH^2$-conforming projector construction and the development of estimates for the nonlinear magnetic relaxation.
These arguments were developed through an iterative process of critical examination and substantial revision.
A summary of earlier exploratory discussions with Gemini (Google) concerning nonlinear magnetic relaxation was also consulted.
The authors checked all mathematical content, references, computational implementations, and numerical results.
The authors assume responsibility for all content.

\crefalias{section}{appendix}
\appendix

\section{Proof of \texorpdfstring{\cref{lem:potential-interpolation}}{Lemma 5.1}}\label{app:proof-interp}

To prove \cref{lem:potential-interpolation}, we first establish local stability for $J_h$ uniformly over the given shape-regular mesh family.
The proof separates element size from shape, bounds the normalized nodal and dual functions, and then uses rescaling and polynomial reproduction.

\subsection{Local representation and proof strategy}\label{app:local-representation}
For each $T\in\cT_h$, set $h_T:=\operatorname{diam}(T)$ and let $\omega_T$ be the patch consisting of the union of all elements sharing a vertex with $T$.
Because of the shape-regularity of $\cT_h$ (which implies local quasi-uniformity), these patches have uniformly bounded overlap and the element diameters within each patch are uniformly comparable.

Here and below, $\nabla^j$ denotes the $j$-fold weak gradient, with $\nabla^0w:=w$ and $\nabla^2w$ the Hessian.
We will prove that, for every $v\in\LL^p(\Omega)$,
\begin{equation}\label{eq:potential-local-stability}
\|\nabla^jJ_hv\|_{0,p;T}
\leq Ch_T^{-j}\|v\|_{0,p;\omega_T},
\qquad 0\leq j\leq2,\quad p\in[1,\infty].
\end{equation}
The lemma will then follow from combining this estimate with polynomial reproduction and approximation properties on the patches $\omega_T$, in the spirit of Bramble--Hilbert.

For each $T\in\cT_h$, let $A(T)$ be the set of indices of the degrees of freedom associated with $S(T)$.
We denote by $\theta_{a,T}\in S(T)$ the local nodal basis function determined by
\begin{equation*}
\lambda_b(\theta_{a,T})=\delta_{ab},
\qquad a,b\in A(T).
\end{equation*}
The defining identity \eqref{eq:local-dual-def} for the local dual function $\psi_a\in S(T_a)$ implies
\begin{equation*}
\int_{T_a}\theta_{b,T_a}\psi_a=\delta_{ab},
\qquad b\in A(T_a).
\end{equation*}
In particular, \eqref{eq:potential-averaged-dofs} and the expansion of $(J_hv)|_T$ in the basis $\{\theta_{a,T}\}_{a\in A(T)}$ give
\begin{equation}\label{eq:potential-local-expansion}
(J_hv)|_T
=\sum_{a\in A(T)}\lambda_a(Q_{T_a}v)\theta_{a,T}
=\sum_{a\in A(T)}\left(\int_{T_a}v\psi_a\right)\theta_{a,T}.
\end{equation}
\subsection{Splitting the contributions of element size and shape}\label{app:normalization}
To bound each of the factors in \eqref{eq:potential-local-expansion}, we separate the effects of element size and shape.
The map $\bF_T$ is a translation and uniform dilation preserving the macroelement space, whereas $\bG_T$ will provide fixed polynomial coordinates for the varying normalized shapes.
Choose a vertex $\bx_T$ of each macroelement $T\in\cT_h$ and define
\begin{equation}\label{eq:potential-normalization-map}
T^\ast:=h_T^{-1}(T-\bx_T),
\qquad
\bF_T(\bxi):=\bx_T+h_T\bxi.
\end{equation}
Thus $\bF_T$ maps the unit-diameter element $T^\ast$ onto $T$.
For each chosen family, composition with $\bF_T$ maps $S(T)$ onto $S(T^\ast)$; in the reduced Walkington case, this follows from \eqref{eq:potential-reduced-similarity}.
Let $\mu_a\in\{0,1,2\}$ denote the derivative order of $\lambda_a$, that is, $\mu_a=0$ for function values, $\mu_a=1$ for first derivatives, and $\mu_a=2$ for second derivatives.
We define the corresponding \emph{normalized nodal} functions $\theta^\ast_{a,T}$ on $T^\ast$ and \emph{normalized dual} functions $\psi^\ast_a$ on $T_a^\ast$ by
\begin{equation}\label{eq:potential-normalized-functions}
\theta^\ast_{a,T} :=h_T^{-\mu_a}\theta_{a,T}\circ \bF_T, \qquad \psi^\ast_a :=h_{T_a}^{d+\mu_a}\psi_a\circ \bF_{T_a}.
\end{equation}
Since $\bF_T$ is a translation and a uniform dilation, the factor $h_T^{-\mu_a}$ makes $\theta^\ast_{a,T}$ the nodal function for the corresponding derivative-value degree of freedom on $T^\ast$, with normal orientations inherited from $T$.

For $b\in A(T)$, define the normalized degree of freedom $\lambda_{b,T}^\ast$ on $T^\ast$ by
\begin{equation}\label{eq:potential-normalized-dofs}
\lambda_{b,T}^\ast(w^\ast)
:=h_T^{\mu_b}\lambda_b(w^\ast\circ\bF_T^{-1}), \qquad \forall w^* \in S(T^*).
\end{equation}
The factor $h_T^{\mu_b}$ cancels the derivative scaling, so $\lambda_{b,T}^\ast$ is the corresponding value or derivative evaluation in the coordinates of $T^\ast$.
Together with \eqref{eq:potential-normalized-functions}, definition \eqref{eq:potential-normalized-dofs} gives $\lambda_{b,T}^\ast(\theta_{a,T}^\ast)=\delta_{ab}$.

The normalization \eqref{eq:potential-normalized-functions} gives
\begin{equation}\label{eq:potential-normalized-duality}
\begin{aligned}
\int_{T_a^\ast}\theta^\ast_{b,T_a}\psi^\ast_a
&=h_{T_a}^{d+\mu_a-\mu_b}\frac{|T_a^\ast|}{|T_a|}\int_{T_a}\theta_{b,T_a}\psi_a\\
&=h_{T_a}^{\mu_a-\mu_b}\int_{T_a}\theta_{b,T_a}\psi_a\\
&=h_{T_a}^{\mu_a-\mu_b}\delta_{ab}=\delta_{ab},
\qquad b\in A(T_a).
\end{aligned}
\end{equation}
The domains $T^\ast$ still vary with $T$.
To control this shape dependence, we represent them by a compact family of nondegenerate matrices.
For any simplex $K$, let $\rho(K)$ denote its inradius, namely the radius of its largest inscribed ball.
Every normalized macroelement $T^\ast$ has diameter one and ordered vertices $\bv_0^\ast=\mathbf 0,\bv_1^\ast,\ldots,\bv_d^\ast$, and shape regularity gives $\rho(T^\ast)=\rho(T)/h_T\geq\rho_0$ for some $\rho_0>0$ independent of $h$ and $T$.
Let $\widehat T$ be the fixed reference simplex with vertices $\widehat\bv_0=\mathbf 0$ and $\widehat\bv_i=\be_i$, $i\in\{1,\dotsc,d\}$.
Denote by $\bG_T:\widehat T\to T^\ast$ the affine map that sends $\widehat\bv_i$ to $\bv_i^\ast$.
Since $\widehat\bv_0=\bv_0^\ast=\mathbf 0$, $\bG_T(\widehat\bx)=\bB_T\widehat\bx$, where the columns of $\bB_T$ are $\bv_1^\ast,\dotsc,\bv_d^\ast$.
Applying \cite[Theorem~15.2]{CiarletHandbookError1991} to $\bG_T$ yields
\begin{equation*}
\|\bB_T\|\leq\frac{1}{2\rho(\widehat T)},
\qquad
\|\bB_T^{-1}\|\leq\frac{\operatorname{diam}(\widehat T)}{2\rho(T^\ast)}
\leq\frac{\operatorname{diam}(\widehat T)}{2\rho_0}.
\end{equation*}
Since the quantities on the right-hand side are independent of $T$ and $h$, we can thus write
\begin{equation}\label{eq:potential-normalized-affine-bounds}
\|\bB_T\|+\|\bB_T^{-1}\|\leq C,
\end{equation}
%
In particular, $\bG_T$ and $\bG_T^{-1}$ are uniformly bounded in $\WW^{1,\infty}$ on their respective domains.
Thus $\bG_T$ maps the fixed reference simplex onto the unit-diameter simplex $T^\ast$, whose shape is determined by $\bB_T$, and $\bF_T$ then restores the size and position of $T$.

Let $\widehat b$ and $b_T^\ast$ be the barycenters of $\widehat T$ and $T^\ast$, respectively.
Since affine maps preserve averages of points,
\begin{equation*}
\bG_T(\widehat b)
=\bG_T\left(\frac{1}{d+1}\sum_{i=0}^d\widehat \bv_i\right)
=\frac{1}{d+1}\sum_{i=0}^d \bv_i^\ast
=b_T^\ast.
\end{equation*}

Moreover, $\bG_T$ maps each face $\widehat f_i$ of $\widehat T$ onto the corresponding face $f_i^\ast$ of $T^\ast$.
Thus it maps each of the $d+1$ subsimplices of the reference Alfeld split onto its counterpart in $T^\ast$, since
\begin{equation*}
\bG_T\left(\operatorname{conv}(\{\widehat b\}\cup\widehat f_i)\right)
=\operatorname{conv}(\{b_T^\ast\}\cup f_i^\ast).
\end{equation*}
This identifies the three subtriangles when $d=2$, or the four subtetrahedra when $d=3$, with fixed reference subsimplices.

Estimate \eqref{eq:potential-normalized-affine-bounds} also supplies the compactness needed for the uniform bounds on $\{\theta^\ast_{a,T}\}_{a\in A(T)}$.
Define
\begin{equation}\label{eq:potential-compact-shapes}
\mathfrak B:=\overline{\{\bB_T:T\in\cT_h,\ h>0\}}\subset\RR^{d\times d};
\end{equation}
that is, $\mathfrak B$ is the closure of the collection of all matrices $\bB_T$, where $T$ ranges over all elements of all meshes in $\{\cT_h\}_{h>0}$.
Estimate \eqref{eq:potential-normalized-affine-bounds} readily implies that $\mathfrak B$ is compact and consists of invertible matrices.
The uniform bound on $\bB_T^{-1}$ excludes singular limit matrices, so every shape in $\mathfrak B$ remains nondegenerate.

\subsection{Uniform bounds on normalized elements}\label{app:normalized-bounds}
We prove continuous dependence of the normalized nodal functions on $\bB_T\in\mathfrak B$ through their coefficients in fixed reference coordinates.
Compactness then yields uniform bounds for these functions and their inverse mass matrices, which determine the dual functions.

For the following argument on dependence on shape, we choose auxiliary normal orientations that vary continuously with the ordered vertices.
Reversing a normal changes the sign of each affected degree of freedom and its associated nodal function, by unisolvence; the corresponding dual function also changes sign by \eqref{eq:local-dual-def}.
These sign changes preserve all the norms considered below.
We use the auxiliary orientations throughout the continuity argument, and the resulting uniform bounds also hold for the fixed global orientations.

\smallskip
\noindent\textbf{(i) Fixed polynomial coordinates.}
Let $\widehat K_i:=\operatorname{conv}(\{\widehat b\}\cup\widehat f_i)$, $0\leq i\leq d$, be the subsimplices of the reference element $\hat T$.
We use $\ell$ for a fixed upper bound on the total polynomial degree on each subsimplex $K$ in the Alfeld split of $T$, so that
\begin{equation*}
w|_K\in\sP_\ell(K),\qquad\forall w\in S(T).
\end{equation*}
The bound $\ell$ is independent of $h$ and $T$ and is distinct from the reproduction degree $r$ in \eqref{eq:potential-reproduction-degree}, which concerns the inclusion $\sP_r(T)\subseteq S(T)$.
Set $n:=\dim\sP_\ell(\widehat K_i)$, which is independent of $i$.
Choose a basis $\{\widehat p_{i,j}\}_{j=1}^n$ of $\sP_\ell(\widehat K_i)$ on each subsimplex and write
\begin{equation}\label{eq:polynomial-expansion}
(\theta^\ast_{a,T}\circ\bG_T)|_{\widehat K_i} =\sum_{j=1}^n c_{a,i,j,T} \, \widehat p_{i,j},
\qquad 0\leq i\leq d,
\end{equation}
The vector $\bc_{a,T}\in\RR^N$, $N:=(d+1)n$, collects the coefficients of $\theta^\ast_{a,T}\circ\bG_T$ on all reference subsimplices, in the fixed ordering
\begin{equation}\label{eq:potential-coefficient-ordering}
(\bc_{a,T})_{in+j}:=c_{a,i,j,T},
\qquad 0\leq i\leq d,\quad 1\leq j\leq n.
\end{equation}

\smallskip
\noindent\textbf{(ii) Equations characterizing the nodal functions.}
Membership constraints and prescribed nodal values give a linear system for $\bc_{a,T}$.
We describe its equations to justify their continuous dependence on the element geometry.
On each subsimplex $K_i^\ast:=\bG_T(\widehat K_i)$, choose an arbitrary polynomial $w_i\in \sP_\ell(K_i^\ast)$.
Let $w$ denote the piecewise polynomial function whose restriction to each $K_i^\ast$ is $w_i$.
For these pieces to define a $C^1$ function, their values and gradients must agree across each internal face $F = K_i^\ast \cap K_k^\ast$:
\begin{equation*}
(w_i-w_k)|_F=0,
\qquad
(\nabla w_i-\nabla w_k)|_F=\mathbf 0.
\end{equation*}
The collection of these polynomial identities is encoded by a matrix $\mathsf C_{\mathrm{cont},T}$.
A second matrix, $\mathsf C_{\mathrm{extra},T}$, encodes the remaining restrictions defining the chosen macroelement space $S(T^\ast)$.
For instance, the normal derivative of a cubic HCT piece is a quadratic polynomial along an outer edge $e$; the reduced HCT condition $\partial_{\bn_e}w|_e\in\sP_1(e)$ therefore requires its degree-two term to vanish.
For the full Walkington element, which has 45 degrees of freedom, the restrictions include matching derivatives of $w_i$ through order four at the barycenter $b_T^\ast$ and requiring $\partial_{\bn_f}w|_f\in\sP_3(f)$ on each outer face $f$; see \cite[Notation 2.1]{Walkington2014}.
For a reduced Walkington space, $\mathsf C_{\mathrm{extra},T}$ encodes the defining relations of the fixed reduction \cite{BrubeckLianParker2026,BrubeckGithub2026}.
These need not retain Walkington's $C^4$ condition at the barycenter; for example, the construction in \cite[eq.~(4.3)]{BrubeckLianParker2026} uses a variational condition in the interior.
These conditions are homogeneous linear relations between the entries of the coefficient vector $\bc\in\RR^N$; the values of $w$ and its normal derivatives need not vanish.
The admissible coefficient vectors are precisely those satisfying
\begin{equation}\label{eq:potential-macroelement-constraints}
\mathsf C_{\mathrm{cont},T}\bc=\mathbf 0,
\qquad
\mathsf C_{\mathrm{extra},T}\bc=\mathbf 0.
\end{equation}

The homogeneous constraints characterize the macroelement space; to select its nodal functions, fix a local ordering $A(T)=\{a_1,\ldots,a_m\}$ of the degree-of-freedom indices, where $m:=|A(T)|$.
Let $w^\ast_{\bc}$ be the piecewise polynomial function with coefficient vector $\bc\in\RR^N$ in the expansion \eqref{eq:polynomial-expansion}, ordered as in \eqref{eq:potential-coefficient-ordering}.
Let $\mathsf D_T\in\RR^{m\times N}$ map this coefficient vector to the ordered normalized degree-of-freedom values, that is:
\begin{equation}\label{eq:potential-nodal-matrix}
(\mathsf D_T\bc)_q
:=\lambda_{a_q,T}^\ast(w^\ast_{\bc}),
\qquad 1\leq q\leq m.
\end{equation}
For an arbitrary coefficient vector $\bc\in\RR^N$, each evaluation $\lambda_{a_q,T}^\ast(w^\ast_{\bc})$ in \eqref{eq:potential-nodal-matrix} is taken from one fixed incident subsimplex.
If $\bc$ satisfies \eqref{eq:potential-macroelement-constraints}, then $w^\ast_{\bc}\in S(T^\ast)$, where the degrees of freedom $\lambda_{b,T}^\ast$ are single-valued.
For such vectors, the choice of incident subsimplex therefore does not affect $\mathsf D_T\bc$.
The matrix $\mathsf D_T$ has one row per local degree of freedom and one column per polynomial coefficient.
The coefficient vector $\bc_{a,T}$ of $\theta^\ast_{a,T}$ solves the consistent system
\begin{equation}\label{eq:potential-nodal-system}
\underbrace{\begin{pmatrix}
\mathsf C_{\mathrm{cont},T}\\
\mathsf C_{\mathrm{extra},T}\\
\mathsf D_T
\end{pmatrix}}_{\mathsf A_T}
\bc_{a,T}
=\underbrace{\begin{pmatrix}\mathbf 0\\\mathbf 0\\\be_a\end{pmatrix}}_{\bb_a},
\qquad |\bb_a|=1,
\end{equation}
where $\be_a\in\RR^m$ has entry $1$ in the unique position $q$ such that $a_q=a$ and zero entries elsewhere.
In \eqref{eq:potential-nodal-system}, the blocks $\mathsf C_{\mathrm{cont},T}$ and $\mathsf C_{\mathrm{extra},T}$ impose \eqref{eq:potential-macroelement-constraints}.
The block $\mathsf D_T$ defined in \eqref{eq:potential-nodal-matrix} prescribes value $1$ for $\lambda_{a,T}^\ast$ and $0$ for $\lambda_{b,T}^\ast$, $b\in A(T)\setminus\{a\}$.
The three blocks of $\mathsf A_T$ are constructed on the normalized element $T^\ast$.
The entries of $\mathsf C_{\mathrm{cont},T}$ and $\mathsf D_T$ depend continuously on $\bB_T$, since they involve only polynomial operations, $\bB_T^{-1}$, barycenters, and continuously oriented unit normals.
The entries of $\mathsf C_{\mathrm{extra},T}$ are continuous for the HCT and full Walkington spaces by their polynomial trace and barycenter constraints.
For a reduced Walkington space, this continuity is one of the requirements imposed in \cref{sec:magnetic-potential}.

\smallskip
\noindent\textbf{(iii) Uniform bounds and continuous dependence.}
For every matrix $\bB$ in the compact set $\mathfrak B$ defined in \eqref{eq:potential-compact-shapes}, the nondegenerate normalized simplex $K_{\bB}:=\bB\widehat T$ admits the construction \eqref{eq:potential-nodal-system}.
The resulting coefficient matrix $\mathsf A(\bB)$ defines a continuous matrix-valued map $\mathsf A$ on $\mathfrak B$ such that $\mathsf A_T=\mathsf A(\bB_T)$.
For every $\bB \in \mathfrak B$, the local degrees of freedom are unisolvent on $S(K_{\bB})$ since $K_{\bB}$ is nondegenerate.
Consequently,
\begin{equation*}
\ker\mathsf A(\bB)=\{\mathbf 0\},
\qquad
\sigma_{\min}(\mathsf A(\bB))
:=\min_{|\bc|=1}|\mathsf A(\bB)\bc|>0,
\qquad \bB\in\mathfrak B.
\end{equation*}
Thus $\mathsf A(\bB)$ has full column rank, even when redundant constraints are retained.
Continuity of the smallest singular value \cite[Theorem~1]{Stewart1990} and compactness of $\mathfrak B$, together with \eqref{eq:potential-nodal-system}, yield
\begin{equation}\label{eq:potential-coefficient-bound}
\sigma_0:=\min_{\bB\in\mathfrak B}\sigma_{\min}(\mathsf A(\bB))>0,
\qquad
|\bc_{a,T}|
\leq\sigma_0^{-1}|\mathsf A_T\bc_{a,T}|
=\sigma_0^{-1}.
\end{equation}
Moreover, since $\mathsf A_T$ has full column rank, \eqref{eq:potential-nodal-system} gives
\begin{equation*}
\bc_{a,T}
=\mathsf A_T^\dagger\bb_a
=\bigl(\mathsf A_T^\top\mathsf A_T\bigr)^{-1}
\mathsf A_T^\top\bb_a,
\end{equation*}
where $\mathsf A_T^\dagger$ denotes the Moore--Penrose inverse.
The map $\mathsf A$ has constant rank $N$ on $\mathfrak B$, so continuity of the Moore--Penrose inverse \cite[Theorem~4.2]{Rakocevic1997} shows that $\bc_{a,T}$ depends continuously on $\bB_T$, and hence on shape.

For Walkington's element, this dependence is also established explicitly in \cite[Section 3.3]{Walkington2014}, and the same compactness argument is used in the proof of \cite[Theorem 2.5]{Walkington2014}.
Combining \eqref{eq:potential-coefficient-bound} with finite-dimensional polynomial norm equivalence on the fixed reference subdivision $\{\widehat K_i\}_{i=0}^d$, \eqref{eq:potential-normalized-affine-bounds}, and the affine transformation estimates of \cite[Theorem 15.1]{CiarletHandbookError1991}, valid for $1\leq p\leq\infty$, gives
\begin{equation}\label{eq:potential-normalized-nodal-bound}
\|\nabla^j\theta^\ast_{a,T}\|_{0,p;T^\ast}\leq C,
\qquad a\in A(T),\quad 0\leq j\leq2,\quad 1\leq p\leq\infty.
\end{equation}

\smallskip
\noindent\textbf{(iv) Dual functions.}
The dual functions are obtained from the nodal functions through the inverse mass matrix.
Define the normalized mass matrix $M^\ast_T$ by
\begin{equation}\label{eq:potential-normalized-mass-matrix}
M^\ast_T
:=\left(\int_{T^\ast}\theta^\ast_{b,T}\theta^\ast_{c,T}\right)_{b,c\in A(T)}.
\end{equation}
Its entries depend continuously on $\bB_T$, as follows by changing variables to $\widehat T$ and using the continuous dependence of $\bc_{a,T}$ on $\bB_T$, namely
\begin{equation*}
(M^\ast_T)_{bc}
= |\det \bB_T| \int_{\widehat T} (\theta^\ast_{b,T} \circ \bG_T)(\theta^\ast_{c,T} \circ \bG_T)
= |\det \bB_T| \sum_{i=0}^d \sum_{j,k=1}^n c_{b,i,j,T} c_{c,i,k,T} \int_{\widehat K_i} \widehat p_{i,j} \widehat p_{i,k}.
\end{equation*}
Moreover, $M^\ast_T$ is positive definite for every admissible shape because $\{\theta^\ast_{a,T}\}_{a\in A(T)}$ is linearly independent.
Compactness of $\mathfrak B$ therefore gives a positive lower bound, uniform in $h$ and $T$, for the smallest eigenvalue of $M^\ast_T$, and hence a uniform bound for $(M^\ast_T)^{-1}$.
Recalling the definition \eqref{eq:local-dual-def} of the dual functions, the normalized duality identity \eqref{eq:potential-normalized-duality} and the mass matrix $M^\ast_{T_a}$ defined by \eqref{eq:potential-normalized-mass-matrix} yield
\begin{equation}\label{eq:potential-normalized-dual-expansion}
\psi^\ast_a
=\sum_{b\in A(T_a)}\bigl((M^\ast_{T_a})^{-1}\bigr)_{ba}\theta^\ast_{b,T_a}.
\end{equation}
The representation \eqref{eq:potential-normalized-dual-expansion}, the uniform bound for $(M^\ast_{T_a})^{-1}$, the fixed number $|A(T_a)|$ of local degrees of freedom, and \eqref{eq:potential-normalized-nodal-bound} imply
\begin{equation}\label{eq:potential-normalized-dual-bound}
\|\psi^\ast_a\|_{0,p;T_a^\ast}\leq C,
\qquad 1\leq p\leq\infty.
\end{equation}

\subsection{Rescaling and local stability}\label{app:rescaling}
Returning to the physical elements, we use \eqref{eq:potential-normalized-nodal-bound} and \eqref{eq:potential-normalized-dual-bound} to prove \eqref{eq:potential-local-stability}.
The powers associated with the derivative order $\mu_a$ cancel between each nodal function and its averaged coefficient.

Let $p,p'\in[1,\infty]$ satisfy $1/p+1/p'=1$, so that $p'=\infty$ when $p=1$ and $p'=1$ when $p=\infty$.
We use the convention $1/\infty:=0$.
Rearranging the definitions \eqref{eq:potential-normalized-functions} of $\theta^\ast_{a,T}$ and $\psi^\ast_a$, and applying the chain rule with $\nabla\bF_T=h_T\bI$ gives
\begin{equation}\label{eq:potential-physical-function-scaling}
\begin{aligned}
\psi_a\circ\bF_{T_a}
&=h_{T_a}^{-d-\mu_a}\psi^\ast_a,\\
(\nabla^j\theta_{a,T})\circ\bF_T
&=h_T^{\mu_a-j}\nabla^j\theta^\ast_{a,T},
\qquad 0\leq j\leq2.
\end{aligned}
\end{equation}
Since $|\det\nabla\bF_T|=h_T^d$, the change of variables $\bx=\bF_T(\bxi)$ yields, for $w\in\LL^q(T^\ast)$,
\begin{equation}\label{eq:potential-dilation-norm}
\|w\circ\bF_T^{-1}\|_{0,q;T}
=h_T^{d/q}\|w\|_{0,q;T^\ast},
\qquad 1\leq q\leq\infty.
\end{equation}
%
The identity naturally extends to vector- and tensor-valued functions.
Applying \eqref{eq:potential-dilation-norm} to \eqref{eq:potential-physical-function-scaling}, using $-d+d/p'=-d/p$ and the uniform bounds \eqref{eq:potential-normalized-nodal-bound} and \eqref{eq:potential-normalized-dual-bound}, gives
\begin{equation}\label{eq:potential-physical-basis-bounds}
\begin{aligned}
\|\psi_a\|_{0,p';T_a}
&=h_{T_a}^{-d-\mu_a+d/p'}\|\psi^\ast_a\|_{0,p';T_a^\ast}
\leq Ch_{T_a}^{-\mu_a-d/p},\\
\|\nabla^j\theta_{a,T}\|_{0,p;T}
&=h_T^{\mu_a-j+d/p}\|\nabla^j\theta^\ast_{a,T}\|_{0,p;T^\ast}
\leq Ch_T^{\mu_a-j+d/p},
\qquad 0\leq j\leq2.
\end{aligned}
\end{equation}
For $v\in\LL^p(T_a)$, the identity \eqref{eq:potential-averaged-dofs}, H\"older's inequality with exponents $p,p'$, and the first bound in \eqref{eq:potential-physical-basis-bounds} yield
\begin{equation}\label{eq:potential-averaged-coefficient-bound}
|\lambda_a(Q_{T_a}v)|
= \left|\int_{T_a}v\psi_a\right|
\leq\|v\|_{0,p;T_a}\|\psi_a\|_{0,p';T_a}
\leq Ch_{T_a}^{-\mu_a-d/p}\|v\|_{0,p;T_a}.
\end{equation}
For every $a\in A(T)$, we have $T_a\subset\omega_T$ and $h_{T_a}\simeq h_T$.
Inserting \eqref{eq:potential-physical-basis-bounds} and \eqref{eq:potential-averaged-coefficient-bound} into the local expansion \eqref{eq:potential-local-expansion} of $(J_hv)|_T$, and using the uniform bound on $m=|A(T)|$, proves the local stability estimate \eqref{eq:potential-local-stability}.

\subsection{Approximation and Sobolev stability}\label{app:approximation}
Subtracting one polynomial on $\omega_T$ accounts for all incident elements $T_a$ in \eqref{eq:potential-local-expansion} and allows us to use polynomial reproduction.

Let $p\in[1,\infty]$ and let $s$ be an integer with $2\leq s\leq r+1$.
The results about polynomial approximation and Poincaré inequalities on patches found in \cite[Lemmas 5.6 and 5.7]{ErnGuermond2017} provide, for every $v\in\WW^{s,p}(\omega_T)$, a single polynomial $q\in\sP_{s-1}(\RR^d)$ such that
\begin{equation}\label{eq:potential-patch-polynomial}
\sum_{j=0}^{2}h_T^j\|\nabla^j(v-q)\|_{0,p;\omega_T}
\leq Ch_T^s|v|_{\WW^{s,p}(\omega_T)};
\end{equation}
the constant $C>0$ is independent of $h$ and $T$ due to shape-regularity (and local quasi-uniformity).
Since $s-1\leq r$, every local projection $Q_{T_a}$ occurring in \eqref{eq:potential-local-expansion} reproduces $q$, and hence $(J_hq)|_T=q|_T$.
This reproduction property and estimates \eqref{eq:potential-local-stability} and \eqref{eq:potential-patch-polynomial} yield
\begin{equation}\label{eq:potential-local-approximation}
\begin{aligned}
\|\nabla^j(v-J_hv)\|_{0,p;T}
&\leq\|\nabla^j(v-q)\|_{0,p;T}
+Ch_T^{-j}\|v-q\|_{0,p;\omega_T}\\
&\leq Ch_T^{s-j}|v|_{\WW^{s,p}(\omega_T)},
\qquad 0\leq j\leq2,\quad 2\leq s\leq r+1.
\end{aligned}
\end{equation}
%
For $1\leq p<\infty$, summing the $p$th powers of \eqref{eq:potential-local-approximation} over $T\in\cT_h$, using the bounded overlap of the patches, and taking the $p$th root gives the global bounds for $\nabla^j(v-J_hv)$.
For $p=\infty$, these bounds follow by taking the supremum over $T\in\cT_h$.
Combining the bounds for the derivative orders up to $j$ and using $h\leq1$ proves \eqref{eq:potential-averaged-approximation}.
Taking $s=2$ and $j=2$ in \eqref{eq:potential-averaged-approximation}, and using the triangle inequality, gives
\begin{equation*}
\|J_hv\|_{2,p}\leq C\|v\|_{2,p},
\qquad 1\leq p\leq\infty.
\end{equation*}
This proves \eqref{eq:potential-averaged-stability} and completes the proof.

\bibliographystyle{abbrv}
\bibliography{references}

@article {DipernaLions1989,
    AUTHOR = {DiPerna, R. J. and Lions, P.-L.},
     TITLE = {Ordinary differential equations, transport theory and
              {S}obolev spaces},
   JOURNAL = {Invent. Math.},
  FJOURNAL = {Inventiones Mathematicae},
    VOLUME = {98},
      YEAR = {1989},
    NUMBER = {3},
     PAGES = {511--547},
      ISSN = {0020-9910,1432-1297},
   MRCLASS = {34A10 (34D20 35Q20 58D25 82A70)},
  MRNUMBER = {1022305},
MRREVIEWER = {B.\ G.\ Pachpatte},
       DOI = {10.1007/BF01393835},
       URL = {https://doi.org/10.1007/BF01393835},
}

@article {AmiratHamdache2009,
    AUTHOR = {Amirat, Youcef and Hamdache, Kamel},
     TITLE = {Strong solutions to the equations of a ferrofluid flow model},
   JOURNAL = {J. Math. Anal. Appl.},
  FJOURNAL = {Journal of Mathematical Analysis and Applications},
    VOLUME = {353},
      YEAR = {2009},
    NUMBER = {1},
     PAGES = {271--294},
      ISSN = {0022-247X,1096-0813},
   MRCLASS = {35Q35 (35D35 76D05)},
  MRNUMBER = {2508866},
       DOI = {10.1016/j.jmaa.2008.11.084},
}

@article {AmiratHamdache2010,
    AUTHOR = {Amirat, Youcef and Hamdache, Kamel},
     TITLE = {Unique solvability of equations of motion for ferrofluids},
   JOURNAL = {Nonlinear Anal.},
  FJOURNAL = {Nonlinear Analysis. Theory, Methods \& Applications. An
              International Multidisciplinary Journal},
    VOLUME = {73},
      YEAR = {2010},
    NUMBER = {2},
     PAGES = {471--494},
      ISSN = {0362-546X,1873-5215},
   MRCLASS = {35Q35 (76D05)},
  MRNUMBER = {2650831},
       DOI = {10.1016/j.na.2010.03.042},
}

@article {AmiratHamdache2008,
    AUTHOR = {Amirat, Youcef and Hamdache, Kamel},
     TITLE = {Global weak solutions to a ferrofluid flow model},
   JOURNAL = {Math. Methods Appl. Sci.},
  FJOURNAL = {Mathematical Methods in the Applied Sciences},
    VOLUME = {31},
      YEAR = {2008},
    NUMBER = {2},
     PAGES = {123--151},
      ISSN = {0170-4214,1099-1476},
   MRCLASS = {76W05},
  MRNUMBER = {2375760},
       DOI = {10.1002/mma.896},
}

@incollection{Shliomis2002,
author="Shliomis, Mark I.",
editor="Odenbach, Stefan",
title="Ferrohydrodynamics: Retrospective and Issues",
bookTitle="Ferrofluids: Magnetically Controllable Fluids and Their Applications",
year="2002",
publisher="Springer Berlin Heidelberg",
address="Berlin, Heidelberg",
pages="85--111",
isbn="978-3-540-45646-9",
doi="10.1007/3-540-45646-5_5",
url="https://doi.org/10.1007/3-540-45646-5_5"
}

@article{Shliomis1972,
  author   = {Shliomis, M. I.},
  title    = {Effective Viscosity of Magnetic Suspensions},
  journal  = {Journal of Experimental and Theoretical Physics},
  volume   = {34},
  number   = {6},
  pages    = {1291--1294},
  year     = {1972}
}

@article{SchererFiguereido2005,
  title = {Ferrofluids: properties and applications},
  volume = {35},
  issn = {0103-9733},
  shorttitle = {Ferrofluids},
  doi = {10.1590/S0103-97332005000400018},
  language = {en},
  number = {3a},
  journal = {Brazilian Journal of Physics},
  author = {Scherer, C. and Figueiredo Neto, A. M.},
  month = sep,
  year = {2005},
  pages = {718--727},
}

@article{SocoliucEtAl2022,
  title = {Ferrofluids and bio-ferrofluids: looking back and stepping forward},
  volume = {14},
  issn = {2040-3364, 2040-3372},
  shorttitle = {Ferrofluids and bio-ferrofluids},
  url = {https://xlink.rsc.org/?DOI=D1NR05841J},
  doi = {10.1039/D1NR05841J},
  language = {en},
  number = {13},
  urldate = {2026-05-27},
  journal = {Nanoscale},
  author = {Socoliuc, V. and Avdeev, M. V. and Kuncser, V. and Turcu, Rodica and Tombácz, Etelka and Vékás, L.},
  year = {2022},
  pages = {4786--4886},
}

@article{ferrofluids-review,
title = {Engineering applications of ferrofluids: A review},
fjournal = {Journal of Magnetism and Magnetic Materials},
journal = {J. Magn. Magn. Mater.},
volume = {537},
pages = {168222},
year = {2021},
issn = {0304-8853},
doi = {https://doi.org/10.1016/j.jmmm.2021.168222},
url = {https://www.sciencedirect.com/science/article/pii/S0304885321004984},
author = {Madhusree Kole and Sameer Khandekar},
}

@book{Rosensweig2013,
  title={Ferrohydrodynamics},
  author={Rosensweig, R.E.},
  isbn={9780486783000},
  series={Dover Books on Physics},
  year={2013},
  publisher={Dover Publications}
}

@ARTICLE{Eickmann2026,
       author = {{E}ickmann, Franziska and {G}uzm{\'a}n, Johnny and {N}eilan, Michael and Scott, L. Ridgway and {T}scherpel, Tabea},
        title = "{A local Fortin projection for the Scott-Vogelius elements on general meshes}",
      journal = {arXiv e-prints},
         year = 2025,
        month = dec,
          eid = {arXiv:2512.18033v2},
        pages = {arXiv:2512.18033v2},
          doi = {10.48550/arXiv.2512.18033},
archivePrefix = {arXiv},
       eprint = {2512.18033v2},
 primaryClass = {cs.NA},
}

@Misc{GuzmanNeilan2026private,
  author       = {Johnny Guzmán and Michael Neilan},
  howpublished = {{Private Communication}},
  year         = {2026},
  month = {June},
}

@article{irksome,
author = {Farrell, Patrick E. and Kirby, Robert C. and Marchena-Men{\'e}ndez, Jorge},
title = {Irksome: {A}utomating {R}unge–{K}utta {T}ime-stepping for {F}inite {E}lement {M}ethods},
year = {2021},
publisher = {Association for Computing Machinery},
address = {New York, NY, USA},
volume = {47},
number = {4},
issn = {0098-3500},
url = {https://doi.org/10.1145/3466168},
doi = {10.1145/3466168},
journal = {ACM Trans. Math. Softw.},
note = {Article 30},
pages = {1--26},
}

@manual{firedrake,
  title        = {Firedrake User Manual},
  author       = {David A. Ham and Paul H. J. Kelly and Lawrence Mitchell and Colin J. Cotter and Robert C. Kirby and Koki Sagiyama and Nacime Bouziani and Sophia Vorderwuelbecke and Thomas J. Gregory and Jack Betteridge and Daniel R. Shapero and Reuben W. Nixon-Hill and Connor J. Ward and Patrick E. Farrell and Pablo D. Brubeck and India Marsden and Thomas H. Gibson and Miklós Homolya and Tianjiao Sun and Andrew T. T. McRae and Fabio Luporini and Alastair Gregory and Michael Lange and Simon W. Funke and Florian Rathgeber and Gheorghe-Teodor Bercea and Graham R. Markall},
  organization = {Imperial College London and University of Oxford and Baylor University and University of Washington},
  edition      = {First},
  year         = {2023},
  month        = may,
  doi          = {10.25561/104839},
}

@book {GiraultRaviart1986,
    AUTHOR = {Girault, Vivette and Raviart, Pierre-Arnaud},
     TITLE = {Finite element methods for {N}avier-{S}tokes equations},
    SERIES = {Springer Series in Computational Mathematics},
    VOLUME = {5},
      NOTE = {Theory and algorithms},
 PUBLISHER = {Springer-Verlag, Berlin},
      YEAR = {1986},
     PAGES = {x+374},
      ISBN = {3-540-15796-4},
   MRCLASS = {65N30 (65-02 76-08)},
  MRNUMBER = {851383},
MRREVIEWER = {Max\ D.\ Gunzburger},
       DOI = {10.1007/978-3-642-61623-5},
}

@article {NochettoTrivisaWeber2019,
    AUTHOR = {Nochetto, Ricardo H. and Trivisa, Konstantina and Weber,
              Franziska},
     TITLE = {On the dynamics of ferrofluids: global weak solutions to the
              {R}osensweig system and rigorous convergence to equilibrium},
   JOURNAL = {SIAM J. Math. Anal.},
  FJOURNAL = {SIAM Journal on Mathematical Analysis},
    VOLUME = {51},
      YEAR = {2019},
    NUMBER = {6},
     PAGES = {4245--4286},
      ISSN = {0036-1410,1095-7154},
   MRCLASS = {76N10 (35D30 35Q35 76W05)},
  MRNUMBER = {4027844},
MRREVIEWER = {Yanjin\ Wang},
       DOI = {10.1137/18M1224957},
       URL = {https://doi.org/10.1137/18M1224957},
}

@article {AmiratHamdacheMurat2008,
    AUTHOR = {Amirat, Youcef and Hamdache, Kamel and Murat, Fran{\c c}ois},
     TITLE = {Global weak solutions to equations of motion for magnetic
              fluids},
   JOURNAL = {J. Math. Fluid Mech.},
  FJOURNAL = {Journal of Mathematical Fluid Mechanics},
    VOLUME = {10},
      YEAR = {2008},
    NUMBER = {3},
     PAGES = {326--351},
      ISSN = {1422-6928,1422-6952},
   MRCLASS = {35Q35 (35A01 35A35 35D30 76D03 76W05)},
  MRNUMBER = {2430804},
MRREVIEWER = {Ezzedine\ Zahrouni},
       DOI = {10.1007/s00021-006-0234-6},
       URL = {https://doi.org/10.1007/s00021-006-0234-6},
}

@article{QinZhang1992,
  title = {Symplectic {Runge}-{Kutta} {Algorithms} for {Hamiltonian} {Systems}},
  volume = {10},
  issn = {0254-9409},
  url = {https://www.jstor.org/stable/45340281},
  urldate = {2026-10-04},
  journal = {Journal of Computational Mathematics},
  publisher = {Institute of Computational Mathematics and Scientific/Engineering Computing},
  author = {Qin, Meng-zhao and Zhang, Mei-qing},
  year = {1992},
  pages = {205--215},
}

@incollection {CiarletHandbookError1991,
    AUTHOR = {Ciarlet, P. G.},
     TITLE = {Basic error estimates for elliptic problems},
 BOOKTITLE = {Handbook of numerical analysis, {V}ol.\ {II}},
    SERIES = {Handb. Numer. Anal.},
    VOLUME = {II},
     PAGES = {17--351},
 PUBLISHER = {North-Holland, Amsterdam},
      YEAR = {1991},
      ISBN = {0-444-70365-9},
   MRCLASS = {65-02 (65N15 65N30)},
  MRNUMBER = {1115237},
}

@article {Ciarlet1978,
    AUTHOR = {Ciarlet, Philippe G.},
     TITLE = {Interpolation error estimates for the reduced
              {H}sieh-{C}lough-{T}ocher triangle},
   JOURNAL = {Math. Comp.},
  FJOURNAL = {Mathematics of Computation},
    VOLUME = {32},
      YEAR = {1978},
    NUMBER = {142},
     PAGES = {335--344},
      ISSN = {0025-5718,1088-6842},
   MRCLASS = {65N30 (41A25)},
  MRNUMBER = {482249},
       DOI = {10.2307/2006147},
}

@article{GiraultScott2002,
    AUTHOR = {Girault, Vivette and Scott, L. Ridgway},
     TITLE = {Hermite interpolation of nonsmooth functions preserving boundary conditions},
   JOURNAL = {Math. Comp.},
  FJOURNAL = {Mathematics of Computation},
    VOLUME = {71},
    NUMBER = {239},
      YEAR = {2002},
     PAGES = {1043--1074},
       DOI = {10.1090/S0025-5718-02-01446-1},
       URL = {https://doi.org/10.1090/S0025-5718-02-01446-1}
}

@article{DuanDuLiuZhang2019,
    AUTHOR = {Duan, Huoyuan and Du, Zhijie and Liu, Wei and Zhang, Shangyou},
     TITLE = {New Mixed Elements for {Maxwell} Equations},
   JOURNAL = {SIAM J. Numer. Anal.},
  FJOURNAL = {SIAM Journal on Numerical Analysis},
    VOLUME = {57},
    NUMBER = {1},
      YEAR = {2019},
     PAGES = {320--354},
       DOI = {10.1137/18M1168054},
       URL = {https://doi.org/10.1137/18M1168054}
}

@article {Walkington2014,
    AUTHOR = {Walkington, Noel J.},
     TITLE = {A {$C^1$} tetrahedral finite element without edge degrees of
              freedom},
   JOURNAL = {SIAM J. Numer. Anal.},
  FJOURNAL = {SIAM Journal on Numerical Analysis},
    VOLUME = {52},
      YEAR = {2014},
    NUMBER = {1},
     PAGES = {330--342},
      ISSN = {0036-1429,1095-7170},
   MRCLASS = {65N30},
  MRNUMBER = {3163246},
MRREVIEWER = {Gerard\ Awanou},
       DOI = {10.1137/130912013},
}

@article {ScottVogelius1985,
    AUTHOR = {Scott, L. R. and Vogelius, M.},
     TITLE = {Norm estimates for a maximal right inverse of the divergence
              operator in spaces of piecewise polynomials},
   JOURNAL = {RAIRO Mod\'el. Math. Anal. Num\'er.},
  FJOURNAL = {RAIRO Mod\'elisation Math\'ematique et Analyse Num\'erique},
    VOLUME = {19},
      YEAR = {1985},
    NUMBER = {1},
     PAGES = {111--143},
      ISSN = {0764-583X,1290-3841},
   MRCLASS = {65N30},
  MRNUMBER = {813691},
       DOI = {10.1051/m2an/1985190101111},
       URL = {https://doi.org/10.1051/m2an/1985190101111},
}

@article {AlfeldSorokina2016,
    AUTHOR = {Alfeld, Peter and Sorokina, Tatyana},
     TITLE = {Linear differential operators on bivariate spline spaces and
              spline vector fields},
   JOURNAL = {BIT},
  FJOURNAL = {BIT. Numerical Mathematics},
    VOLUME = {56},
      YEAR = {2016},
    NUMBER = {1},
     PAGES = {15--32},
      ISSN = {0006-3835,1572-9125},
   MRCLASS = {41A15 (41A10 65D05 65D07)},
  MRNUMBER = {3486451},
MRREVIEWER = {Tian-Xiao\ He},
       DOI = {10.1007/s10543-015-0557-x},
}

@article {GuzmanNeilan2018,
    AUTHOR = {Guzm\'an, Johnny and Neilan, Michael},
     TITLE = {Inf-sup stable finite elements on barycentric refinements
              producing divergence-free approximations in arbitrary
              dimensions},
   JOURNAL = {SIAM J. Numer. Anal.},
  FJOURNAL = {SIAM Journal on Numerical Analysis},
    VOLUME = {56},
      YEAR = {2018},
    NUMBER = {5},
     PAGES = {2826--2844},
      ISSN = {0036-1429,1095-7170},
   MRCLASS = {65N30 (65N12)},
  MRNUMBER = {3853609},
MRREVIEWER = {Marius\ Ghergu},
       DOI = {10.1137/17M1153467},
}

@inproceedings{ArnoldQin1992,
  author    = {Douglas N. Arnold and Jinshui Qin},
  title     = {Quadratic velocity/linear pressure {S}tokes elements},
  booktitle = {Proceedings of Advances in Computer Methods for Partial Differential Equations VII},
  editor    = {R. Vichnevetsky and R. S. Steplemen},
  publisher = {AICA},
  year      = {1992},
  pages     = {28--34},
  url       = {https://www.umn.edu/~arnold/papers/imacs92.pdf}
}

@article {GuzmanNeilan2014-triangle,
    AUTHOR = {Guzm\'an, Johnny and Neilan, Michael},
     TITLE = {Conforming and divergence-free {S}tokes elements on general
              triangular meshes},
   JOURNAL = {Math. Comp.},
  FJOURNAL = {Mathematics of Computation},
    VOLUME = {83},
      YEAR = {2014},
    NUMBER = {285},
     PAGES = {15--36},
      ISSN = {0025-5718,1088-6842},
   MRCLASS = {65N30 (65N12 76D07 76M10)},
  MRNUMBER = {3120580},
MRREVIEWER = {Marius\ Ghergu},
       DOI = {10.1090/S0025-5718-2013-02753-6},
       URL = {https://doi.org/10.1090/S0025-5718-2013-02753-6},
}

@article {GuzmanNeilan2014-tetrahedra,
    AUTHOR = {Guzm\'an, Johnny and Neilan, Michael},
     TITLE = {Conforming and divergence-free {S}tokes elements in three
              dimensions},
   JOURNAL = {IMA J. Numer. Anal.},
  FJOURNAL = {IMA Journal of Numerical Analysis},
    VOLUME = {34},
      YEAR = {2014},
    NUMBER = {4},
     PAGES = {1489--1508},
      ISSN = {0272-4979,1464-3642},
   MRCLASS = {65N30 (65N12 76D07)},
  MRNUMBER = {3269433},
MRREVIEWER = {David\ Maltese},
       DOI = {10.1093/imanum/drt053},
       URL = {https://doi.org/10.1093/imanum/drt053},
}

@ARTICLE{BrubeckLianParker2026,
       author = {Brubeck, Pablo D. and Liang, Yizhou and Parker, Charles},
        title = "{Cohomology of Finite Element Stokes Complexes on Alfeld Splits}",
      journal = {arXiv e-prints},
         year = 2026,
          eid = {arXiv:2605.31348},
        pages = {arXiv:2605.31348},
          doi = {10.48550/arXiv.2605.31348},
archivePrefix = {arXiv},
       eprint = {2605.31348},
 primaryClass = {cs.NA},
}

@article {ScottZhang1990,
    AUTHOR = {Scott, L. Ridgway and Zhang, Shangyou},
     TITLE = {Finite element interpolation of nonsmooth functions satisfying
              boundary conditions},
   JOURNAL = {Math. Comp.},
  FJOURNAL = {Mathematics of Computation},
    VOLUME = {54},
      YEAR = {1990},
    NUMBER = {190},
     PAGES = {483--493},
      ISSN = {0025-5718,1088-6842},
   MRCLASS = {65D05 (65N30)},
  MRNUMBER = {1011446},
MRREVIEWER = {Qian\ Li},
       DOI = {10.2307/2008497},
}

@misc{BrubeckGithub2026,
  author       = {Pablo D. Brubeck},
  title        = {Add reduced {Hermite}, reduced {Walkington}, and rotated {Bernardi-Raugel} (Pull Request \#301)},
  year         = {2026},
  publisher    = {GitHub},
  journal      = {GitHub repository},
  howpublished = {\url{https://github.com/firedrakeproject/fiat/pull/301}},
  note         = {firedrakeproject/fiat}
}

@article{ErnGuermond2017,
    AUTHOR = {Ern, Alexandre and Guermond, Jean-Luc},
     TITLE = {Finite element quasi-interpolation and best approximation},
   JOURNAL = {ESAIM Math. Model. Numer. Anal.},
  FJOURNAL = {ESAIM: Mathematical Modelling and Numerical Analysis},
    VOLUME = {51},
      YEAR = {2017},
    NUMBER = {4},
     PAGES = {1367--1385},
  MRNUMBER = {3702417},
       DOI = {10.1051/m2an/2016066},
       URL = {https://doi.org/10.1051/m2an/2016066},
}

@book{Teschl2012,
  author    = {Teschl, Gerald},
  title     = {Ordinary Differential Equations and Dynamical Systems},
  series    = {Graduate Studies in Mathematics},
  volume    = {140},
  publisher = {American Mathematical Society},
  address   = {Providence, RI},
  year      = {2012},
  isbn      = {978-0-8218-8328-0},
  url       = {https://www.mat.univie.ac.at/~gerald/ftp/book-ode/ode.pdf}
}

@misc{Prazak2024,
  author       = {Pra{\v{z}}{\'{a}}k, Dalibor},
  title        = {{Carath\'eodory} theory of {ODEs}},
  year         = {2024},
  howpublished = {\url{https://www.karlin.mff.cuni.cz/~prazak/vyuka/Odr2/Skripta/en_acODR-24.pdf}},
  note         = {Lecture notes, Fall 2024}
}

@techreport{Stewart1990,
  author      = {Stewart, G. W.},
  title       = {Perturbation Theory for the Singular Value Decomposition},
  institution = {University of Maryland},
  address     = {College Park, MD},
  type        = {Technical Report},
  number      = {UMIACS-TR-90-124},
  year        = {1990},
  month       = sep,
  url         = {https://users.math.msu.edu/users/iwenmark/Teaching/MTH995/Papers/SVD_Stewart.pdf}
}

@article{Rakocevic1997,
  author  = {Rako{\v{c}}evi{\'{c}}, Vladimir},
  title   = {On continuity of the {Moore--Penrose} and {Drazin} inverses},
  journal = {Mat. Vesnik},
  volume  = {49},
  number  = {3--4},
  year    = {1997},
  pages   = {163--172},
  url     = {https://www.vesnik.math.rs/vol/mv973404.pdf}
}

@article {LanthalerWeber2026,
    AUTHOR = {Lanthaler, Samuel and Weber, Franziska},
     TITLE = {Global {E}xistence of {W}eak {S}olutions to a {T}wo-{P}hase
              {D}iffuse {I}nterface {M}odel of {F}errofluids {D}ynamics},
   JOURNAL = {SIAM J. Math. Anal.},
  FJOURNAL = {SIAM Journal on Mathematical Analysis},
    VOLUME = {58},
      YEAR = {2026},
    NUMBER = {5},
     PAGES = {4723--4771},
      ISSN = {0036-1410,1095-7154},
   MRCLASS = {35Q35 (35D30 35M33)},
  MRNUMBER = {5128505},
       DOI = {10.1137/25M1772769},
}

@article {NochettoSalgadoTomas2016,
    AUTHOR = {Nochetto, Ricardo H. and Salgado, Abner J. and Tomas, Ignacio},
     TITLE = {The equations of ferrohydrodynamics: modeling and numerical
              methods},
   JOURNAL = {Math. Models Methods Appl. Sci.},
  FJOURNAL = {Mathematical Models and Methods in Applied Sciences},
    VOLUME = {26},
      YEAR = {2016},
    NUMBER = {13},
     PAGES = {2393--2449},
      ISSN = {0218-2025,1793-6314},
   MRCLASS = {65M60 (35Q35 35Q60 65M12 76T20 76W99)},
  MRNUMBER = {3579306},
MRREVIEWER = {Rafael\ F.\ Santos},
       DOI = {10.1142/S0218202516500573},
       URL = {https://doi.org/10.1142/S0218202516500573},
}

@article {NochettoSalgadoTomas2016TwoPhase,
    AUTHOR = {Nochetto, Ricardo H. and Salgado, Abner J. and Tomas, Ignacio},
     TITLE = {A diffuse interface model for two-phase ferrofluid flows},
   JOURNAL = {Comput. Methods Appl. Mech. Engrg.},
  FJOURNAL = {Computer Methods in Applied Mechanics and Engineering},
    VOLUME = {309},
      YEAR = {2016},
     PAGES = {497--531},
      ISSN = {0045-7825,1879-2138},
   MRCLASS = {65M60 (35Q35 35Q61 65M12 76T99)},
  MRNUMBER = {3543019},
       DOI = {10.1016/j.cma.2016.06.011},
       URL = {https://doi.org/10.1016/j.cma.2016.06.011},
}

@article {WuXie2023,
    AUTHOR = {Wu, Yongke and Xie, Xiaoping},
     TITLE = {Energy-stable mixed finite element methods for a ferrofluid
              flow model},
   JOURNAL = {Commun. Nonlinear Sci. Numer. Simul.},
  FJOURNAL = {Communications in Nonlinear Science and Numerical Simulation},
    VOLUME = {125},
      YEAR = {2023},
     PAGES = {Paper No. 107330, 40},
      ISSN = {1007-5704,1878-7274},
   MRCLASS = {65N55 (65F10 65N22 65N30)},
  MRNUMBER = {4599452},
       DOI = {10.1016/j.cnsns.2023.107330},
       URL = {https://doi.org/10.1016/j.cnsns.2023.107330},
}

@article {WuXie2025,
    AUTHOR = {Wu, Yongke and Xie, Xiaoping},
     TITLE = {Energy-stable mixed finite element methods for the
              {R}osensweig ferrofluid flow model},
   JOURNAL = {Math. Models Methods Appl. Sci.},
  FJOURNAL = {Mathematical Models and Methods in Applied Sciences},
    VOLUME = {35},
      YEAR = {2025},
    NUMBER = {12},
     PAGES = {2513--2559},
      ISSN = {0218-2025,1793-6314},
   MRCLASS = {65M60 (65M12 65M15)},
  MRNUMBER = {4969116},
MRREVIEWER = {Beny\ Neta},
       DOI = {10.1142/S0218202525500459},
       URL = {https://doi.org/10.1142/S0218202525500459},
}

@article {DongHuangHuang2025,
    AUTHOR = {Dong, Xiaojing and Huang, Huayi and Huang, Yunqing},
     TITLE = {A first-order, semi-implicit, and unconditionally
              energy-stable scheme for an incompressible ferrohydrodynamics
              flow},
   JOURNAL = {J. Comput. Math.},
  FJOURNAL = {Journal of Computational Mathematics},
    VOLUME = {43},
      YEAR = {2025},
    NUMBER = {4},
     PAGES = {866--897},
      ISSN = {0254-9409,1991-7139},
   MRCLASS = {65N30 (35Q60 65M60 65N12 76W05)},
  MRNUMBER = {4931852},
MRREVIEWER = {Thi-Thao-Phuong\ Hoang},
       DOI = {10.4208/jcm.2402-m2023-0181},
       URL = {https://doi.org/10.4208/jcm.2402-m2023-0181},
}

@article {DongHuangHuangTang2025,
    AUTHOR = {Dong, Xiaojing and Huang, Huayi and Huang, Yunqing and Tang,
              Qili},
     TITLE = {On the {R}osensweig model: a linear, energy-stable, and
              convergent finite element method},
   JOURNAL = {Sci. China Math.},
  FJOURNAL = {Science China. Mathematics},
    VOLUME = {68},
      YEAR = {2025},
    NUMBER = {11},
     PAGES = {2753--2772},
      ISSN = {1674-7283,1869-1862},
   MRCLASS = {65M60 (65M12 65M15)},
  MRNUMBER = {4970478},
MRREVIEWER = {Wenjie\ Liu},
       DOI = {10.1007/s11425-024-2393-y},
       URL = {https://doi.org/10.1007/s11425-024-2393-y},
}

@article {KeramHuangHe2026,
    AUTHOR = {Keram, Aytura and Huang, Pengzhan and He, Yinnian},
     TITLE = {Stability and convergence of a variable time step scheme for
              {R}osensweig's ferrohydrodynamics equations},
   JOURNAL = {Commun. Nonlinear Sci. Numer. Simul.},
  FJOURNAL = {Communications in Nonlinear Science and Numerical Simulation},
    VOLUME = {152},
      YEAR = {2026},
     PAGES = {Paper No. 109271, 28},
      ISSN = {1007-5704,1878-7274},
   MRCLASS = {65M12 (65M60)},
  MRNUMBER = {4959639},
MRREVIEWER = {Yi\ Qin},
       DOI = {10.1016/j.cnsns.2025.109271},
       URL = {https://doi.org/10.1016/j.cnsns.2025.109271},
}

@article {MaoSun2025,
    AUTHOR = {Mao, Shipeng and Sun, Jiaao},
     TITLE = {Error estimates of finite element method for the
              incompressible ferrohydrodynamics equations},
   JOURNAL = {Commun. Appl. Math. Comput.},
  FJOURNAL = {Communications on Applied Mathematics and Computation},
    VOLUME = {7},
      YEAR = {2025},
    NUMBER = {2},
     PAGES = {485--535},
      ISSN = {2096-6385,2661-8893},
   MRCLASS = {76M10 (35Q35 65M15 65M60 65N12)},
  MRNUMBER = {4865882},
       DOI = {10.1007/s42967-023-00347-w},
       URL = {https://doi.org/10.1007/s42967-023-00347-w},
}

@article {WuXie2023Parallel,
    AUTHOR = {Wu, Yongke and Xie, Xiaoping},
     TITLE = {Mixed finite element methods for the ferrofluid model with
              magnetization paralleled to the magnetic field},
   JOURNAL = {Numer. Math. Theory Methods Appl.},
  FJOURNAL = {Numerical Mathematics. Theory, Methods and Applications},
    VOLUME = {16},
      YEAR = {2023},
    NUMBER = {2},
     PAGES = {489--510},
      ISSN = {1004-8979,2079-7338},
   MRCLASS = {65N30 (65F10 65N22 65N55)},
  MRNUMBER = {4573911},
MRREVIEWER = {Olga\ A.\ Lavrova},
       DOI = {10.4208/nmtma.oa-2022-0163},
       URL = {https://doi.org/10.4208/nmtma.oa-2022-0163},
}

@techreport{BenavidesCaucaoGaticaSobral2026,
  author      = {Benavides, Gonzalo A. and Caucao, Sergio and Gatica, Gabriel N. and Sobral, Yuri D.},
  title       = {New {Banach} spaces-based mixed finite element methods for steady-state flows of magnetic fluids},
  institution = {Centro de Investigaci{\'o}n en Ingenier{\'i}a Matem{\'a}tica, Universidad de Concepci{\'o}n},
  type        = {Preprint},
  number      = {2026-18},
  year        = {2026},
  url         = {https://www.ci2ma.udec.cl/pdf/pre-publicaciones2/2026/pp26-18.pdf}
}

@article {ZhangHeYang2021,
    AUTHOR = {Zhang, Guo-Dong and He, Xiaoming and Yang, Xiaofeng},
     TITLE = {Decoupled, linear, and unconditionally energy stable fully
              discrete finite element numerical scheme for a two-phase
              ferrohydrodynamics model},
   JOURNAL = {SIAM J. Sci. Comput.},
  FJOURNAL = {SIAM Journal on Scientific Computing},
    VOLUME = {43},
      YEAR = {2021},
    NUMBER = {1},
     PAGES = {B167--B193},
      ISSN = {1064-8275,1095-7197},
   MRCLASS = {65M60 (35Q60 65M12 76W05)},
  MRNUMBER = {4202511},
MRREVIEWER = {Beny\ Neta},
       DOI = {10.1137/19M1288280},
       URL = {https://doi.org/10.1137/19M1288280},
}

@article {ZhangHeYang2023,
    AUTHOR = {Zhang, Guo-Dong and He, Xiaoming and Yang, Xiaofeng},
     TITLE = {Reformulated weak formulation and efficient fully discrete
              finite element method for a two-phase ferrohydrodynamics
              {S}hliomis model},
   JOURNAL = {SIAM J. Sci. Comput.},
  FJOURNAL = {SIAM Journal on Scientific Computing},
    VOLUME = {45},
      YEAR = {2023},
    NUMBER = {3},
     PAGES = {B253--B282},
      ISSN = {1064-8275,1095-7197},
   MRCLASS = {65N12 (65M12 65M70)},
  MRNUMBER = {4595845},
       DOI = {10.1137/22M1499376},
       URL = {https://doi.org/10.1137/22M1499376},
}

@article {ZhangPanHeYang2025,
    AUTHOR = {Zhang, Guo-Dong and Pan, Kejia and He, Xiaoming and Yang,
              Xiaofeng},
     TITLE = {Energy-stable and efficient finite element schemes for the
              {S}hliomis model of ferrofluid flows},
   JOURNAL = {Adv. Comput. Math.},
  FJOURNAL = {Advances in Computational Mathematics},
    VOLUME = {51},
      YEAR = {2025},
    NUMBER = {4},
     PAGES = {Paper No. 36, 27},
      ISSN = {1019-7168,1572-9044},
   MRCLASS = {65M60 (65M12)},
  MRNUMBER = {4933245},
       DOI = {10.1007/s10444-025-10249-5},
       URL = {https://doi.org/10.1007/s10444-025-10249-5},
}

@article {ZhangHeYang2025,
    AUTHOR = {Zhang, Guo-Dong and He, Xiaoming and Yang, Xiaofeng},
     TITLE = {Efficient fully discrete finite element scheme for the
              ferrohydrodynamic {R}osensweig model and simulations of
              ferrofluid rotational flow problems},
   JOURNAL = {SIAM J. Sci. Comput.},
  FJOURNAL = {SIAM Journal on Scientific Computing},
    VOLUME = {47},
      YEAR = {2025},
    NUMBER = {1},
     PAGES = {B28--B58},
      ISSN = {1064-8275,1095-7197},
   MRCLASS = {65M70 (65M12 76W05)},
  MRNUMBER = {4846762},
       DOI = {10.1137/24M1640914},
       URL = {https://doi.org/10.1137/24M1640914},
}

@article {DongHuangHuangYang2026,
    AUTHOR = {Dong, Xiaojing and Huang, Huayi and Huang, Yunqing and Yang,
              Xiaofeng},
     TITLE = {A fully discrete finite element method with second-order
              temporal accuracy for the {R}osensweig model},
   JOURNAL = {J. Comput. Phys.},
  FJOURNAL = {Journal of Computational Physics},
    VOLUME = {549},
      YEAR = {2026},
     PAGES = {Paper No. 114617, 19},
      ISSN = {0021-9991,1090-2716},
   MRCLASS = {65M60 (65M12)},
  MRNUMBER = {5009159},
       DOI = {10.1016/j.jcp.2025.114617},
       URL = {https://doi.org/10.1016/j.jcp.2025.114617},
}

@misc{BenavidesNochettoShakipov2026,
  author        = {Benavides, Gonzalo A. and Nochetto, Ricardo H. and Shakipov, Mansur},
  title         = {Ferrofluids: modeling and approximation},
  year          = {2026},
  howpublished  = {arXiv:2606.01392},
  archivePrefix = {arXiv},
  eprint        = {2606.01392},
  primaryClass  = {math.NA},
  doi           = {10.48550/arXiv.2606.01392},
  url           = {https://arxiv.org/abs/2606.01392}
}

@article {DiPietroErn2010,
    AUTHOR = {Di Pietro, Daniele A. and Ern, Alexandre},
     TITLE = {Discrete functional analysis tools for discontinuous
              {G}alerkin methods with application to the incompressible
              {N}avier-{S}tokes equations},
   JOURNAL = {Math. Comp.},
  FJOURNAL = {Mathematics of Computation},
    VOLUME = {79},
      YEAR = {2010},
    NUMBER = {271},
     PAGES = {1303--1330},
      ISSN = {0025-5718,1088-6842},
   MRCLASS = {65N30 (35Q30 46E39 46N20 65N12 76D05)},
  MRNUMBER = {2629994},
MRREVIEWER = {Mario\ Ohlberger},
       DOI = {10.1090/S0025-5718-10-02333-1},
       URL = {https://doi.org/10.1090/S0025-5718-10-02333-1},
}

@article{KirbyMitchell2019,
author = {Kirby, Robert C. and Mitchell, Lawrence},
title = {Code Generation for Generally Mapped Finite Elements},
year = {2019},
publisher = {Association for Computing Machinery},
address = {New York, NY, USA},
volume = {45},
number = {4},
issn = {0098-3500},
url = {https://doi.org/10.1145/3361745},
doi = {10.1145/3361745},
journal = {ACM Trans. Math. Softw.},
articleno = {41},
numpages = {23},
note = {Article 41, 23 pages}
}

\end{document}